\documentclass[reqno]{amsart}
\usepackage{amsmath,amsthm,amssymb,latexsym,fullpage,setspace,graphicx,float,xcolor,hyperref,verbatim,mathtools}
\usepackage[foot]{amsaddr}
\usepackage{tikz-cd,tikz,pgfplots}
\usetikzlibrary{decorations.markings,math,positioning,graphs}
\usepackage[utf8]{inputenc}
\usepackage[english]{babel}

\theoremstyle{plain}
\newtheorem{thm}{Theorem}[section]
\newtheorem{cor}[thm]{Corollary}
\newtheorem{lem}[thm]{Lemma}

\newtheorem{prop}[thm]{Proposition}

\newtheorem{remark}[thm]{Remark}
\newtheorem{eg}[thm]{Example}

\theoremstyle{definition}
\newtheorem{defn}{Definition}[section]

\makeatletter
\@namedef{subjclassname@2020}{%
  \textup{2020} Mathematics Subject Classification}
\makeatother

\title{Ergodic properties of a parameterised family of symmetric golden maps: the matching phenomenon revisited}
\author{Karma Dajani}
\author{Slade Sanderson}
\address{Department of Mathematics, Utrecht University, P.O.~Box 80010, 3508TA Utrecht, The Netherlands}
\email{k.dajani1@uu.nl and s.b.sanderson@uu.nl}
\date{January 20, 2023}

\begin{document}
\maketitle
\tikzset{->-/.style={decoration={markings,mark=at position #1 with {\arrow{>}}},postaction={decorate}}}

\begin{abstract}
We study a one-parameter family of interval maps $\{T_\alpha\}_{\alpha\in[1,\beta]}$, with $\beta$ the golden mean, defined on $[-1,1]$ by $T_\alpha(x)=\beta^{1+|t|}x-t\beta\alpha$ where $t\in\{-1,0,1\}$.  For each $T_\alpha,\ \alpha>1$, we construct its unique, absolutely continuous invariant measure and show that on an open, dense subset of parameters $\alpha$, the corresponding density is a step function with finitely many jumps.  We give an explicit description of the maximal intervals of parameters on which the density has at most the same number of jumps.  A main tool in our analysis is the phenomenon of matching, where the orbits of the left and right limits of discontinuity points meet after a finite number of steps.  Each $T_\alpha$ generates signed expansions of numbers in base $1/\beta$; via Birkhoff's ergodic theorem, the invariant measures are used to determine the asymptotic relative frequencies of digits in generic $T_\alpha$-expansions.  In particular, the frequency of $0$ is shown to vary continuously as a function of $\alpha$ and to attain its maximum $3/4$ on the maximal interval $[1/2+1/\beta,1+1/\beta^2]$.
\end{abstract}

\section{Introduction}\label{Introduction}

Dynamical systems given by piecewise monotone maps $T:I\to I$ of an interval have a rich history: besides having applications in various fields---including population ecology (\cite{botella_oteo_ros_09}) and controlled switching circuits (\cite{banergee_karthik_yaun_yorke_00})---these systems are often used to produce expansions of numbers from the underlying interval $I$.  Examples include decimal, $n$-ary, continued fraction, (generalised) L\"uroth and $\beta$-expansions, though this list is far from exhaustive.  A common theme in the study of these expansions is the investigation of asymptotic relative frequencies of digits occurring in typical (i.e. Lebesgue--almost all) expansions.  To this end, the standard procedure is the construction of an ergodic, $T$-invariant measure $\mu$ equivalent to Lebesgue measure $\lambda$ and a calculation of the $\mu$-measure of the subinterval of $I$ corresponding to the digit(s) in question.  Birkhoff's ergodic theorem asserts that the measure of this subinterval equals the desired asymptotic frequency.

In \cite{dajani_kalle_20}, invariant measures and frequencies of digits are studied for a family of \textit{symmetric doubling maps} $\{D_\eta\}_{\eta\in [1,2]}$ defined on $[-1,1]$ by $D_{\eta}(x)=2x-d(x)\eta$ with $d(x)\in\{-1,0,1\}$.  These maps produce \textit{signed binary expansions} of numbers $x\in[-1,1]$ of the form $x=\eta\sum_{n\ge 1}d_n/2^n$ with each $d_n\in\{-1,0,1\}$.  It is shown that each $D_\eta$, $\eta>1$, admits an ergodic, invariant measure equivalent to Lebesgue measure.  The authors use a curious property called \textit{matching}---defined in the sequel---to prove that there is a countable collection of disjoint, open subintervals of $[1,2]$ whose union has full measure, and such that on each such subinterval, the densities of the corresponding invariant measures are step functions with at most the same, finite number of jumps.  These explicitly constructed measures are then used to study the asymptotic frequency of the digit $0$ in generic expansions.  This frequency is shown to be continuous as a function of $\eta$ and attains a maximal value of $2/3$ on the maximal interval $[6/5,3/2]$.  Moreover, the frequency function is either constant, strictly increasing or strictly decreasing on each of the aforementioned subintervals of $[1,2]$.

The present article continues these themes of inquiry with a parameterised family of \textit{skewed symmetric golden maps} $\{T_\alpha\}_{\alpha\in[1,\beta]}$, with $\beta=(\sqrt{5}+1)/2$ the golden mean, i.e. the positive real solution to $\beta^2=\beta+1$.  Each $T_\alpha:[-1,1]\to[-1,1]$ is defined by
\[T_\alpha(x):=\begin{cases}
\beta^2x+\beta\alpha, & x\in [-1,-1/\beta)\\
\beta x, & x\in[-1/\beta,1/\beta]\\
\beta^2x-\beta\alpha, & x\in(1/\beta,1]
\end{cases};\]
see Figure \ref{T_alpha_pic}.  Setting $J_{-1}:=[-1,-1/\beta),\ J_0:=[-1/\beta,1/\beta]$ and $J_1:=(1/\beta,1]$, the map $T_\alpha$ may be written more succinctly as 
\begin{equation}\label{T_alpha}
T_\alpha(x)=\beta^{1+|t(x)|}x-t(x)\beta\alpha,
\end{equation}
where $t(x)\in\{-1,0,1\}$ is the unique index for which $x\in J_{t(x)}$.  For $j\ge 1$, set $t_{\alpha,j}(x):=t(T_\alpha^{j-1}(x))$; the sequence of digits $(t_{\alpha,j}(x))_{j\ge 1}\in\{-1,0,1\}^\mathbb{N}$ records indices of the subsequent subintervals $J_{-1},\ J_0$ or $J_1$ entered by the forward orbit of $x$.  With this notation, equation (\ref{T_alpha}) gives for each $j\ge 1$
\[T_\alpha^j(x)=\beta^{1+|t_{\alpha,j}(x)|}T_\alpha^{j-1}(x)-t_{\alpha,j}(x)\beta\alpha.\]
Solving this for $T_\alpha^{j-1}(x)$, induction shows that for any $n\ge 1$,
\[x=\alpha\sum_{j=1}^n\frac{t_{\alpha,j}(x)}{\beta^{j-1+\sum_{k=1}^j |t_{\alpha,k}(x)|}}+\frac{T_\alpha^n(x)}{\beta^{n+\sum_{k=1}^n |t_{\alpha,k}(x)|}}.\]
Taking the limit $n\to\infty$ and recalling that $|T_\alpha^n(x)|\le 1$ gives
\[x=\alpha\sum_{j\ge1}\frac{t_{\alpha,j}(x)}{\beta^{j-1+\sum_{k=1}^j |t_{\alpha,k}(x)|}}.\]
Note that for fixed $\alpha$, this process determines a unique expansion for each $x\in[-1,1]$.  We refer to both this expansion and the corresponding sequence of digits $(t_{\alpha,j}(x))_{j\ge 1}$ as the \textit{$T_\alpha$-expansion} of $x$.  

Phenomena analogous to those observed in \cite{dajani_kalle_20} are found to occur for the skewed symmetric binary maps $T_\alpha$.  In particular, we prove: 
\begin{thm}\label{main_measure_thm}
For each $\alpha\in(1,\beta]$, the map $T_\alpha$ has a unique---hence ergodic---absolutely continuous invariant probability measure $\mu_\alpha$.  Moreover, $\mu_\alpha$ is equivalent to Lebesgue measure $\lambda$, and there is a countable collection $\{I_\mathbf d\}_{\mathbf d\in\mathcal{M}}$ of disjoint open subintervals of $[1,\beta]$ of full Lebesgue measure, such that for fixed $\mathbf d\in\mathcal{M}$ the density of each $\mu_\alpha$ with $\alpha\in I_\mathbf{d}$ is a step function with at most the same, finite number of jumps.  
\end{thm}
Via Birkhoff's ergodic theorem, these measures are employed to show the following: 
\begin{thm}\label{main_freq_thm}
The asymptotic relative frequency of the digit $0$ in Lebesgue-a.e. $T_\alpha$-expansion depends continuously on $\alpha\in[1,\beta]$ and attains a maximum value of $3/4$ on the (maximal) interval $[1/2+1/\beta,1+1/\beta^2]$.  Furthermore, the frequency function is either constant, strictly increasing or strictly decreasing on each $I_\mathbf d$.
\end{thm}

As in \cite{dajani_kalle_20}, the main tool used to construct the $T_\alpha$-invariant measures is a property called matching.  An interval map $T:I\to I$ is said to have \textit{matching} if for each critical point $c\in I$, the orbits of the left and right limits $y_\pm:=\lim_{x\to c^\pm}T(x)$ agree after some finite number of steps.\footnote{Some authors require that the one-sided derivatives also agree at these times, in which case the map may be said to have \textit{strong matching} (\cite{dajani_kalle_maggioni_20}).  This extra condition is not needed for our purposes.}  That is, for each critical point $c\in I$ there are integers $M,N\ge 0$ for which $T^M(y_-)=T^N(y_+)$. 

Matching has gained considerable attention in recent years.  Intricacies of the metric entropy function of Nakada's $\alpha$-continued fraction maps have been studied using matching in \cite{nakada_natsui_08}, \cite{carminati_marmi_profeti_10}, \cite{carminati_tiozzo_12}, \cite{kraaikamp_schmidt_steiner_12}, \cite{bonanno_carminati_isola_tiozzo_13} and \cite{carminati_tiozzo_13}.  In particular, matching is used in \cite{kraaikamp_schmidt_steiner_12} to determine the natural extension for each $\alpha$-continued fraction transformation, and it is shown that the set of $\alpha\in [0,1]$ for which matching does not occur has zero Lebesgue measure.  The Lebesgue measure of this set of non-matching parameters---in addition to the fact that its Hausdorff dimension is $1$---is also shown in \cite{carminati_tiozzo_12}.  Matching is used in \cite{dajani_kraaikamp_steiner_09} to determine invariant measures for the related family of $\alpha$-Rosen continued fraction transformations.  A parameterised family of linear maps with one increasing and one decreasing branch are considered in \cite{botella_oteo_ros_glendinning_13}, and matching is used to show that in some parameter regions, the Lyapunov exponent and topological entropy are constant.  A geometric explanation of matching for a similar family of maps is given in \cite{cosper_misiurewicz_18}, and further implications of matching for these maps---including smoothness of entropy on an open dense subset of parameters---is considered in \cite{bruin_carminati_marmi_profeti_19}.  The notion of matching is extended to random dynamical systems in \cite{dajani_kalle_maggioni_20} and is used to study the asymptotic frequency of the digit $0$ in typical signed binary expansions arising from a family of random interval maps.  Matching has also been investigated for generalised $\beta$-transformations, a certain class of continued fraction expansions with finite digit sets, and Lorenz maps (see \cite{bruin_caminati_kalle_17}, \cite{chen_kraaikamp_22} and \cite{cholewa_oprocha_21}, respectively).  

The present paper exploits the phenomenon of matching in a fashion similar to that of \cite{dajani_kalle_20}.  There the authors use results of \cite{kopf_90}, which gives formulas for densities of the absolutely continuous invariant measures of piecewise linear expanding interval maps.  These densities are---in general---infinite sums of (finite) step functions which are determined by the orbits of the left and right limits at critical points of the underlying interval map.  However, when matching occurs the infinite sum becomes finite, and the density itself is a finite step function depending only on these orbits before matching.  In \cite{dajani_kalle_20}, it is shown that matching occurs for the symmetric doubling map $D_\eta$ on a set of parameters $\eta$ in $[1,2]$ of full Lebesgue measure.  For these matching parameters, the orbits of the left and right limits at the critical points before matching are studied in detail, and this information is used to provide an explicit formula for the density of the (unique) absolutely continuous invariant probability measure for each $D_\eta$ with matching.  The parameter space $[1,2]$ is divided into a countable union of (maximal) open intervals---called \textit{matching intervals}---where each $D_\eta$ has matching, and a Lebesgue-null set of non-matching parameters with Hausdorff dimension $1$.  On each matching interval, matching occurs after the same number of steps, and for each left/right limit at a critical point, the digits of the corresponding signed binary expansions agree before matching.  

While the results of the present paper imply that the same direct approach of understanding matching for the skewed symmetric golden maps $T_\alpha$ can be applied to construct the invariant measures asserted in Theorem \ref{main_measure_thm}, we find that the unequal slopes of the different branches present difficulties.  To circumvent these, we instead study matching for a family of \textit{symmetric golden maps} $\{S_\alpha\}_{\alpha\in[1,\beta]}$ of constant slope for which the skewed symmetric golden maps $\{T_\alpha\}_{\alpha\in[1,\beta]}$ are jump transformations, and it is subsequently shown that the parameters $\alpha$ for which the maps $T_\alpha$ and $S_\alpha$ have matching coincide (Proposition \ref{T_matching_iff_S_matching}).  Equipped with this result, one could then use the formulas from \cite{kopf_90} to determine invariant densities for the $T_\alpha$ with matching; however, we proceed in the simpler setting of the symmetric maps $S_\alpha$---determining invariant densities and the frequencies of digits for these---and finally use the fact that $T_\alpha$ is the jump transformation of $S_\alpha$ to determine invariant measures and frequencies of digits for the original skewed symmetric golden maps.  

The paper is organised as follows.  In \S\ref{Symmetric golden maps} the symmetric golden maps $\{S_\alpha\}_{\alpha\in[1,\beta]}$ are introduced.  These are shown in \S\ref{Matching almost everywhere} to have matching for Lebesgue--a.e. $\alpha\in[1,\beta]$, and we also prove here that the matching parameters of both families $\{S_\alpha\}_{\alpha\in[1,\beta]}$ and $\{T_\alpha\}_{\alpha\in[1,\beta]}$ coincide.  Subsections \ref{Matching words and intervals} and \ref{Cascades of matching intervals} are devoted to understanding the finer structure of the set of matching parameters.  The former provides a classification of all matching intervals and of the orbits of all left and right limits at critical points before matching occurs.  In the latter, it is shown that all (but two) of the matching intervals generate in a natural fashion a whole `cascade' of countably many matching intervals with adjacent endpoints.  In \S\ref{Invariant measures and frequencies of digits} we use the results of the preceding section to prove Theorems \ref{main_measure_thm} and \ref{main_freq_thm}.  In particular, explicit formulas for densities of the unique, absolutely continuous invariant measures of the symmetric golden maps $S_\alpha$ are provided in \S\ref{Invariant measures}, and the invariant measures of the skewed maps $T_\alpha$ are expressed in terms of these.  These measures are used in \S\ref{Frequencies of digits} to determine expressions for the asymptotic frequencies of the digit 0 in typical $S_\alpha$- and $T_\alpha$-expansions.  The maximal frequencies of the digit $0$ as functions of $\alpha$ are considered in \S\ref{Maximal frequency of zero}.. Proofs of some technical results are provided in an appendix (\S\ref{Appendix: proofs of technical lemmas}).

\bigskip

{\flushleft \textbf{Acknowledgments.}} This work is part of project number 613.009.135 of the research programme Mathematics Clusters which is financed by the Dutch Research Council (NWO).

\tikzmath{\B = (1+sqrt(5))/2; \alph0=1.4;} 
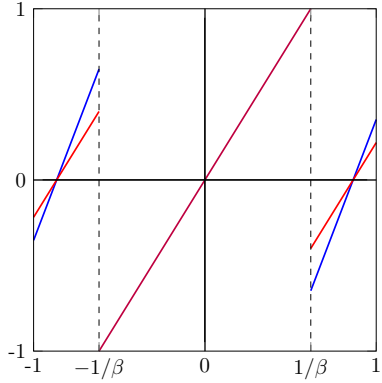
\begin{figure}[t]
\begin{tikzpicture}[scale=.8]
\begin{axis}[axis on top,
axis equal image,
grid style=dashed,
ymin=-1, ymax=1, xmin=-1, xmax=1, ytick={-1,0,1}, xtick={-1,-1/\B,0,1/\B,1}, yticklabels={-1,0,1}, xticklabels={-1,$-1/\beta$,0,$1/\beta$,1}]
\addplot[domain=-1:-1/\B, thick, color=blue]{\B*\B*x+\B*\alph0};
\addplot[domain=-1/\B:1/\B, thick, color=purple]{\B *x};
\addplot[domain=1/\B:1, thick, color=blue]{\B*\B*x-\B*\alph0};
\addplot[domain=-1:-1/\B, thick, color=red]{\B *x+\alph0};
\addplot[domain=1/\B:1, thick, color=red]{\B *x-\alph0};
\draw[color=black, thick] (axis cs: -1,0) -- (axis cs: 1,0);
\draw[color=black, thick] (axis cs: 0,-1) -- (axis cs: 0,1);
\draw[color=black, dashed] (axis cs: -1/\B,-1) -- (axis cs: -1/\B,1);
\draw[color=black, dashed] (axis cs: 1/\B,-1) -- (axis cs: 1/\B,1);
\end{axis}
\end{tikzpicture}
\caption{The maps $T_\alpha$ (blue) and $S_\alpha$ (red) with $\alpha=\alph0$.  Note that $T_\alpha=S_\alpha$ on the middle interval $J_0=[-1/\beta,1/\beta]$.}
\label{T_alpha_pic}
\end{figure}

\section{Symmetric golden maps $S_\alpha$}\label{Symmetric golden maps}

As mentioned in \S\ref{Introduction}, we determine invariant measures and the frequencies of digits for a family of \textit{symmetric golden maps} $\{S_\alpha\}_{\alpha\in[1,\beta]}$ for which the $\{T_\alpha\}_{\alpha\in[1,\beta]}$ are jump transformations.  These invariant measures and frequencies are then used to determine the invariant measures and frequencies of digits for the original $T_\alpha$.  The maps $S_\alpha$ are defined as follows: for $\alpha\in[1,\beta]$, let $S_\alpha:[-1,1]\to[-1,1]$ be given by 
\[
S_\alpha(x):=\beta x-t(x)\alpha,
\]
with $t(x)\in\{-1,0,1\}$ as in $\S\ref{Introduction}$; see Figure \ref{T_alpha_pic}.  Note that $S_\alpha(x)\in J_0$ for each $x\in J_{-1}\cup J_1$.  Using this, one readily verifies that
\begin{equation}\label{T_in_terms_of_S}
T_\alpha(x)=\begin{cases}
S_\alpha(x), & x\in J_0\\
S_\alpha^2(x), & x\in J_{-1}\cup J_1
\end{cases},
\end{equation}
i.e. $T_\alpha$ is the jump transformation for $S_\alpha$ with respect to the sweep-out set $J_0=[-1/\beta,1/\beta]$ (see, e.g. \S11.4 of \cite{dajani_kalle_21}).  For each $j\ge 1$, let $s_{\alpha,j}(x):=t(S_\alpha^{j-1}(x)).$  With induction one finds that for each $k\ge 0$,
\begin{equation}\label{S_alpha^k_eqn}
S_\alpha^k(x)=\beta^k\left(x-\alpha\sum_{j=1}^{k}s_{\alpha,j}(x)/\beta^j\right)
\end{equation}
(with the summation for $k=0$ understood to be $0$).  
Since $|S_\alpha^k|\le 1$, dividing both sides by $\beta^k$ and taking the limit as $k$ approaches infinity gives
\begin{equation}\label{S_alpha_expansion}
x=\alpha\sum_{j\ge 1}s_{\alpha,j}(x)/\beta^j.
\end{equation}
Following our convention from \S\ref{Introduction}, we refer to both the right-hand side of Equation (\ref{S_alpha_expansion}) and the corresponding sequence $(s_{\alpha,j}(x))_{j\ge 1}$ of digits in $\{0,\pm 1\}^\mathbb{N}$ as the \textit{$S_\alpha$-expansion} of $x$.  Again this process determines---for fixed $\alpha$---a unique expansion for each $x\in[-1,1]$; moreover, if $x,y\in [-1,1]$ have the same $S_\alpha$-expansion, then Equation (\ref{S_alpha^k_eqn}) can be used to show that $x=y$.  Also note that not every sequence in $\{0,\pm 1\}^\mathbb{N}$ is an $S_\alpha$-expansion; in particular, a $1$ or $-1$ is necessarily followed by a $0$.

As the orbits of $1$ and $1-\alpha$ will be studied in detail below, we fix special notation for their $S_\alpha$-expansions: let $d_{\alpha,j}:=s_{\alpha,j}(1)$ and $e_{\alpha,j}:=s_{\alpha,j}(1-\alpha)$ for each $\alpha\in[1,\beta]$ and $j\ge 1$.  When $\alpha$ is understood, it is suppressed from the notation, and we simply write $d_j:=d_{\alpha,j}$ and $e_j:=e_{\alpha,j}$.

\subsection{Matching almost everywhere}\label{Matching almost everywhere}

In this section, we show that the maps $S_\alpha$ (and $T_\alpha$) have matching on a set of full Lebesgue measure.\footnote{The general approach to proving this result largely follows that of \S2.2 of \cite{dajani_kalle_20}; however, we shall see that the dynamics of the symmetric golden maps $S_\alpha$ are---in a sense---more delicate than those of the previously studied symmetric binary maps (compare, e.g. Proposition \ref{difference_of_orbits} below with Proposition 2.1 of \cite{dajani_kalle_20}).} The map $S_\alpha$ has two critical points $\pm1/\beta$.  Due to symmetry, it suffices to consider the matching criteria only for the positive critical point $1/\beta$.  Note that $\lim_{x\to 1/\beta^-}S_\alpha(x)=1$ and $\lim_{x\to 1/\beta^+}S_\alpha(x)=1-\alpha$.  Hence $S_\alpha$ has matching if and only if there are integers $M,N\ge 1$ for which $S_\alpha^M(1)=S_\alpha^N(1-\alpha)$.

We begin by investigating matching in a number of specific cases.  First, note that $1\in J_1$ and $1-\alpha\in J_0$ for all $\alpha\in[1,\beta].$

\begin{enumerate}
\item[(i)]  If $\alpha\in(1+1/\beta^2,\beta]$, then
\begin{alignat*}{2}
S_\alpha(1)&=\beta-\alpha\in[0,1/\beta^3) \subset J_0,\ & S_\alpha(1-\alpha)&=\beta-\beta\alpha\in[-1,-1/\beta) \subset J_{-1},\\
S_\alpha^2(1)&=\beta^2-\beta\alpha\in J_0\ \ \ \ \ \ \ \text{and}\ \ \ \ \ \ \ & S_\alpha^2(1-\alpha)&=\beta^2-\beta^2\alpha+\alpha=\beta^2-\beta\alpha\in J_0
\end{alignat*}
shows that $S_\alpha$ has matching with $M=N=2$.  

\item[(ii)] If $\alpha=1+1/\beta^2$, then
\begin{alignat*}{2}
S_\alpha(1)&=\beta-\alpha=1/\beta^3 \in J_0, & S_\alpha(1-\alpha)&=\beta-\beta\alpha=-1/\beta\in J_0,\\
S_\alpha^2(1)&=1/\beta^2\in J_0, & S_\alpha^2(1-\alpha)&=-1\in J_{-1},\\
S_\alpha^3(1)&=1/\beta\in J_0,\ \ \ \ & S_\alpha^3(1-\alpha)&=-1/\beta^3\in J_0,\\
S_\alpha^4(1)&=1\in J_1\ \ \ \ \ \ \ \ \ \ \text{and}\ \ \ \ \ \ \ \ & S_\alpha^4(1-\alpha)&=-1/\beta^2=1-\alpha\in J_0,
\end{alignat*}
so $S_\alpha$ has a Markov partition, namely 
\[\left\{[-1/\beta^3/1/\beta^3],\ \pm(1/\beta^3,1/\beta^2],\ \pm(1/\beta^2,1/\beta],\ \pm(1/\beta,1]\right\},\]
and no matching.  

\item[(iii)]  If $\alpha\in(1+1/\beta^3,1+1/\beta^2)$,
\begin{alignat*}{2}
S_\alpha(1)&=\beta-\alpha\in(1/\beta^3,1/\beta^2)\subset J_0, & S_\alpha(1-\alpha)&=\beta-\beta\alpha\in(-1/\beta,-1/\beta^2)\subset J_0, \\
S_\alpha^2(1)&=\beta^2-\beta\alpha\in(1/\beta^2,1/\beta)\subset J_0, & S_\alpha^2(1-\alpha)&=\beta^2-\beta^2\alpha\in(-1,-1/\beta)\subset J_{-1},\\
S_\alpha^3(1)&=\beta^3-\beta^2\alpha\in(1/\beta,1)\subset J_1, & S_\alpha^3(1-\alpha)&=\beta^3-(\beta^3-1)\alpha\in(-1/\beta^3,1/\beta^3)\subset J_0,\\
S_\alpha^4(1)&=\beta^4-(\beta^3+1)\alpha\in J_0\ \ \ \ \ \ \ \text{and}\ \ \ \ \ \ \ & S_\alpha^4(1-\alpha)&=\beta^4-(\beta^4-\beta)\alpha\in J_0.
\end{alignat*}
Since $\beta^4-\beta^3=\beta^2=\beta+1$, we find that $S^4(1)=S^4(1-\alpha)$, so $S_\alpha$ has matching with $M=N=4$. 

\item[(iv)] If $\alpha=1+1/\beta^3$, 
\begin{alignat*}{2}
S_\alpha(1)&=\beta-\alpha=1/\beta^2\in J_0, \ \ \ \ & S_\alpha(1-\alpha)&=\beta-\beta\alpha=-1/\beta^2\in J_0,\\
S_\alpha^2(1)&=1/\beta\in J_0, & S_\alpha^2(1-\alpha)&=-1/\beta\in J_0,\\
S_\alpha^3(1)&=1\in J_1,\ \ \ \ & S_\alpha^3(1-\alpha)&=-1\in J_{-1}\ \ \ \text{and}\\
& & S_\alpha^4(1-\alpha)&=-1/\beta^2\in J_0,
\end{alignat*}
so $S_\alpha$ has a Markov partition and no matching. 

\item[(v)]  If $\alpha\in(1,1+1/\beta^3)$, then
\begin{alignat*}{2}
S_\alpha(1)&=\beta-\alpha\in(1/\beta^2,1/\beta)\subset J_0, & S_\alpha(1-\alpha)&=\beta-\beta\alpha\in(-1/\beta^2,0)\subset J_0,\\
S_\alpha^2(1)&=\beta^2-\beta\alpha\in(1/\beta,1)\subset J_1, & S_\alpha^2(1-\alpha)&=\beta^2-\beta^2\alpha\in(-1/\beta,0)\subset J_0,\\
S_\alpha^3(1)&=\beta^3-(\beta^2+1)\alpha\in(-1/\beta^3,1/\beta)\subset J_0 & \ \ \ \ \ \text{and}\ \ \ \ \ S_\alpha^3(1-\alpha)&=\beta^3-\beta^3\alpha\in(-1,0)\subset J_{-1}\cup J_0.
\end{alignat*}
This case will be considered more closely in what follows.

\item[(vi)]  If $\alpha=1$, then $S_\alpha(1)=1/\beta\in J_0,\ S_\alpha^2(1)=1\in J_1$ and $S_\alpha(1-\alpha)=0=1-\alpha\in J_0$.  Thus there is a Markov partition and no matching. 
\end{enumerate}

Note that in the cases above in which there is matching---namely (i) and (iii)---we have $M=N$ (a property called \textit{neutral matching} in \cite{bruin_carminati_marmi_profeti_19}).  We shall see below that this is always the case, i.e. $S_\alpha$ has matching if and only if there is some $m\ge 1$ for which $S_\alpha^m(1)=S_\alpha^m(1-\alpha)$.  For this we need the following proposition---key to a number of arguments throughout---which states that the difference between subsequent points in the orbits of $1$ and $1-\alpha$ can take on at most four values.  Recall that $(d_j)_{j\ge 1}$ and $(e_j)_{j\ge 1}$ denote the $S_\alpha$-expansions of $1$ and $1-\alpha$, respectively.  

\begin{prop}\label{difference_of_orbits}
For every $\alpha\in[1,\beta]$ and $j\ge 0$,
\[S_\alpha^j(1)-S_\alpha^j(1-\alpha)\in\{0,\alpha/\beta,\alpha,\beta\alpha\}.\]
\end{prop}
\begin{proof}
For $\alpha\notin (1,1+1/\beta^3)$, the statement is verified with the cases above, so assume $\alpha\in (1,1+1/\beta^3)$.  We use induction on $j$.  The result clearly holds for $j=0$; assume for some $j=k-1\ge 0$ that
\[S_\alpha^{k-1}(1)-S_\alpha^{k-1}(1-\alpha)=y\]
for some $y\in\{0,\alpha/\beta,\alpha,\beta\alpha\}$.  If $y=0,$ then also $S_\alpha^j(1)-S_\alpha^j(1-\alpha)=0$ for all $j\ge k-1$.  Suppose $y\neq 0$, and note that
\[S_\alpha^{k}(1)-S_\alpha^{k}(1-\alpha)=(\beta S_\alpha^{k-1}(1)-d_{k}\alpha)-(
\beta S_\alpha^{k-1}(1-\alpha)-e_{k}\alpha)=\beta y-(d_{k}-e_{k})\alpha.\]
We determine the difference above for each $y\in\{\alpha/\beta,\alpha,\beta\alpha\}$: 
\begin{enumerate}
\item[(i)] $y=\alpha/\beta$:  Since $1/\beta<y<2/\beta$, we have $(d_{k},e_{k})=(1,0), (0,-1)$ or $(0,0)$.  In the first two cases
\[S_\alpha^{k}(1)-S_\alpha^{k}(1-\alpha)=0,\]
and in the third 
\[S_\alpha^{k}(1)-S_\alpha^{k}(1-\alpha)=\alpha.\]

\item[(ii)] $y=\alpha$:  Since $1/\beta<y<1+1/\beta^3=2/\beta$, we again have $(d_{k},e_{k})=(1,0), (0,-1)$ or $(0,0)$.  In the first two cases
\[S_\alpha^{k}(1)-S_\alpha^{k}(1-\alpha)=\beta\alpha-\alpha=\alpha/\beta,\]
and in the third
\[S_\alpha^{k}(1)-S_\alpha^{k}(1-\alpha)=\beta\alpha.\]

\item[(iii)] $y=\beta\alpha$:  Since $y>2/\beta$, we must have $(d_{k},e_{k})=(1,-1)$, and hence
\[S_\alpha^{k}(1)-S_\alpha^{k}(1-\alpha)=\beta^2\alpha-2\alpha=\alpha/\beta.\]
\end{enumerate}
\end{proof}

The previous proposition can be used to give an equivalent definition of matching:
\begin{prop}
The map $S_\alpha$ has matching if and only if there is some $m\ge 1$ for which $S_\alpha^m(1)=S_\alpha^m(1-\alpha)$.
\end{prop}
\begin{proof}
One direction is immediate; for the other, suppose there are distinct $M,N\ge 1$ for which $S_\alpha^M(1)=S_\alpha^N(1-\alpha)$.  Assume for the sake of contradiction that $S_\alpha^j(1)\neq S_\alpha^j(1-\alpha)$ for all $j\ge 1$.  By Proposition \ref{difference_of_orbits},
\[S_\alpha^j(1)-S_\alpha^j(1-\alpha)\ge \alpha/\beta\ge1/\beta,\]
and hence 
\[S_\alpha^j(1-\alpha)\le S^j(1)-1/\beta\le 1-1/\beta=1/\beta^2\]
for each $j$.  If $S_\alpha^j(1-\alpha)\in(0,1/\beta^2]$, then there is some $k\ge 0$ for which $S_\alpha^{j+k}(1-\alpha)=\beta^kS_\alpha^j(1-\alpha)>1/\beta^2$, contradicting the above, and thus $S_\alpha^j(1-\alpha)\le 0$ for each $j$.  A similar argument implies $S_\alpha^j(1)\ge 0$ for each $j$.  But $S_\alpha^M(1)=S_\alpha^N(1-\alpha)$, so this common value must be $0$.  Since $0$ is fixed by $S_\alpha$, we have the contradiction that $S_\alpha^m(1)=0=S_\alpha^m(1-\alpha)$ with $m=\text{max}\{M,N\}$.
\end{proof}

We can now define a canonical index to describe when matching occurs:
\begin{defn}
The \textit{matching index} of $S_\alpha$ is
\[m(\alpha):=\inf\{m\ge 1\ |\ S_\alpha^m(1)=S_\alpha^m(1-\alpha)\}\in\mathbb{N}\cup\{\infty\}.\]
\end{defn}

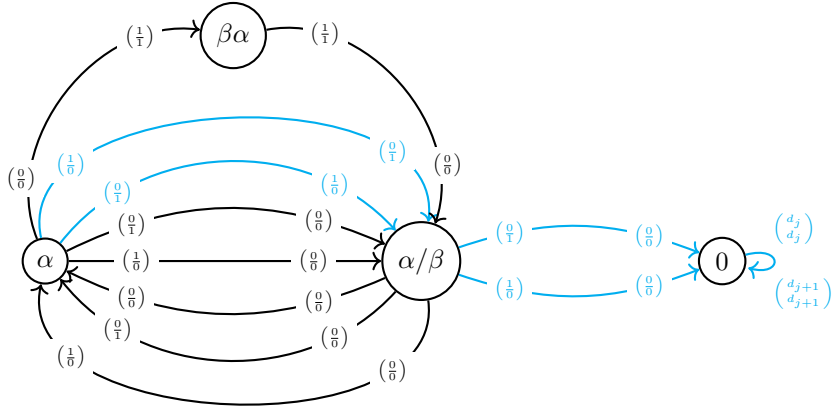
\begin{figure}
\begin{tikzpicture}
\begin{scope}[every node/.style={circle,thick,draw}]
	\node (0) at (4,0) {$0$};
	\node (a/b) at (0,0) {$\alpha/\beta$};
	\node (a) at (-5,-0) {$\alpha$};
	\node (ba) at (-2.5,3) {$\beta\alpha$};
\end{scope}
\begin{scope}[every edge/.style={draw,thick},
		      every node/.style={fill=white}]
	\path [->] (a/b) edge[bend left=20,color=cyan] node[scale=.8, pos=.2] {\tiny $\left(\begin{smallmatrix}0\\ \overline{1}\end{smallmatrix}\right)$} node[scale=.8, pos=.8] {\tiny $\left(\begin{smallmatrix}0\\ \overline{0}\end{smallmatrix}\right)$} (0);
	\path [->] (a/b) edge[bend right=20,color=cyan] node[scale=.8, pos=.2] {\tiny $\left(\begin{smallmatrix}1\\ \overline{0}\end{smallmatrix}\right)$} node[scale=.8, pos=.8] {\tiny $\left(\begin{smallmatrix}0\\ \overline{0}\end{smallmatrix}\right)$} (0);
	\path [->] (a/b) edge[bend left=100] node[scale=.8, pos=.2] {\tiny $\left(\begin{smallmatrix}0\\ \overline{0}\end{smallmatrix}\right)$} node[scale=.8, pos=.8] {\tiny $\left(\begin{smallmatrix}1\\ \overline{0}\end{smallmatrix}\right)$} (a);
	\path [->] (a/b) edge[bend left=50] node[scale=.8, pos=.2] {\tiny $\left(\begin{smallmatrix}0\\ \overline{0}\end{smallmatrix}\right)$} node[scale=.8, pos=.8] {\tiny $\left(\begin{smallmatrix}0\\ \overline{1}\end{smallmatrix}\right)$} (a);
	\path [->] (a/b) edge[bend left=25] node[scale=.8, pos=.2] {\tiny $\left(\begin{smallmatrix}0\\ \overline{0}\end{smallmatrix}\right)$} node[scale=.8, pos=.8] {\tiny $\left(\begin{smallmatrix}0\\ \overline{0}\end{smallmatrix}\right)$} (a);	
	\path [->] (a) edge[bend left=0] node[scale=.8, pos=.2] {\tiny $\left(\begin{smallmatrix}1\\ \overline{0}\end{smallmatrix}\right)$} node[scale=.8, pos=.8] {\tiny $\left(\begin{smallmatrix}0\\ \overline{0}\end{smallmatrix}\right)$}(a/b);		
	\path [->] (a) edge[bend left=100,color=cyan] node[scale=.8, pos=.2] {\tiny $\left(\begin{smallmatrix}1\\ \overline{0}\end{smallmatrix}\right)$} node[scale=.8, pos=.8] {\tiny $\left(\begin{smallmatrix}0\\ \overline{1}\end{smallmatrix}\right)$} (a/b);		
	\path [->] (a) edge[bend left=50,color=cyan] node[scale=.8, pos=.2] {\tiny $\left(\begin{smallmatrix}0\\ \overline{1}\end{smallmatrix}\right)$} node[scale=.8, pos=.8] {\tiny $\left(\begin{smallmatrix}1\\ \overline{0}\end{smallmatrix}\right)$} (a/b);		
	\path [->] (a) edge[bend left=25] node[scale=.8, pos=.2] {\tiny $\left(\begin{smallmatrix}0\\ \overline{1}\end{smallmatrix}\right)$} node[scale=.8, pos=.8] {\tiny $\left(\begin{smallmatrix}0\\ \overline{0}\end{smallmatrix}\right)$} (a/b);		
	\path [->] (a) edge[bend left=60] node[left=-.2, scale=.8, pos=.2] {\tiny $\left(\begin{smallmatrix}0\\ \overline{0}\end{smallmatrix}\right)$} node[scale=.8, pos=.8] {\tiny $\left(\begin{smallmatrix}1\\ \overline{1}\end{smallmatrix}\right)$} (ba);	
	\path [->] (ba) edge[bend left=60] node[scale=.8, pos=.2] {\tiny $\left(\begin{smallmatrix}1\\ \overline{1}\end{smallmatrix}\right)$} node[right=-.2, scale=.8, pos=.8] {\tiny $\left(\begin{smallmatrix}0\\ \overline{0}\end{smallmatrix}\right)$} (a/b);
	\path [->] (0) edge[loop right,color=cyan] node[above right=.02,scale=.8, pos=.2] {\tiny $\left(\begin{smallmatrix}d_{j}\\ d_{j}\end{smallmatrix}\right)$} node[below right=.03, scale=.8, pos=.8] {\tiny $\left(\begin{smallmatrix}d_{j+1}\\ d_{j+1}\end{smallmatrix}\right)$} ();
\end{scope}
\end{tikzpicture}
\caption[]{A graphical representation of the interdependence of the orbits of $1$ and $1-
\alpha$ for $\alpha\in[1,\beta]$.  Vertices represent the differences $S_\alpha^{j-1}(1)-S_\alpha^{j-1}(1-\alpha)$ for $j\ge 1$, and the beginnings and ends of edges are marked $\left(\begin{smallmatrix}d_{j}\\ e_{j}\end{smallmatrix}\right)$ and $\left(\begin{smallmatrix}d_{j+1}\\ e_{j+1}\end{smallmatrix}\right)$, respectively, where $\overline{w}:=-w$ for $w\in\{0,\pm 1\}$.  Cyan edges are taken if and only if $S_\alpha$ has matching.} 
\label{differences_graph}
\end{figure}

The cases above together with the proof of Proposition \ref{difference_of_orbits} reveal a strong interdependence between the orbits of $1$ and $1-\alpha$, which is summarised in the graph of Figure \ref{differences_graph}.  In particular, note that if matching occurs with matching index $m:=m(\alpha)$, then $S_{\alpha}^{m-1}(1)-S_\alpha^{m-1}(1-\alpha)=\alpha/\beta$ and $(d_m,e_m)\in\{(1,0),(0,-1)\}$.  Since $S_\alpha$-expansions cannot contain consecutive non-zero digits, this implies $S_{\alpha}^{m-2}(1)-S_\alpha^{m-2}(1-\alpha)=\alpha$ and $(d_{m-1},e_{m-1})\in\{(1,0),(0,-1)\}$.  For $m>2$, this further implies $S_{\alpha}^{m-3}(1)-S_\alpha^{m-3}(1-\alpha)=\alpha/\beta$ and $(d_{m-2},e_{m-2})=(0,0)$.  Thus if $S_\alpha$ has matching with index $m>2$, then the final three digits of the $S_\alpha$-expansions of $1$ and $1-\alpha$ before matching are given by
\begin{equation}\label{matching_final_digits}
\begin{pmatrix}d_{m-2}d_{m-1}d_m\\ e_{m-2}e_{m-1}e_m\end{pmatrix}\in\left\{\begin{pmatrix}010\\\overline{001}\end{pmatrix},\begin{pmatrix}001\\\overline{010}\end{pmatrix}\right\},
\end{equation}
where $\overline{w}:=-w$ for $w\in\{0,\pm 1\}$.  Conversely, if for some $m>2$, three consecutive digits of the $S_\alpha$-expansions of $1$ and $1-\alpha$ are given by (\ref{matching_final_digits}), then the proof implies that $S_\alpha$ has matching with index $m$.

A number of characterisations of matching for $S_\alpha$ can be derived from Proposition \ref{difference_of_orbits} and Figure \ref{differences_graph}.  For these we fix some notation: for each $x\in[-1,1]$ and $\alpha\neq 1$, let
\[\ell_\alpha(x):=\inf_{j\ge 0}\{S_\alpha^{j}(|x|)\le 0\}-1,\]
and set
\[\ell_\alpha:=\min\{\ell_\alpha(1),\ell_\alpha(1-\alpha)\}.\]

\begin{lem}\label{ell(alpha)}
For $\alpha\neq 1$, $S_\alpha$ has matching if and only if $\ell_\alpha<\infty$.  Moreover, if $\ell_\alpha<\infty$, then $m(\alpha)\in\{\ell_\alpha+1,\ell_\alpha+2\}$.  
\end{lem}
\begin{proof}
Let $\ell:=\ell_\alpha$.  That matching implies $\ell<\infty$ is immediate.  Now suppose $\ell<\infty$, and assume without loss of generality that $\ell=\ell_\alpha(1-\alpha)$ and thus $S_\alpha^{\ell+1}(1-\alpha)\ge 0$ (the other case is similar).  The definitions of $\ell$ and $m(\alpha)$ give $\ell+1\le m(\alpha)$.  By Proposition \ref{difference_of_orbits}, $S_\alpha^{\ell+1}(1-\alpha)\ge 0$ and $\alpha>1$ imply
\[S_\alpha^{\ell+1}(1)-S_\alpha^{\ell+1}(1-\alpha)\in\{0,\alpha/\beta\}.\]
The result holds if the difference is $0$.  If the difference is $\alpha/\beta$, we must have $(d_{\ell+2},e_{\ell+2})=(1,0)$.  From Figure \ref{differences_graph}, this implies 
\[S_\alpha^{\ell+2}(1)-S_\alpha^{\ell+2}(1-\alpha)=0.\]
\end{proof}

\begin{cor}\label{S_alpha_matching_hole}
For $\alpha\neq 1$, $S_\alpha$ has matching if and only if there exists some $j\ge 1$ such that 
\[S_\alpha^j(1)\in (1/\beta,\alpha/\beta]\ \ \ \text{or}\ \ \ S_{\alpha}^j(1-\alpha)\in[-\alpha/\beta,-1/\beta).\]
Moreover, $\ell_\alpha(1)$ and $\ell_\alpha(1-\alpha)$, respectively, are the infimums over all $j$ for which the above inclusions hold.
\end{cor}
\begin{proof}
This follows from Lemma \ref{ell(alpha)} and the facts that 
\[S_\alpha^{-1}([-1,0])\cap(0,1]=(1/\beta,\alpha/\beta]\]
and
\[S_\alpha^{-1}([0,1])\cap[-1,0)=[-\alpha/\beta,1/\beta).\]
\end{proof}

Due to symmetry, the above corollary states that $S_\alpha$ has matching if and only if the orbit of either $1$ or of $\alpha-1$ enters the region $(1/\beta,\alpha/\beta]$.  We shall see that this occurs for Lebesgue--a.e. $\alpha$ by relating the beginnings of these orbits to the beginnings of certain orbits of the (ergodic) \textit{$\beta$-transformation} $B:[0,1]\to[0,1]$ defined by $B(x)=\beta x\ (\text{mod}\ 1)$.  Set
\[
b(x):=\begin{cases}
0, & x<1/\beta\\
1, & x\ge 1/\beta
\end{cases},
\]
and for each $j\ge 1$, let
\[b_j(x):=b(B^{j-1}(x)).\]
We call the sequence $(b_j(x))_{j\ge 1}$ the \textit{$\beta$-expansion} (also referred to as the \textit{greedy-expansion}) of $x$.  Via induction, one finds that for each $k\ge 0$,
\begin{equation}\label{B^k_eqn}
B_\alpha^k(x)=\beta^k\left(x-\sum_{j=1}^{k}b_j(x)/\beta^j\right).
\end{equation}

\begin{lem}\label{S_alpha_and_B}
Let $x\in\{1,\alpha-1\},\ \alpha\neq 1$.  Then
\begin{enumerate}
\item[(i)] $S_\alpha^j(x)=\alpha B^j(x/\alpha)$ for each $0\le j\le\ell_\alpha(x)$,

\item[(ii)] $s_{\alpha,j}(x)=b_{j}(x/\alpha)$ for each $1\le j\le\ell_\alpha(x)$ and

\item[(iii)] $\ell_\alpha(x)$ is the infimum over all $j$ for which $B^j(x/\alpha)\in (1/\beta\alpha,1/\beta]$.
\end{enumerate}
\end{lem}
\begin{proof}
Claim (iii) will follow from claim (i), Corollary \ref{S_alpha_matching_hole} and the fact that $\ell_\alpha(x)=\ell_\alpha(-x)$.  We prove claim (i) via induction on $j$.  Certainly $S_\alpha^j(x)=\alpha B^j(x/\alpha)$ for $j=0$.  Now suppose this equality holds for some $j=k-1$ with $0\le k-1<\ell_\alpha(x)$.  By Corollary \ref{S_alpha_matching_hole}, $S_\alpha^{k-1}(x)\in[0,1]\backslash(1/\beta,\alpha/\beta]$, and we find
\begin{align*}
S_\alpha^k(x)&=\begin{cases}
\beta S_\alpha^{k-1}(x), & S_\alpha^{k-1}(x)\in[0,1/\beta]\\
\beta S_\alpha^{k-1}(x)-\alpha, & S_\alpha^{k-1}(x)\in(\alpha/\beta,1]\\
\end{cases}\\
&=\begin{cases}
\beta\alpha B^{k-1}(x/\alpha), & B_\alpha^{k-1}(x/\alpha)\in[0,1/\beta\alpha]\\
\beta\alpha B^{k-1}(x/\alpha)-\alpha, & B_\alpha^{k-1}(x/\alpha)\in(1/\beta,1/\alpha]\\
\end{cases}\\
&=\alpha B^{k}(x/\alpha),
\end{align*}
so the first claim holds. Furthermore, the equality in (i) gives for each $1\le j\le\ell_\alpha(x)$ that $S_\alpha^{j-1}(x)\in[0,1/\beta]$ if and only if $B^{j-1}(x/\alpha)\in[1,1/\beta\alpha]$ and $S_\alpha^{j-1}(x)\in(\alpha/\beta,1]$ if and only if $B^{j-1}(x/\alpha)\in(1/\beta,1/\alpha]$.  Thus $s_{\alpha,j}(x)=b_{j}(x/\alpha)$ for such $j$, proving claim (ii).
\end{proof}

Corollary \ref{S_alpha_matching_hole}, Lemma \ref{S_alpha_and_B} and symmetry of $S_\alpha$ give yet another characterisation of matching in terms of the map $B$:

\begin{cor}\label{B_matching_hole}
For $\alpha\neq 1$, $S_\alpha$ has matching if and only if there exists some $j\ge0$ such that 
\[B^j(1/\alpha)\in (1/\beta\alpha,1/\beta]\ \ \ \text{or}\ \ \ B^j(1-1/\alpha)\in (1/\beta\alpha,1/\beta].\]
Moreover, $\ell_\alpha(1)$ and $\ell_\alpha(1-\alpha)$, respectively, are the infimums over all $j$ for which the above inclusions hold.
\end{cor}

The previous results together with ergodicity of $B$ can now be used to prove that $S_\alpha$ has matching for a set of parameters $\alpha$ of full Lebesgue measure.  The proof is nearly identical to that of Proposition 2.3 of \cite{dajani_kalle_20} but is included here for the ease of the reader.

\begin{prop}\label{matching_ae}
The map $S_\alpha$ has matching for Lebesgue--a.e. $\alpha\in[1,\beta]$.  
\end{prop}
\begin{proof}
Let $\alpha\in(1,\beta]$ and $k\in\mathbb{N}$ with $k>\beta^3$.  By ergodicity of $B$ with respect to Lebesgue measure (\S 4 of \cite{renyi_57}), for Lebesgue--a.e. $x\in[0,1]$ there exists some $j\ge 1$ such that $B^j(x)\in(1/\beta-1/k,1/\beta]$.  Note that $1/\beta\alpha<1/\beta-1/k$ if and only if $\alpha>k/(k-\beta)$.  Thus for Lebesgue--a.e. $\alpha\in(k/(k-\beta),\beta]$, there exists some $j\ge 1$ such that 
\[B^j(1/\alpha)\in(1/\beta-1/k,1/\beta]\subset(1/\beta\alpha,1/\beta].\]
By Corollary \ref{B_matching_hole}, $S_\alpha$ has matching for Lebesgue--a.e. $\alpha\in(k/(k-\beta),\beta]$.  Let $A_k$ denote the set of $\alpha\in(k/(k-\beta),\beta]$ for which $S_\alpha$ does not have matching.  Then $\cup_{k>\beta^3}A_k$ has Lebesgue measure $0$ and equals the set of all $\alpha\in(1,\beta]$ for which $S_\alpha$ does not have matching.  
\end{proof}

The finer structure of the set of matching parameters $\alpha\in[1,\beta]$ is considered in \S\S\ref{Matching words and intervals} and \ref{Cascades of matching intervals} below.  Before investigating this structure, we show that matching occurs for $S_\alpha$ if and only if it occurs for the corresponding jump transformation $T_\alpha$.  The following lemma may be deduced from the general theory of jump transformations, but a proof is included for completeness.

\begin{lem}\label{T_as_passage_of_S}
Fix $x\in[-1,1]$ and let $j_1<j_2<j_3<\dots$ be an enumeration of the set
\[\{j\ge 0\ |\ S_\alpha^j(x)\in J_0\}.\]
Then $T_\alpha^k(x)=S_\alpha^{j_k+1}(x)$ for all $k\ge 1$.
\end{lem}
\begin{proof}
The claim is immediate for $k=1$ by (\ref{T_in_terms_of_S}) and the fact that $S_\alpha(J_{-1}\cup J_1)\subset J_0$.  Now suppose the result holds for some $k\ge 1$, and let $i\in\{0,1\}$ be minimal such that $S_\alpha^i(S_\alpha^{j_k+1}(x))\in J_0.$  By definition, then, $j_{k+1}=j_k+i+1$, and 
\[T_\alpha^{k+1}x=T_{\alpha}(S_\alpha^{j_k+1}x)=S_\alpha^{i+1}(S_\alpha^{j_k+1}x)=S_\alpha^{j_{k+1}+1}x.\]
\end{proof}

\begin{prop}\label{T_matching_iff_S_matching}
The matching parameters $\alpha\in[1,\beta]$ for $T_\alpha$ and for $S_\alpha$ coincide. 
\end{prop}
\begin{proof}
Recall that $T_\alpha$ has critical points at $\pm 1/\beta$, and note that $\lim_{x\to 1/\beta^-}T_\alpha(x)=1$ while $\lim_{x\to 1/\beta^+}T_\alpha(x)=\beta(1-\alpha)$.  Due to symmetry, $T_\alpha$ has matching if and only if there are integers $M,N>0$ for which $T_\alpha^M(1)=T_\alpha^N(\beta(1-\alpha))$.

Suppose first that $T_\alpha$ has matching.  Then $T_\alpha^M(1)=T_\alpha^N(\beta(1-\alpha))$ for some $M,N>0$.  By (\ref{T_in_terms_of_S}) and the fact that $S_\alpha(1-\alpha)=\beta(1-\alpha)$, this implies the existence of some $M',N'>0$ for which $S_\alpha^{M'}(1)=S_\alpha^{N'}(1-\alpha)$.  

Conversely, suppose $S_\alpha$ has matching with matching index $m:=m(\alpha)$.  From the proof of Proposition \ref{difference_of_orbits} it is clear that $S_\alpha^m(1)=S_\alpha^m(1-\alpha)\in J_0$.  By Lemma \ref{T_as_passage_of_S}, there are $M,N>0$ for which
\[T_\alpha^M(1)=S_\alpha^{m+1}(1)=S_\alpha^{m+1}(1-\alpha)=S_\alpha^m(\beta(1-\alpha))=T_\alpha^N(\beta(1-\alpha)).\]
\end{proof}

\subsection{Matching words and intervals}\label{Matching words and intervals}

When $S_\alpha$ has matching, we call the first $m(\alpha)<\infty$ digits of the $S_\alpha$-expansion of $1$ the \textit{matching word} corresponding to $\alpha$.  A maximal subinterval of $[1,\beta]$ on which matching words coincide is called a \textit{matching interval} corresponding to the common matching word.  Here we classify matching words and matching intervals (Corollary \ref{matching_iff_property_M}); as all matching parameters belong to some matching interval, this gives a complete classification of matching parameters $\alpha\in[1,\beta]$.  (Propositions \ref{matching_implies_property_M}, \ref{I_d_determines_expansions} and \ref{matching_intervals_thm} imply that this also classifies the first $m(\alpha)<\infty$ digits of the $S_\alpha$-expansions of $1-\alpha$ for $S_\alpha$ with matching and the maximal subintervals of parameters $\alpha$ on which these digits coincide.)  Note that matching words and intervals for $\alpha\in[1,\beta]\backslash (1,1+1/\beta^3)$ have been implicitly determined via the cases considered in \S\ref{Matching almost everywhere}.  For instance, $(1+1/\beta^2,\beta]$ is the matching interval corresponding to the matching word $10$, and the $S_\alpha$-expansion of $1-\alpha$ for each $\alpha\in(1+1/\beta^2,\beta]$ begins with $0(-1)$.  Similarly, $(1+1/\beta^3,1+1/\beta^2)$ is the matching interval corresponding to the matching word $1001$, and the $S_\alpha$-expansion of $1-\alpha$ for each $\alpha$ in this interval begins with $00(-1)0$.  

Denote by $\prec$ the lexicographical ordering on $\{0,\pm 1\}^\mathbb{N}$.  Note that $\prec$ may also be defined on the set $\{0,\pm 1\}^*$ of finite words with alphabet $-1,0,1$ by first sending $\mathbf w\in\{0,\pm 1\}^*$ to $\mathbf w0^\infty$.  
\begin{defn}
Let
\[\mathbf{w}_0:=00\prec\mathbf{w}_1:=001\prec\mathbf{w}_2:=01.\]
We say that $\mathbf d\in\{0,1\}^*$ is in \textit{admissible block form} if $\mathbf d=10$ or
\[\mathbf d=1\mathbf{w}_{i_1}\mathbf{w}_{i_2}\cdots\mathbf{w}_{i_n}(1-i_n/2)\]
for some $i_1,\dots,i_n\in\{0,1,2\},\ n\ge 1$ with $i_n\neq 1$, and, when $n\ge 2$, $i_1=2$.  The collection of all words in admissible block form is denoted $\mathcal{B}$.
\end{defn}

The condition that a word in admissible block form ends in $\mathbf w_{i_n}(1-i_n/2),\ i_n\neq 1$, guarantees that the final three digits are either $001$ or $010$ (recall (\ref{matching_final_digits})); however, not every word ending this way belongs to $\mathcal{B}$:

\begin{eg}\label{adm_block_form_eg}
One verifies that
\[\mathbf d:=1\mathbf w_2\mathbf w_0\mathbf w_1 \mathbf w_0 1=10100001001\in\mathcal{B},\]
whereas
\[\mathbf d':=1010001\notin\mathcal{B}.\]
\end{eg}

Note that the indices $i_j$ for $\mathbf d\in\mathcal{B}$ are uniquely determined; that is, if
\[1\mathbf{w}_{i_1}\mathbf{w}_{i_2}\cdots\mathbf{w}_{i_n}(1-i_n/2)=1\mathbf{w}_{j_1}\mathbf{w}_{j_2}\cdots\mathbf{w}_{j_m}(1-j_m/2),\]
then $m=n$ and $i_k=j_k$ for each $1\le k\le n$.  Define $\varphi:\mathcal{B}\to\{0,-1\}^*$ by $\varphi(10)=\overline{01}$ and for each $\mathbf d\in\mathcal{B}$ of the form 
\[\mathbf d=1\mathbf{w}_{i_1}\mathbf{w}_{i_2}\cdots\mathbf{w}_{i_n}(1-i_n/2),\]
by 
\[\varphi(\mathbf d):=\overline{0\mathbf{w}_{2-i_1}\mathbf{w}_{2-i_2}\cdots\mathbf{w}_{2-i_n}(i_n/2)},\]
where $\overline{\mathbf w}:=-\mathbf w$ for each $\mathbf w\in\{0,\pm 1\}^*$.

Let $\sigma:\{0,\pm 1\}^\mathbb{N}\to \{0,\pm 1\}^\mathbb{N}$ denote the left shift defined by $\sigma((w_j)_{j\ge 1})=(w_{j+1})_{j\ge 1}$ for each $(w_j)_{j\ge1}\in\{0,\pm 1\}^\mathbb{N}$; as with the lexicographical ordering, $\sigma$ is also defined on the set $\{0,\pm 1\}^*$ of finite words by sending $\mathbf w\in\{0,\pm 1\}^*$ to $\mathbf w0^\infty$.  We remark that for each $T\in\{S_\alpha,T_\alpha,B\}$, the left shift of the $T$-expansion of $x$ equals the $T$-expansion of $T(x)$.

\begin{defn}\label{property_M}
A word $\mathbf d\in\mathcal{B}$ satisfies \textit{Property $M$} if, for each $j\ge 0$, both $\sigma^j(\mathbf d)\preceq \mathbf d$ and $\sigma^j(\overline{\varphi(\mathbf d)})\preceq \mathbf d$.  Denote by $\mathcal{M}\subset\mathcal{B}$ the collection of all words $\mathbf d$ satisfying Property $M$.  We call $10$ and $1001$ the \textit{exceptional} words in $\mathcal{M}$ and denote by $\mathcal{M}_U:=\mathcal{M}\backslash\{10,1001\}$ the collection of \textit{unexceptional} words in $\mathcal{M}$.
\end{defn}

\begin{eg}
Let $\mathbf d\in\mathcal{B}$ be as in Example \ref{adm_block_form_eg}.  Then
\[\varphi(\mathbf d)=\overline{0\mathbf w_0\mathbf w_{2}\mathbf w_{1}\mathbf w_20}=\overline{00001001010},\]
and since both $\sigma^j(\mathbf d)\preceq\mathbf d$ and $\sigma^j(\overline{\varphi(\mathbf d)})\preceq\mathbf d$ for all $j\ge 0$, we have $\mathbf d\in\mathcal{M}$.
\end{eg}

We shall see that Property $M$ classifies matching words of the maps $S_\alpha$.  To show that $\mathcal{M}$ contains all matching words we need the following observation, which is not novel, but for which a proof is included for completeness: 

\begin{lem}\label{lexicographical_ordering}
Fix $\alpha\in[1,\beta]$ and $x,y\in[-1,1]$.  Then $x<y$ if and only if $(s_{\alpha,j}(x))_{j\ge 1}\prec(s_{\alpha,j}(y))_{j\ge 1}$.  Similarly, for $x,y\in[0,1]$, $x<y$ if and only if $(b_j(x))_{j\ge 1}\prec (b_j(y))_{j\ge 1}$.
\end{lem}
\begin{proof}
Suppose $x,y\in[-1,1]$ with $x<y$, and let $n:=\min_{j\ge 1}\{s_{\alpha,j}(x)\neq s_{\alpha,j}(y)\}$.  We first claim for each $0\le j< n$ that $S_\alpha^j(x)<S_\alpha^j(y)$.  This is true by assumption for $j=0$.  If $n=1$, we're finished.  Assume $n>1$ and that the claim holds for some $j=k-1$ with $0\le k-1<n-1$.  Since $s_{\alpha,k}(x)=s_{\alpha,k}(y)$, we have that $S_\alpha$ restricts to a linear function with positive slope on an interval containing $S_\alpha^{k-1}(x)$ and $S_\alpha^{k-1}(y)$.  But $S_\alpha^{k-1}(x)<S_\alpha^{k-1}(y)$ by assumption, so also $S_\alpha^{k}(x)<S_\alpha^{k}(y)$ and the claim holds.  Since $s_{\alpha,n}(x)\neq s_{\alpha,n}(y)$ and $S_\alpha^{n-1}(x)<S_\alpha^{n-1}(y)$, it must be true that $s_{\alpha,n}(x)<s_{\alpha,n}(y)$ and hence $(s_{\alpha,j}(x))_{j\ge 1}\prec(s_{\alpha,j}(y))_{j\ge 1}$.

Now suppose $x\ge y$.  If equality holds, then by uniqueness of $S_\alpha$-expansions, $(s_{\alpha,j}(x))_{j\ge 1}=(s_{\alpha,j}(y))_{j\ge 1}$.  If the inequality is strict, the argument above applies with $x$ and $y$ interchanged.

The proof of the second statement is identical, \textit{mutatis mutandis}.
\end{proof}

\begin{prop}\label{matching_implies_property_M}
Suppose for some $\alpha\in[1,\beta]$
that $S_\alpha$ has matching with index $m:=m(\alpha)$, and let $\mathbf d:=d_1\cdots d_m$ denote the corresponding matching word.  Then $\mathbf d\in\mathcal M$, and $\mathbf e:=\varphi(\mathbf d)$ agrees with the first $m$ digits $e_1\cdots e_m$ of the $S_\alpha$-expansion of $1-\alpha$.
\end{prop}
\begin{proof}
From the cases of \S\ref{Matching almost everywhere}, the result holds for $\alpha\notin(1,1+1/\beta^3)$; in particular, $\alpha\in (1+1/\beta^2,\beta]$ and $\alpha\in(1+1/\beta^3,1+1/\beta^2)$ correspond to the exceptional words $\mathbf d=10$ and $\mathbf d=1001$, respectively, in $\mathcal{M}$, and $\varphi(10)=\overline{01},\ \varphi(1001)=\overline{0010}$.  Now assume $\alpha\in(1,1+1/\beta^3)$.  Note that $d_1=1,\ e_1=0$, and 
\[S_\alpha(1)-S_\alpha(1-\alpha)=(\beta-\alpha)-\beta(1-\alpha)=\alpha/\beta.\]
Recall from Equation (\ref{matching_final_digits}) and the discussion preceding it that
\[\begin{pmatrix}d_{m-2}d_{m-1}d_m\\ \overline{e_{m-2}e_{m-1}e_m}\end{pmatrix}\in\left\{\begin{pmatrix}001\\ \overline{010}\end{pmatrix},\begin{pmatrix}010\\ \overline{001}\end{pmatrix}\right\}=\left\{\begin{pmatrix}\mathbf w_01\\ \overline{\mathbf w_20}\end{pmatrix},\begin{pmatrix}\mathbf w_20\\ \overline{\mathbf w_01}\end{pmatrix}\right\},\]
and $S_\alpha^{m-3}(1)-S_\alpha^{m-3}(1-\alpha)=\alpha/\beta$.  The remaining digits
\[\begin{pmatrix}d_2d_3\cdots d_{m-3}\\ \overline{e_2e_3\cdots e_{m-3}}\end{pmatrix}\]
are thus determined by edge labels of cycles in the graph of Figure \ref{differences_graph} beginning and ending at vertex $\alpha/\beta$.  There are three possible cycles, whose edge labels give
\[\begin{pmatrix}d_{j}d_{j+1}\\ \overline{e_{j}e_{j+1}}\end{pmatrix}=\begin{pmatrix}01\\ \overline{00}\end{pmatrix}=\begin{pmatrix}\mathbf w_2\\ \overline{\mathbf w_0}\end{pmatrix},\
\begin{pmatrix}d_{j}d_{j+1}\\ \overline{e_{j}e_{j+1}}\end{pmatrix}=\begin{pmatrix}00\\ \overline{01}\end{pmatrix}=\begin{pmatrix}\mathbf w_0\\ \overline{\mathbf w_2}\end{pmatrix},\ \text{and}\
\begin{pmatrix}d_jd_{j+1}d_{j+2}\\ \overline{e_je_{j+1}e_{j+2}}\end{pmatrix}=\begin{pmatrix}001\\ \overline{001}\end{pmatrix}=\begin{pmatrix}\mathbf w_1\\ \overline{\mathbf w_1}\end{pmatrix}.\]
It follows that $\mathbf d=1\mathbf{w}_{i_1}\mathbf{w}_{i_2}\cdots\mathbf{w}_{i_n}(1-i_n/2)$ and $e_1\cdots e_m=\overline{0\mathbf{w}_{2-i_1}\mathbf{w}_{2-i_2}\cdots\mathbf{w}_{2-i_n}(i_n/2)}$ for some $i_1,\dots,i_n\in\{0,1,2\},\ n\ge 1$ and $i_n\neq 1$.  Moreover, note from case (v) of \S\ref{Matching almost everywhere} that $d_1d_2d_3d_4=1010$, so $i_1=2$.  Thus $\mathbf d\in\mathcal{B}$ and $\mathbf e=e_1\cdots e_m=\varphi(\mathbf d)$.  From Lemma \ref{lexicographical_ordering}, the facts that $S_\alpha^j(1),\ S_\alpha^j(1-\alpha)\in[-1,1]$ for each $j\ge 0$ imply that $\sigma^j(\mathbf d),\sigma^j(\overline{\mathbf e})\preceq\mathbf d$ for each $j\ge 0$.  Thus $\mathbf d\in\mathcal{M}$.
\end{proof}

The previous result states that every matching word belongs to $\mathcal{M}$.  Before proving the converse (Propositions \ref{I_delta_nonempty_etc} and \ref{I_d_determines_expansions}), we define and investigate properties of the \textit{valuation function} $v:\mathcal{S}\to\mathbb{R}$ given by the (absolutely) convergent series
\[v((w_j)_{j\ge 1}):=\sum_{j\ge 1}w_j/\beta^j,\]
where $\mathcal{S}\subset \mathbb{Z}^\mathbb{N}$ consists of all sequences $(w_j)_{j\ge 1}$ whose entries are bounded above and below.  The valuation function is also defined on the set $\mathcal{S}^*\subset \mathcal{S}$ of finite words by considering the corresponding finite sum and setting $v(\varepsilon)=0$ for the empty word $\varepsilon$.  It is not difficult to check for finite words $\mathbf w,\mathbf w'\in\{0,\pm 1\}^*$ with no consecutive nonzero digits that $\mathbf w\prec \mathbf w'$ if and only if $v(\mathbf w)<v(\mathbf w')$. 

\begin{lem}\label{valuation_of_blocks}
If $\mathbf w:=w_1w_2\cdots w_k\in\{0,1,2\}^*$ is $\varepsilon$ (in which case we set $k=0$) or consists solely of blocks of $01$'s and $002$'s, then
\[v(\mathbf w)=1/\beta-1/\beta^{k+1}.\]
\end{lem}
\begin{proof}
The case that $\mathbf w=\varepsilon$ is trivial, so suppose $\mathbf w\neq\varepsilon$.  One easily verifies that
\[v((01)^3)=v((002)^2)\ \ \ \ \ \text{and}\ \ \ \ \ v(01002)=v(00201).\]
These observations, together with the fact that for each $1\le j\le k$,
\[v(\mathbf w)=v(w_1\cdots w_j)+(1/\beta^j)v(w_{j+1}\cdots w_k),\]
imply that
\[v(\mathbf w)=\begin{cases}
v((002)^{k/3}), &k\equiv 0\ (\text{mod}\ 3)\\
v((002)^{(k-4)/3}(01)^2), &k\equiv 1\ (\text{mod}\ 3)\\
v((002)^{(k-2)/3}01), &k\equiv 2\ (\text{mod}\ 3)\\
\end{cases}.\]
Notice that for any $j\ge 1$, 
\begin{align*}
v((002)^j)&=2\sum_{i=1}^j(1/\beta^3)^i\\
&=2\cdot\frac{1/\beta^3-1/\beta^{3j+3}}{1-1/\beta^3}\\
&=2\cdot\frac{1-1/\beta^{3j}}{\beta^3-1}\\
&=2\cdot\frac{1-1/\beta^{3j}}{2\beta}\\
&=1/\beta-1/\beta^{3j+1}.
\end{align*}
If $k\equiv0\ (\text{mod}\ 3)$, setting $j=k/3$ gives the result.  If $k\equiv1\ (\text{mod}\ 3)$, we compute
\begin{align*}
v(\mathbf w)&=v((002)^{(k-4)/3}(01)^2)\\
&=v((002)^{(k-4)/3})+(1/\beta^{k-4})v((01)^2)\\
&=1/\beta-1/\beta^{k-3}+(1/\beta^{k-4})(1/\beta^2+1/\beta^4)\\
&=1/\beta-1/\beta^{k-3}+1/\beta^{k-2}+1/\beta^{k}\\
&=1/\beta-1/\beta^{k+1}.
\end{align*}
Similarly, if $k\equiv2\ (\text{mod}\ 3)$,
\begin{align*}
v(\mathbf w)&=v((002)^{(k-2)/3}01)\\
&=v((002)^{(k-2)/3})+(1/\beta^{k-2})v(01)\\
&=1/\beta-1/\beta^{k-1}+1/\beta^{k}\\
&=1/\beta-1/\beta^{k+1}.
\end{align*}
\end{proof}

For equal-length words $\mathbf x,\mathbf y\in\{0,\pm 1\}^*$, define $\mathbf x+\mathbf y,\mathbf x-\mathbf y\in\{0,\pm 1,\pm 2\}^*$ where addition and subtraction, respectively, are performed entry-wise.  Note that
\[\mathbf w_2-\overline{\mathbf w_0}=01,\ \ \ \mathbf w_0-\overline{\mathbf w_2}=01,\ \ \ \text{and}\ \ \ \mathbf w_1-\overline{\mathbf w_1}=002.\]
Suppose $\mathbf d$ satisfies Property $M$ with $m:=\text{len}(\mathbf d)$.  Since $\mathbf d$ is in admissible block form, the definition of $\mathbf e:=\varphi(\mathbf d)$ implies that $\mathbf d-\mathbf e=1\mathbf w 1$ for some word $\mathbf w$ consisting solely of blocks of $01$'s and $002$'s or $\mathbf w=\varepsilon$.  Using Lemma \ref{valuation_of_blocks}, we compute
\[v(\mathbf d-\mathbf e)=v(1\mathbf w1)=1/\beta+(1/\beta)(1/\beta-1/\beta^{m-1})+1/\beta^m=1.\]
This proves the following:

\begin{prop}\label{property_M_difference}
If $\mathbf d\in\mathcal{M}$ and $\mathbf e:=\varphi(\mathbf d)$, then
\[v(\mathbf d)-v(\mathbf e)=v(\mathbf d-\mathbf e)=1.\]
\end{prop}

For $\mathbf d=10$, set $I_{\mathbf d}=(\alpha_{\mathbf d}^-,\alpha_{\mathbf d}^+]:=(1+1/\beta^2,\beta]$, and for all other $\mathbf d=d_1\cdots d_m\in\mathcal{M}$, define
\begin{equation}\label{matching_interval}
I_{\mathbf d}=(\alpha_{\mathbf d}^-,\alpha_{\mathbf d}^+):=\left(\frac{\beta^m+\beta^{d_m}}{\beta^mv(\mathbf d)+\beta^{d_m}},\frac{\beta^m-\beta^{1-d_m}}{\beta^mv(\mathbf d)-\beta^{1-d_m}}\right).
\end{equation}

\begin{prop}\label{I_delta_nonempty_etc}
For each $\mathbf d\in\mathcal{M}$, $I_{\mathbf d}$ is a nonempty subinterval of $(1,\beta]$.
\end{prop}
\begin{proof}
The result is true for $\mathbf d=10$, so assume $\mathbf d\neq 10$.  We first show that $I_{\mathbf d}\neq\varnothing$, i.e. that
\[\frac{\beta^m+\beta^{d_m}}{\beta^mv(\mathbf d)+\beta^{d_m}}<\frac{\beta^m-\beta^{1-d_m}}{\beta^mv(\mathbf d)-\beta^{1-d_m}},\]
or
\[(\beta^m+\beta^{d_m})(\beta^mv(\mathbf d)-\beta^{1-d_m})<(\beta^m-\beta^{1-d_m})(\beta^mv(\mathbf d)+\beta^{d_m}).\]
Distributing and cancelling terms gives that this is equivalent to
\[\beta^{m+d_m}v(\mathbf d)-\beta^{m+1-d_m}<\beta^{m+d_m}-\beta^{m+1-d_m}v(\mathbf d),\]
or $v(\mathbf d)<1$.  Since $\mathbf d$ has no consecutive $1$'s, one finds that $v(\mathbf d)<v((10)^\infty)=1$ (see also Lemma 1 of \cite{parry_60}). 

Next we show that $I_{\mathbf d}\subset(1,\beta]$.  The left endpoint of $I_{\mathbf d}$ is greater than $1$ again since $v(\mathbf d)<1$.  It remains to show that 
\[\frac{\beta^m-\beta^{1-d_m}}{\beta^mv(\mathbf d)-\beta^{1-d_m}}\le \beta.\]
Recall that $d_1=1$, and if $d_m=0$, then $d_{m-1}=1$; thus $v(\mathbf d)\ge 1/\beta+\beta^{1-d_m}/\beta^m$, and
\[\beta^{m+1}v(\mathbf d)-\beta^{2-d_m}\ge \beta^{m+1}(1/\beta+\beta^{1-d_m}/\beta^m)-\beta^{2-d_m}>\beta^m-\beta^{1-d_m}.\]
Dividing both sides by $\beta^mv(\mathbf d)-\beta^{1-d_m}$ gives the desired inequality.
\end{proof}

For each $\mathbf u\in\{0,1\}^*$, let $\Delta(\mathbf u)$ denote the cylinder of points $x\in[0,1]$ for which the $\beta$-expansion of $x$ begins with $\mathbf u$.  One finds for each $\mathbf u=u_1\cdots u_n$ with $u_ju_{j+1}=0,\ 1\le j<n$, that
\begin{equation}\label{cylinder_interval}
\Delta(\mathbf u)=\begin{cases}
[v(\mathbf u),v(\mathbf u)+1/\beta^n), & u_n=0\\
[v(\mathbf u),v(\mathbf u)+1/\beta^{n+1}), & u_n=1\\
\end{cases}.
\end{equation}

The following lemma is needed in Proposition \ref{I_d_determines_expansions} below.  

\begin{lem}\label{reciprocal_endpt_lem}
Let $\mathbf d\in\mathcal{M}_U$.  Then $B^j(1/\alpha_{\mathbf d}^-)\le 1/\alpha_{\mathbf d}^-$ and $B^j(1-1/\alpha_{\mathbf d}^+)\le 1/\alpha_{\mathbf d}^+$ for all $j>0$.
\end{lem}
\begin{proof}
This is a corollary of two technical results (Lemmas \ref{reciprocal_expansions} and \ref{lex_ineq_lemma}), whose statements and proofs are provided in the appendix.
\end{proof}

The next result---together with Proposition \ref{I_delta_nonempty_etc}---states that every word $\mathbf d\in\mathcal{M}$ is in fact a matching word, thus completing our classification of matching words as the set $\mathcal{M}$.  Moreover, it states that the interval $I_{\mathbf d}$ is contained in a matching interval corresponding to the matching word $\mathbf d$.

\begin{prop}\label{I_d_determines_expansions}
For any $\mathbf d\in\mathcal{M}$ and $\alpha\in I_{\mathbf d}$, the $S_\alpha$-expansions of $1$ and $1-\alpha$ begin with $\mathbf d$ and $\varphi(\mathbf d)$, respectively.  Moreover, $S_\alpha$ has matching with matching index $m(\alpha)=\text{len}(\mathbf d)$.
\end{prop}
\begin{proof}
The result is shown for exceptional words $\mathbf d\in\{10,1001\}$ in \S\ref{Matching almost everywhere}, so assume $\mathbf d\in\mathcal{M}_U$.  Suppose the first statement holds.  That $S_\alpha$ has matching with index $m(\alpha)=\text{len}(\mathbf d)$ is implied by the final three digits of $\mathbf d$ and $\mathbf e$ (see the discussion surrounding Equation (\ref{matching_final_digits})), so we need only prove the first statement.
Let $\alpha\in I_{\mathbf d}$, and write $\mathbf d=d_1\cdots d_m$ and $\mathbf e:=\varphi(\mathbf d)=e_1\cdots e_m$.  We must show that 
\[d_{\alpha,1}\cdots d_{\alpha,m}=d_1\cdots d_m\]
and
\[e_{\alpha,1}\cdots e_{\alpha,m}=e_1\cdots e_m.\]
Assume that 
\[\begin{pmatrix} d_{m-2}d_{m-1}d_m\\ e_{m-2}e_{m-1}e_m\end{pmatrix}=\begin{pmatrix}001\\\overline{010}\end{pmatrix},\]
and set $\alpha_0:=1/v(\mathbf d)$ (the case that $d_m=0$ is similar).
Proposition \ref{I_delta_nonempty_etc} together with the fact that $v(\mathbf d)<1$ imply $\alpha^-<\alpha_0<\alpha^+$, where, for ease of notation, $\alpha^\pm:=\alpha_{\mathbf d}^\pm$.  We claim that it suffices to show the following:
\begin{enumerate}
\item[(i)] if $\alpha\in (\alpha^-,\alpha_0),$ then $\ell_\alpha(1)>m-1,\ \ell_\alpha(1-\alpha)=m-2$,
\[b_1(1/\alpha)\cdots b_m(1/\alpha)=d_1\cdots d_{m},\]
and
\[b_1(1-1/\alpha)\cdots b_{m-2}(1-1/\alpha)=\overline{e_1\cdots e_{m-2}};\]

\item[(ii)] if $\alpha\in (\alpha_0,\alpha^+),$ then $\ell_\alpha(1)=m-1,\ \ell_\alpha(1-\alpha)>m-2$,
\[b_1(1/\alpha)\cdots b_{m-1}(1/\alpha)=d_1\cdots d_{m-1},\]
and
\[b_1(1-1/\alpha)\cdots b_{m}(1-1/\alpha)=\overline{e_1\cdots e_m};\]
and

\item[(iii)] if $\alpha=\alpha_0$, then $\ell_\alpha(1)=m-1,\ \ell_\alpha(1-\alpha)=m-2$,
\[b_1(1/\alpha)\cdots b_{m-1}(1/\alpha)=d_1\cdots d_{m-1},\]
\[b_1(1-1/\alpha)\cdots b_{m-2}(1-1/\alpha)=\overline{e_1\cdots e_{m-2}},\]
and $B^{m-1}(1/\alpha)=B^{m-2}(1-1/\alpha)=1/\beta$.
\end{enumerate}
Indeed, suppose (i) holds.  Lemma \ref{S_alpha_and_B} implies 
\[d_{\alpha,1}\cdots d_{\alpha,m}=d_1\cdots d_m\]
and 
\[e_{\alpha,1}\cdots e_{\alpha,m-2}=e_1\cdots e_{m-2}.\]
Since $\ell_\alpha(1-\alpha)=m-2$, Corollary \ref{S_alpha_matching_hole} gives $S_\alpha^{m-2}(1-\alpha)\in[-\alpha/\beta,-1/\beta)$, so $e_{\alpha,m-1}=-1$ and $e_{\alpha,m}=0$.  In case (ii), Lemma \ref{S_alpha_and_B} again gives
\[d_{\alpha,1}\cdots d_{\alpha,m-1}=d_1\cdots d_{m-1}\]
and 
\[e_{\alpha,1}\cdots e_{\alpha,m-1}=e_1\cdots e_{m-1}.\]
Moreover, $\ell_\alpha(1)=m-1$ implies $S_\alpha^{m-1}(1)\in(1/\beta,\alpha,\beta]$ and hence $d_{\alpha,m}=1$.  Since $e_{\alpha,{m-1}}=e_{m-1}=-1$, it follows that $e_{\alpha,{m}}=0=e_{m}$.  In (iii), we have 
\[d_{\alpha,1}\cdots d_{\alpha,m-1}=d_1\cdots d_{m-1}\] 
and
\[e_{\alpha,1}\cdots e_{\alpha,m-2}=e_1\cdots e_{m-2}.\]
Moreover, Lemma \ref{S_alpha_and_B} gives $S_\alpha^{m-1}(1)=-S_\alpha^{m-2}(1-\alpha)=\alpha/\beta$, so $d_{\alpha,m}=\overline{e_{\alpha,m-1}}=1$ and $e_{\alpha,m}=0$.

By Corollary \ref{B_matching_hole}, (i), (ii) and (iii) are implied by showing:
\begin{enumerate}
\item[(a)] $1/\overline{I_{\mathbf d}}\subsetneq\Delta(d_1\cdots d_{m-1})$ and $1-1/\overline{I_{\mathbf d}}\subsetneq\Delta(\overline{e_1\cdots e_{m-2}})$;
\item[(b)] $B^j(1/\alpha)\notin(1/\beta\alpha,1/\beta]$ for each $0\le j<m-1$, and $B^j(1-1/\alpha)\notin (1/\beta\alpha,1/\beta]$ for each $0\le j<m-2$;
\item[(c)] if $\alpha\in(\alpha^-,\alpha_0)$, then $B^{m-1}(1/\alpha)>1/\beta$ and $B^{m-2}(1-1/\alpha)\in(1/\beta\alpha,1/\beta]$;
\item[(d)] if $\alpha\in(\alpha_0,\alpha^+)$, then $B^{m-1}(1/\alpha)\in(1/\beta\alpha,1/\beta]$ and $B^{m-2}(1-1/\alpha)>1/\beta$; and
\item[(e)] if $\alpha=\alpha_0$, then $B^{m-1}(1/\alpha)=B^{m-2}(1-1/\alpha)=1/\beta$.
\end{enumerate}
We prove each of (a), (b), (c), (d) and (e):

\begin{enumerate}
\item[(a)]
The first inclusion is equivalent to
\begin{equation}\label{d_cyl_inequalities}
v(d_1\cdots d_{m-1})<1/\alpha^+<1/\alpha^-<v(d_1\cdots d_{m-1})+1/\beta^{m-1}.
\end{equation}
Note that $v(d_1\cdots d_{m-1})<1/\alpha^+$ if and only if
\[v(\mathbf d)-1/\beta^m<\frac{\beta^mv(\mathbf d)-1}{\beta^m-1}.\]
Multiplying both sides by $\beta^m-1$, cancelling and rearranging terms, this is equivalent to $v(\mathbf d)> 1/\beta^m$.  This latter inequality holds since $v(\mathbf d)\ge v(d_1)=1/\beta$ and $m>1$.  Next, $1/\alpha^-<v(d_1\cdots d_{m-1})+1/\beta^{m-1}$ if and only if
\[\frac{\beta^mv(\mathbf d)+\beta}{\beta^m+\beta}<v(\mathbf d)-1/\beta^{m}+1/\beta^{m-1}.\]
Using the fact that $1/\beta^{m-1}=1/\beta^m+1/\beta^{m+1}$ and multiplying both sides by $\beta^m+\beta$, this is equivalent to 
\[\beta^mv(\mathbf d)+\beta<(\beta^m+\beta)(v(\mathbf d)+1/\beta^{m+1}),\]
or
\[\beta^mv(\mathbf d)+\beta<\beta^mv(\mathbf d)+1/\beta+\beta v(\mathbf d)+1/\beta^m.\]
Simplifying, this is equivalent to showing $1<\beta v(\mathbf d)+1/\beta^m$, which again holds since $v(\mathbf d)\ge 1/\beta$. 
Thus $1/\overline{I_{\mathbf d}}\subsetneq\Delta(d_1\cdots d_{m-1})$.

The second inclusion is equivalent to
\[v(\overline{e_1\cdots e_{m-2}})< 1-1/\alpha^-<1-1/\alpha^+<v(\overline{e_1\cdots e_{m-2}})+1/\beta^{m-2}.\]
Now $v(\overline{e_1\cdots e_{m-2}})<1-1/\alpha^-$ if and only if $1/\alpha^-<1-(v(\overline{\mathbf e})-1/\beta^{m-1})$.  By Proposition \ref{property_M_difference}, the fact that $v(\overline{\mathbf e})=-v(\mathbf e)$ and (\ref{d_cyl_inequalities}),
\[1-(v(\overline{\mathbf e})-1/\beta^{m-1})=v(\mathbf d)+1/\beta^{m-1}>v(d_1\cdots d_{m-1})+1/\beta^{m-1}>1/\alpha^-.\]
Lastly, $1-1/\alpha^+<v(\overline{e_1\cdots e_{m-2}})+1/\beta^{m-2}$ if and only if  $1-1/\alpha^+<v(\overline{\mathbf e})-1/\beta^{m-1}+1/\beta^{m-2}$, or $v(\mathbf d)<1/\alpha^++1/\beta^m$.  From (\ref{d_cyl_inequalities}), we find
\[v(\mathbf d)-1/\beta^m=v(d_1\cdots d_{m-1})<1/\alpha^+.\]
Thus $1-1/\overline{I_{\mathbf d}}\subsetneq\Delta(\overline{e_1\cdots e_{m-2}})$.

\item[(b)]  Fix $0\le j<m-1$.  If $d_{j+1}=1$, then part (a) and Lemma \ref{lexicographical_ordering} imply that $B^j(1/\alpha)>B^j(1/\alpha^+)\ge 1/\beta$.  Now suppose $d_{j+1}=0$.  By (a), $B^j(1/\alpha^-)\in(1/\beta\alpha^-,1/\beta]$ if and only if $B^{j+1}(1/\alpha^-)\in(1/\alpha^-,1]$.  Lemma \ref{reciprocal_endpt_lem} thus implies $B^j(1/\alpha^-)\notin(1/\beta\alpha^-,1/\beta]$.  By Equation (\ref{B^k_eqn}), it also holds for each $x\in\Delta(d_1\cdots d_{m-1})$ that $B^j(x)\notin (x/\beta,1/\beta]$ if and only if
\[\beta^j(x-v(d_1\cdots d_j))\le x/\beta,\]
or
\[x\le \frac{\beta^jv(d_1\cdots d_j)}{\beta^j-1/\beta}.\]
Since $1/\alpha,1/\alpha^-\in\Delta(d_1\cdots d_{m-1})$ and $B^j(1/\alpha^-)\notin (1/\beta\alpha^-,1/\beta]$, we have
\[1/\alpha<1/\alpha^-\le \frac{\beta^jv(d_1\cdots d_j)}{\beta^j-1/\beta},\]
which implies $B^j(1/\alpha)\notin(1/\beta\alpha,1/\beta]$.  Thus $B^j(1/\alpha)\notin (1/\beta\alpha,1/\beta]$ for each $0\le j<m-1$.

The proof that $B^j(1-1/\alpha)\notin (1/\beta\alpha,1/\beta]$ for each $0\le j<m-2$ is similar.

\item[(c)]  Suppose $\alpha\in(\alpha^-,\alpha_0)$.  From Equation (\ref{B^k_eqn}) and part (a), we have for each $x\in1/\overline{I_{\mathbf d}}$ that
\begin{flalign}\label{B^{m-1}(x)}
B^{m-1}(x)&=\beta^{m-1}(x-v(d_1\cdots d_{m-1}))\\ \nonumber
&=\beta^{m-1}(x-(v(\mathbf d)-1/\beta^{m}))
\end{flalign}
Since $1/\alpha>1/\alpha_0=v(\mathbf d)$, we have $B^{m-1}(1/\alpha)>1/\beta$.  Also from Equation (\ref{B^k_eqn}), part (a) and Proposition \ref{property_M_difference}, for each $x\in 1/\overline{I_{\mathbf d}}$,
\begin{flalign}\label{B^{m-2}(1-x)}
B^{m-2}(1-x)&=\beta^{m-2}(1-x-v(\overline{e_1\cdots e_{m-2}}))\\ \nonumber
&=\beta^{m-2}(1-x+v(\mathbf e)+1/\beta^{m-1})\\ \nonumber
&=\beta^{m-2}(-x+v(\mathbf d)+1/\beta^{m-1})\\ \nonumber
&=-\beta^{m-2}x+\beta^{m-2}v(\mathbf d)+1/\beta.
\end{flalign}
Hence 
\[B^{m-2}(1-1/\alpha)<B^{m-2}(1-1/\alpha_0)=1/\beta,\]
and $B^{m-2}(1-1/\alpha)>1/\beta\alpha$ if and only if
\[\frac{\beta^{m-2}v(\mathbf d)+1/\beta}{\beta^{m-2}+1/\beta}>1/\alpha.\]
But the left hand side equals $1/\alpha^-$, so the inequality holds.  

\item[(d)] Suppose $\alpha\in(\alpha_0,\alpha^+)$.  From Equation (\ref{B^{m-1}(x)}), $1/\alpha<1/\alpha_0=v(\mathbf d)$ implies $B^{m-1}(1/\alpha)<1/\beta$.  Moreover, $B^{m-1}(1/\alpha)>1/\beta\alpha$ if and only if 
\[1/\alpha>\frac{\beta^{m-1}v(\mathbf d)-1/\beta}{\beta^{m-1}-1/\beta}.\]
The right-hand side equals $1/\alpha^+$, and $\alpha<\alpha^+$ by assumption.  We also find from Equation (\ref{B^{m-2}(1-x)}) that
\[B^{m-2}(1-1/\alpha)=\beta^{m-2}(v(\mathbf d)-1/\alpha)+1/\beta>1/\beta\]
since $1/\alpha<1/\alpha_0=v(\mathbf d)$.

\item[(e)]  This again follows from Equations (\ref{B^{m-1}(x)}) and (\ref{B^{m-2}(1-x)}), setting $x=1/\alpha_0=v(\mathbf d)$.
\end{enumerate}
\end{proof}

The following proposition states that the interval $I_{\mathbf d}$ contains the matching intervals corresponding to the matching word $\mathbf d$; together with Proposition \ref{I_d_determines_expansions}, this characterises matching intervals as the collection $\{I_{\mathbf d}\}_{\mathbf d\in\mathcal{M}}$.

\begin{prop}\label{matching_intervals_thm}
If $S_\alpha$ has matching with $m(\alpha)=m$, then $\alpha\in I_{\mathbf d}$, where $\mathbf d=d_1\cdots d_m$ is beginning of the $S_\alpha$-expansion of $1$.
\end{prop}
\begin{proof}
By Proposition \ref{matching_implies_property_M}, $\mathbf d\in\mathcal{M}$, so $I_{\mathbf d}$ is defined.  The result holds for $m\le 2$ by the cases in \S\ref{Matching almost everywhere}, so assume $m>2$ and let $\mathbf e=e_1\cdots e_m$ denote the beginning of the $S_\alpha$-expansion of $1-\alpha$.  Recall from Equation (\ref{matching_final_digits}) that
\[\begin{pmatrix}d_{m-2}d_{m-1}d_m\\ e_{m-2}e_{m-1}e_m\end{pmatrix}\in\left\{\begin{pmatrix}010\\\overline{001}\end{pmatrix},\begin{pmatrix}001\\\overline{010}\end{pmatrix}\right\}.\]
Assume $d_m=0$ (the other case is similar).  Lemma \ref{ell(alpha)}, Corollary \ref{S_alpha_matching_hole} and the final digits of $\mathbf d$ and $\mathbf e$ imply that either 
\[\text{(i)}\ \ S_\alpha^{m-2}(1)\in(1/\beta,\alpha/\beta]\ \ \ \ \text{or}\ \ \ \ \text{(ii)}\ \ S_\alpha^{m-1}(1-\alpha)\in[-\alpha/\beta,-1/\beta).\]
It suffices to show that both (i) and (ii) imply 
\[\alpha\in I_{\mathbf d}=\left(\frac{\beta^m+1}{\beta^mv(\mathbf d)+1},\frac{\beta^m-\beta}{\beta^mv(\mathbf d)-\beta}\right).\]

\begin{enumerate}
\item[(i)]  Equation (\ref{S_alpha^k_eqn}) gives 
\[S_\alpha^{m-2}(1)=\beta^{m-2}(1-\alpha v(d_1\cdots d_{m-2}))\in(1/\beta,\alpha/\beta].\]
Note that $v(d_1\cdots d_{m-2})=v(\mathbf{d})-1/\beta^{m-1}$, so
\[1-\alpha(v(\mathbf{d})-1/\beta^{m-1})\in(1/\beta^{m-1},\alpha/\beta^{m-1}].\] 
Now
\[1-\alpha(v(\mathbf{d})-1/\beta^{m-1})>1/\beta^{m-1}\]
implies
\[\alpha<\frac{1-1/\beta^{m-1}}{v(\mathbf d)-1/\beta^{m-1}}=\frac{\beta^m-\beta}{\beta^mv(\mathbf d)-\beta}.\]
Moreover,
\[1-\alpha(v(\mathbf{d})-1/\beta^{m-1})\le\alpha/\beta^{m-1}\]
gives $1\le \alpha v(\mathbf{d})$.  Thus we have
\[\alpha\in\left[\frac1{v(\mathbf d)},\frac{\beta^m-\beta}{\beta^mv(\mathbf d)-\beta}\right),\]
and it suffices to show that 
\[\frac{\beta^m+1}{\beta^mv(\mathbf d)+1}<\frac1{v(\mathbf d)}.\]
But this is true since $v(\mathbf d)<v((10)^\infty)=1$. 

\item[(ii)]  Again from Equation (\ref{S_alpha^k_eqn}),
\[S_\alpha^{m-1}(1-\alpha)=\beta^{m-1}(1-\alpha(1+v(e_1\cdots e_{m-1})))\in[-\alpha/\beta,-1/\beta).\]
The assumption that $e_m=-1$ together with Proposition \ref{property_M_difference} give 
\[1+v(e_1\cdots e_{m-1})=1+v(\mathbf e)+1/\beta^m=v(\mathbf d)+1/\beta^m,\]
so
\[1-\alpha(v(\mathbf{d})+1/\beta^m)\in[-\alpha/\beta^m,-1/\beta^m).\]
Now
\[1-\alpha(v(\mathbf{d})+1/\beta^m)\ge -\alpha/\beta^m\]
implies $1\ge\alpha v(\mathbf d)$.  Furthermore, 
\[1-\alpha(v(\mathbf{d})+1/\beta^m)<-1/\beta^m\]
gives
\[\alpha>\frac{1+1/\beta^m}{v(\mathbf d)+1/\beta^m}=\frac{\beta^m+1}{\beta^mv(\mathbf d)+1}.\]
Hence 
\[\alpha\in\left(\frac{\beta^m+1}{\beta^mv(\mathbf d)+1},\frac1{v(\mathbf d)}\right],\]
and it suffices to show 
\[\frac1{v(\mathbf d)}<\frac{\beta^m-\beta}{\beta^mv(\mathbf d)-\beta}.\]
This is true again since $v(\mathbf d)<1$.
\end{enumerate}
\end{proof}

The implications of Propositions \ref{matching_implies_property_M}, \ref{I_delta_nonempty_etc}, \ref{I_d_determines_expansions} and \ref{matching_intervals_thm} are summarised in the following:

\begin{cor}\label{matching_iff_property_M}
The sets $\mathcal{M}$ and $\{I_{\mathbf d}\}_{\mathbf d\in\mathcal{M}}$ classify the matching words and intervals, respectively, of the maps $S_\alpha$.
\end{cor}

\begin{remark}\label{es_classified}
The results of this subsection also imply that $\varphi(\mathcal{M})$ classifies the first $m(\alpha)<\infty$ digits of the $S_\alpha$-expansions of $1-\alpha$ for matching parameters $\alpha\in [1,\beta]$.  Moreover, the intervals $I_{\mathbf d}$ in $\{I_{\mathbf d}\}_{\mathbf d\in\mathcal{M}}=\{I_{\varphi^{-1}(\mathbf e)}\}_{\mathbf e\in\varphi(\mathcal{M})}$ classify the maximal subintervals of matching parameters $\alpha$ for which these first $m(\alpha)$ digits coincide (and equal $\mathbf e=\varphi(\mathbf d)$).
\end{remark}

\begin{remark}
While not needed for our purposes, we briefly mention that the sets $\mathcal{M}$ (or $\varphi(\mathcal{M})$) and $\{I_{\mathbf d}\}_{\mathbf d\in\mathcal{M}}$ also give rise to classifications of the $T_\alpha$-expansions of $1$ (resp. $\beta(1-\alpha)$) before matching and the maximal intervals of parameters $\alpha$ on which these expansions coincide.  In particular, if $\mathbf d\in\mathcal{M}$ (resp. $\mathbf e:=\varphi(\mathbf d)\in\varphi(\mathcal{M})$), then the corresponding $T_\alpha$-word $\mathbf d'$ (resp. $\mathbf e'$) `forgets' each non-terminal $0$ which immediately follows a $1$ (resp. $-1$, and $\mathbf e'$ also forgets the initial $0$ of $\mathbf e$).  The matching intervals $I_{\mathbf d}$ are unchanged.  For instance, $\mathbf d=10100001$ and $\mathbf e=\varphi(\mathbf d)=\overline{00001010}$ give rise to the words $\mathbf d'=110001$ and $\mathbf e'=\overline{000110}$ for $T_\alpha$, and each of these words corresponds to the matching interval $I_{\mathbf d}=\left(\frac{\beta^8+\beta}{\beta^7+\beta^5+\beta^2},\frac{\beta^8-1}{\beta^7+\beta^5}\right)$.
\end{remark}

\subsection{Cascades of matching intervals}\label{Cascades of matching intervals}

Here it is shown that each \textit{unexceptional matching interval} $I_{\mathbf d},\ \mathbf d\in\mathcal{M}_U$, generates a whole `cascade' of unexceptional matching intervals with adjacent endpoints.  Define $\psi:\mathcal{M}_U\to\{0,1\}^*$, where for $\mathbf d=d_1\cdots d_m\in\mathcal{M}_U$ and $\mathbf e:=\varphi(\mathbf d)=e_1\cdots e_m$, 
\[\psi(\mathbf d)=\begin{cases}
\mathbf d\overline{\mathbf e}, & d_m=0\\
\mathbf d\overline{e_2\cdots e_m}, & d_m=1\\
\end{cases}.\]
Recall the definition of the matching interval $I_{\mathbf d}=(\alpha_{\mathbf d}^-,\alpha_{\mathbf d}^+)$ from (\ref{matching_interval}).  
\begin{prop}\label{cacscades_preserve_prop_M}
The map $\psi$ preserves Property $M$, i.e. $\psi(\mathcal{M}_U)\subset \mathcal{M}_U$.  Moreover, $\alpha_{\mathbf d}^-=\alpha_{\psi(\mathbf d)}^+$ for each $\mathbf d\in\mathcal{M}$.
\end{prop}
\begin{proof}
Let $\mathbf d=d_1\cdots d_m\in\mathcal{M}_U$, and assume $d_m=0$ (the other case is similar).  
We first show $\alpha_{\mathbf d}^-=\alpha_{\psi(\mathbf d)}^+$, assuming $\psi(\mathcal{M}_U)\subset\mathcal{M}_U$.
We compute
\begin{align*}
\alpha_{\psi(\mathbf d)}^+&=\frac{\beta^{2m}-1}{\beta^{2m}v(\mathbf d\overline{\mathbf e})-1}\\
&=\frac{(\beta^m+1)(\beta^m-1)}{\beta^{2m}(v(\mathbf d)-(1/\beta^m)v(\mathbf e))-1}\\
&=\frac{(\beta^m+1)(\beta^m-1)}{\beta^{2m}v(\mathbf d)-\beta^m(v(\mathbf d)-1)-1}\\
&=\frac{(\beta^m+1)(\beta^m-1)}{(\beta^{m}v(\mathbf d)+1)(\beta^m-1)}\\
&=\frac{\beta^m+1}{\beta^{m}v(\mathbf d)+1}\\
&=\alpha_{\mathbf d}^-
\end{align*}
as desired.  Now we prove that $\mathbf d':=\psi(\mathbf d)\in\mathcal{M}_U$.  Clearly $\mathbf d'\notin\{10,1001\}$, so we need only show $\mathbf d'\in\mathcal{M}$.  Write
\[\mathbf d=1\mathbf w_{i_1}\cdots \mathbf w_{i_n}0\]
with $i_n=2$ and
\[\mathbf e=\varphi(\mathbf d)=\overline{0\mathbf w_{2-i_1}\cdots \mathbf w_{2-i_n}1}.\]
Then 
\begin{align*}
\mathbf d'&=\mathbf d\overline{\mathbf e}\\
&=1\mathbf w_{i_1}\cdots \mathbf w_{i_n}00\mathbf w_{2-i_1}\cdots \mathbf w_{2-i_n}1\\
&=1\mathbf w_{i_1}\cdots \mathbf w_{i_n}\mathbf w_0\mathbf w_{2-i_1}\cdots \mathbf w_{2-i_n}1,
\end{align*}
so $\mathbf d'\in\mathcal B$ is in admissible block form.  To prove $\mathbf d'\in\mathcal{M}$, it remains to show for each $j\ge0$ that (i) $\sigma^j(\mathbf d')\preceq \mathbf d'$ and (ii) $\sigma^j(\overline{\varphi(\mathbf d')})\preceq \mathbf d'$.  (Recall that $\mathbf d\in\mathcal{M}$ implies the analogous inequalities hold for $\mathbf d$.)
\begin{enumerate}
\item[(i)]  If $j\ge m$, then 
\[\sigma^j(\mathbf d')=\sigma^j(\mathbf d\overline{\mathbf e})=\sigma^{j-m}(\overline{\mathbf e})\preceq\mathbf d\preceq\mathbf d'.\]
Assume $j<m$, and suppose for the sake of contradiction that $\sigma^j(\mathbf d')\succ\mathbf d'$.  Since $\mathbf d'$ begins with $1$, so does $\sigma^j(\mathbf d')$.  Thus either 
\[\sigma^j(\mathbf d')=1\mathbf w_{i_\ell}\cdots \mathbf w_{i_n}\mathbf w_0\mathbf w_{2-i_1}\cdots \mathbf w_{2-i_n}1\]
for some $1<\ell\le n$, or
\[\sigma^j(\mathbf d')=1\mathbf w_0\mathbf w_{2-i_1}\cdots \mathbf w_{2-i_{n}}1.\]
Since $\mathbf w_0\prec\mathbf w_2=\mathbf w_{i_1}$, the second case is impossible and we must have
\[1\mathbf w_{i_\ell}\cdots \mathbf w_{i_n}\mathbf w_0\mathbf w_{2-i_1}\cdots \mathbf w_{2-i_n}1\succ 1\mathbf w_{i_1}\cdots \mathbf w_{i_n}\mathbf w_0\mathbf w_{2-i_1}\cdots \mathbf w_{2-i_n}1\]
for some $\ell$.  Since $\sigma^j(\mathbf d)\preceq \mathbf d$, it follows that
\[1\mathbf w_{i_\ell}\cdots \mathbf w_{i_n}=1\mathbf w_{i_1}\cdots \mathbf w_{i_{n-\ell+1}}\]
and thus
\[\mathbf w_0\mathbf w_{2-i_1}\cdots \mathbf w_{2-i_n}1\succ \mathbf w_{i_{n-\ell+2}}\cdots \mathbf w_{i_n}\mathbf w_0\mathbf w_{2-i_1}\cdots \mathbf w_{2-i_n}1.\]
Then either there is some $1\le p\le \ell-3$ for which
\[(0,2-i_1,\dots,2-i_{p-1})=(i_{n-\ell+2},i_{n-\ell+3},\dots,i_{n+p-\ell+1})\]
and $2-i_p>i_{n+p-\ell+2}$, or 
\[(0,2-i_1,\dots,2-i_{\ell-2})=(i_{n-\ell+2},i_{n-\ell+3},\dots,i_n).\]
In the first case, 
\[(2-i_{n-\ell+2},2-i_{n-\ell+3},\dots,2-i_{n+p-\ell+1})=(2,i_1,\dots,i_{p-1})\]
and $2-i_{n+p-\ell+2}>i_p$.  Thus there exists some $k\ge 0$ for which 
\begin{align*}
\sigma^k(\overline{\mathbf e})&=1\mathbf w_{2-i_{n-\ell+3}}\cdots \mathbf w_{2-i_{n+p-\ell+1}}\mathbf w_{2-i_{n+p-\ell+2}}\cdots \mathbf w_{2-i_n}1\\
&\succ 1\mathbf w_{i_1}\cdots \mathbf w_{i_{p-1}}\mathbf w_{i_p}\cdots \mathbf w_{i_n}0\\
&=\mathbf d,
\end{align*}
contradicting the fact that $\mathbf d\in\mathcal{M}$.  In the second case, 
\[(2-i_{n-\ell+2},2-i_{n-\ell+3},\dots,2-i_n)=(2,i_1,\dots,i_{\ell-2}).\]
Since $i_n=2$ implies $i_{\ell-2}=0$, there is again some $k\ge 0$ for which 
\begin{align*}
\sigma^k(\overline{\mathbf e})&=1\mathbf w_{2-i_{n-\ell+3}}\cdots  \mathbf w_{2-i_{n-1}}\mathbf w_{2-i_n}1\\
&=1\mathbf w_{2-i_{n-\ell+3}}\cdots  \mathbf w_{2-i_{n-1}}\mathbf w_1\\
&\succ 1\mathbf w_{i_1}\cdots \mathbf w_{i_{\ell-3}}\mathbf w_{i_{\ell-2}}\cdots \mathbf w_{i_n}0\\
&=\mathbf d,
\end{align*}
contradicting $\mathbf d\in\mathcal{M}$.

\item[(ii)]  Set $\mathbf e':=\varphi(\mathbf d')=e_1\cdots e_m$, and recall that $d_m=0$ implies $e_m=-1$.  Then 
\begin{align*}
\mathbf e'&=\overline{0\mathbf w_{2-i_1}\cdots\mathbf w_{2-i_n}\mathbf w_2\mathbf w_{i_1}\cdots\mathbf w_{i_n}0}\\
&=e_1\cdots e_{m-1}0\overline{\mathbf d}.
\end{align*}
If $j<{m-1}$, then 
\[\sigma^j(\overline{\mathbf e'})=\overline{e_{j+1}\cdots e_{m-1}}0\mathbf d\prec \overline{e_{j+1}\cdots e_{m}}=\sigma^j(\overline{\mathbf e})\preceq\mathbf d\preceq\mathbf d'.\]
If $j=m-1$, then
\[\sigma^j(\overline{\mathbf e'})=0\mathbf d\prec\mathbf d',\]
and if $j\ge m$, then 
\[\sigma^j(\overline{\mathbf e'})=\sigma^{j-m}(\mathbf d)\preceq\mathbf d\preceq\mathbf d'.\]
\end{enumerate}
This concludes the proof that $\mathbf d'=\psi(\mathbf d)\in\mathcal{M}$ and thus $\psi(\mathcal{M}_U)\subset\mathcal{M}_U$.
\end{proof}

\section{Invariant measures and frequencies of digits}\label{Invariant measures and frequencies of digits}

As noted above, our main interest in matching arises from results of \cite{kopf_90} which provide explicit expressions for the densities of absolutely continuous invariant measures.  These densities depend on the orbits of the left and right limits at critical points and are in general infinite sums of (finite) step functions; however, the infinite sum becomes finite when either matching or a Markov partition occurs.  These observations are used in this section to obtain explicit invariant measures $\nu_\alpha$ and $\mu_\alpha$ for the maps $S_\alpha$ and $T_\alpha$, respectively, and asymptotic relative frequencies of digits occurring in their respective generic expansions. These measures and frequencies are used in the proofs of Theorems \ref{main_measure_thm} and \ref{main_freq_thm}.

Recall that $B(x):=\beta x\ (\text{mod}\ 1)$.  It is well known that 
\[h(x):=\begin{cases}
\frac{5+3\sqrt{5}}{10}, & x\in[0,1/\beta)\\
\frac{5+\sqrt{5}}{10}, & x\in[1/\beta,1]
\end{cases}\]
is the density of a unique, ergodic, $B$-invariant probability measure which is equivalent to Lebesgue measure $\lambda$ (\cite{renyi_57}).  By Birkhoff's ergodic theorem, the frequency of $0$ in $\lambda$-a.e. $\beta$-expansion is $\int_{[0,1/\beta)}hd\lambda=(5+\sqrt{5})/10$.  When $\alpha=1$, the map $S_\alpha=S_1$ restricts on $[0,1]\backslash\{1/\beta\}$ to $B$ and on $[-1,0]\backslash\{-1/\beta\}$ to $-B(-x)$.  Since $S_1$ is invariant on $\pm[0,1]$, we find that the frequency of $0$ in $\lambda$-a.e. $S_1$-expansion is also $(5+\sqrt{5})/10$.  Define $f_1:[-1,1]\to [-1,1]$ by $f_1(x)=h(|x|)/2$, and recall the definitions of the subintervals $J_i\subset[-1,1],\ i\in\{-1,0,1\}$ from \S\ref{Introduction}.  Note, then, that the measure $\nu_1$ defined on Lebesgue-measurable $A\subset[-1,1]$ by $\nu_1(A)=\int_Af_1d\lambda$ satisfies $\nu_1(J_0):=(5+\sqrt{5})/10$.  

A similar analysis (with Lebesgue measure) reveals that the frequency of $0$ in $\lambda$-a.e. $T_1$-expansion is $1/\beta$.  Setting $\mu_1:=\lambda/2$ as normalised Lebesgue measure gives $\mu_1(J_0)=1/\beta$.  In what follows we consider $\alpha\neq 1$.

\subsection{Invariant measures}\label{Invariant measures}

Let $\alpha\in(1,\beta]$.  Following a procedure completely analogous to that in \S2.1 of \cite{dajani_kalle_20}, results of \cite{kopf_90} imply that the collection of absolutely continuous $S_\alpha$-invariant measures forms a one real-dimensional linear space and thus there is a unique---and hence ergodic---absolutely continuous invariant probability measure $\nu_\alpha$.  Moreover, its corresponding probability density is given explicitly by
\[f_\alpha(x):=\frac1C\sum_{t\ge0}\frac{1}{\beta^{t+1}}\left(1_{[-1,S_\alpha^t(\alpha-1))}(x)-1_{[-1,S_\alpha^t(-1))}(x)+1_{[-1,S_\alpha^t(1))}(x)-1_{[-1,S_\alpha^t(1-\alpha))}(x)\right),\]
where $C\in\mathbb{R}$ is some normalising constant.  Symmetry of $S_\alpha$ together with Proposition \ref{difference_of_orbits} allow us to rewrite $f_\alpha(x)$ as 
\begin{equation}\label{invariant_density_simplified}
f_\alpha(x)=\frac1C\sum_{t\ge 0}\frac1{\beta^{t+1}}\left(1_{[S_\alpha^t(-1),S_\alpha^t(\alpha-1))}(x)+1_{[S_\alpha^t(1-\alpha),S_\alpha^t(1))}(x)\right).
\end{equation}
Note that $f_\alpha$ is bounded away from $0$ on $[-1,1)$, so $\nu_\alpha$ is in fact equivalent to Lebesgue measure $\lambda$.  Also observe that when matching (or a Markov partition) occurs, the summation becomes a finite sum and $f_\alpha(x)$ is a (finite) step function (see Figure \ref{densities_pic}).  

The measure $\nu_\alpha$ can now be used to obtain a unique, absolutely continuous $T_\alpha$-invariant measure $\mu_\alpha=\int g_\alpha d\lambda$.  For each $\alpha\in (1,\beta]$, define a probability measure  
\begin{equation}\label{mu_alpha}
\mu_\alpha(A):=\frac{\nu_\alpha\left(S_\alpha^{-1}(A)\cap J_0\right)}{\nu_\alpha(J_0)}.
\end{equation}
on $[-1,1]$, where $A\subset[-1,1]$ is Lebesgue-measurable.  Note that $S_\alpha^{-1}(A)\cap J_0=\frac1\beta A$, so $\mu_\alpha$ may also be written $\mu_\alpha(A)=\nu_\alpha(\frac1\beta A)/\nu_\alpha(J_0)$.

\begin{thm}\label{mu_alpha_thm}
The measure $\mu_\alpha$ is the unique---hence ergodic---invariant probability measure for $T_\alpha$ which is absolutely continuous with respect to Lebesgue measure.  Moreover, $\mu_\alpha$ is equivalent to Lebesgue measure.
\end{thm}
\begin{proof}
Since $T_\alpha$ is an expanding, piecewise $C^2$ monotone map, results of \cite{lasota_yorke_73} imply the existence of an invariant probability measure $\rho_\alpha$ for $T_\alpha$ which is absolutely continuous with respect to Lebesgue measure.  Let $J_{\pm 1}:=J_{-1}\cup J_1$.  As $T_\alpha$ is a jump transformation for $S_\alpha$, the measure $\rho_\alpha$ induces an $S_\alpha$-invariant measure defined by
\begin{equation}\label{rho_alpha}
\tilde\rho_\alpha(A):=\rho_\alpha(A)+\rho_\alpha\left(S_\alpha^{-1}(A)\cap J_{\pm 1}\right)
\end{equation}
(see, e.g. Proposition 11.4.1 of \cite{dajani_kalle_21}).  Note that for any $A\subset J_{\pm 1}$ we have $S_\alpha^{-1}(A)\subset J_0$, so (\ref{rho_alpha}) gives  $\tilde\rho_\alpha(A)=\rho_\alpha(A)$.  Then for any measurable $A\subset [-1,1]$,
\[\tilde\rho_\alpha\left(S_\alpha^{-1}(A)\cap J_{\pm 1}\right)=\rho_\alpha\left(S_\alpha^{-1}(A)\cap J_{\pm 1}\right)\]
and (\ref{rho_alpha}) gives
\[\rho_\alpha(A)=\tilde\rho_\alpha(A)-\tilde\rho_\alpha\left(S_\alpha^{-1}(A)\cap J_{\pm 1}\right).\]
Since $\tilde\rho_\alpha$ is $S_\alpha$-invariant, the previous line may be rewritten
\[\rho_\alpha(A)=\tilde\rho_\alpha(S_\alpha^{-1}(A))-\tilde\rho_\alpha\left(S_\alpha^{-1}(A)\cap J_{\pm 1}\right)=\tilde\rho_\alpha(S_\alpha^{-1}(A)\cap J_0).\]
Recall that $\nu_\alpha$ is the unique invariant, absolutely continuous probability measure for $S_\alpha$, so $\tilde\rho_\alpha=c\nu_\alpha$ for some $c>0$.  Thus
\[\rho_\alpha(A)=c\nu_\alpha\left(S_\alpha^{-1}(A)\cap J_0\right),\]
and setting $A=[-1,1]$ gives $c=1/\nu_\alpha(J_0)$.  Hence $\rho_\alpha=\mu_\alpha$. 

That $\mu_\alpha$ is equivalent to Lebesgue measure $\lambda$ follows immediately from the fact that $\nu_\alpha$ is equivalent to $\lambda$ and the observation above that $\mu_\alpha(A)=\nu_\alpha(\frac1\beta A)/\nu_\alpha(J_0)$.
\end{proof}

We are now ready to prove Theorem \ref{main_measure_thm}:

\begin{proof}[Proof of Theorem \ref{main_measure_thm}]
Theorem \ref{mu_alpha_thm} asserts the existence of a unique, absolutely continuous $T_\alpha$-invariant probability measure $\mu_\alpha$ which is in fact equivalent to Lebesgue measure.  
It remains to show that for fixed $\mathbf d\in\mathcal{M}$, the density $g_\alpha$ of each $\mu_\alpha,\ \alpha\in I_{\mathbf d},$ is a step function with at most the same, finite number of jumps.  Using a change of variables, one finds that
\[\mu_\alpha(A)=\frac{\nu_\alpha(\frac1\beta A)}{\nu_\alpha(J_0)}=\frac1{\nu_\alpha(J_0)}\int_{\frac1\beta  A}f_\alpha(x) d\lambda(x)=\frac1{\beta\nu_\alpha(J_0)}\int_{A}f_\alpha(x/\beta) d\lambda(x),\]
so
\[g_\alpha(x)=\frac{f_\alpha(x/\beta)}{\beta\nu_\alpha(J_0)}.\]
Since, by (\ref{invariant_density_simplified}), $f_\alpha$ is a linear combination of at most $2m(\alpha)$ indicator functions and $m(\alpha)$ is constant on $I_{\mathbf d}$, the result follows.  
\end{proof}

\begin{figure}[t]
\begin{minipage}[t]{.32\textwidth}
\includegraphics[width=1\textwidth]{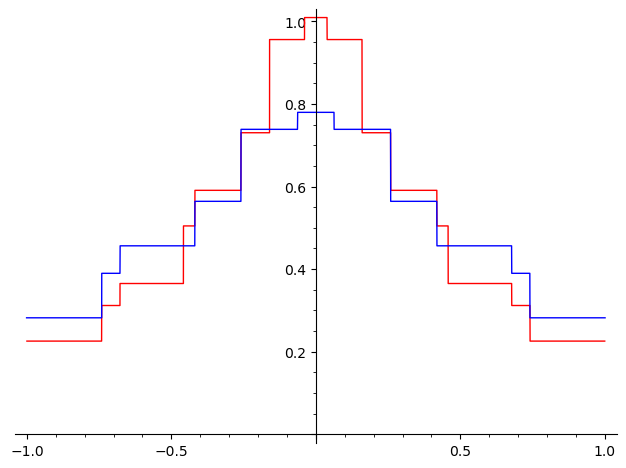}
\end{minipage}
\begin{minipage}[t]{.32\textwidth}
\includegraphics[width=1\textwidth]{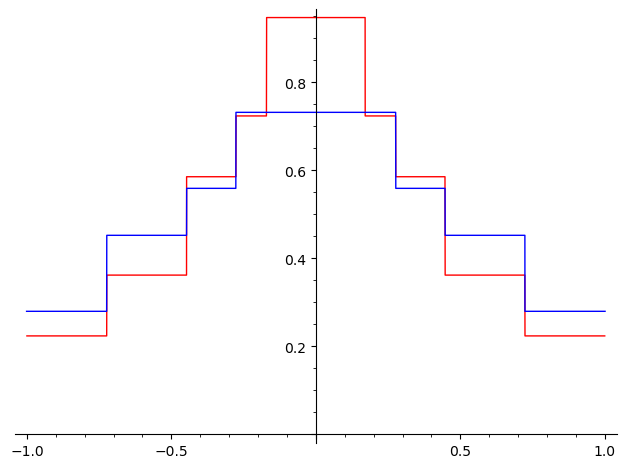}
\end{minipage}
\begin{minipage}[t]{.32\textwidth}
\includegraphics[width=1\textwidth]{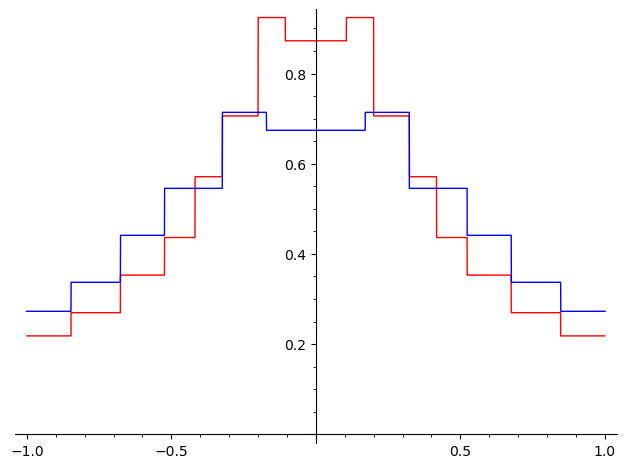}
\end{minipage}
\caption{The invariant densities $f_\alpha$ for $S_\alpha$ (red) and $g_\alpha$ for $T_\alpha$  (blue) with $\alpha=1.16$ (left), $\alpha=1/v(1010)\approx 1.17082\dots$ (center) and $\alpha=1.2$ (right).}
\label{densities_pic}
\end{figure}
 
\begin{remark}
The number of jumps of the invariant densities $f_\alpha$ and $g_\alpha$ for $S_\alpha$ and $T_\alpha$, respectively, are non-constant on matching intervals $I_\mathbf{d}$.  Figure \ref{densities_pic} shows these densities for three values of $\alpha$ in the matching interval $I_{\mathbf d}\approx(1.14589\dots,1.23606\dots)$ with $\mathbf d=1010$.  Note that the number of jumps is fewer for $\alpha=1/v(\mathbf d)$.  One can show that this phenomenon generalises to all matching intervals; in fact, for each $\mathbf d\in\mathcal{M}$, the number of jumps of $f_\alpha$ and $g_\alpha$, respectively, are constant for all but finitely many $\alpha\in I_\mathbf{d}$, and the number of jumps decreases for $\alpha=1/v(\mathbf d)\in I_\mathbf{d}$.
\end{remark}

\subsection{Frequencies of digits}\label{Frequencies of digits}

We are now in a position to determine the frequencies of digits in generic $S_\alpha$- and $T_\alpha$-expansions.  Define $\mathfrak{f}_S,\mathfrak{f}_T:[1,\beta]\to [0,1]$ by
\[\mathfrak{f}_S(\alpha):=\nu_\alpha(J_0)\ \ \ \ \ \text{and} \ \ \ \ \ \mathfrak{f}_T(\alpha):=\mu_\alpha(J_0).\]
For $\alpha\neq 1$, Birkhoff's ergodic theorem---together with the equivalence of the ergodic measures $\nu_\alpha$ and $\mu_\alpha$ with Lebesgue measure $\lambda$---implies that the asymptotic frequencies
\[\lim_{n\to\infty}\frac1n\sum_{i=0}^{n-1}1_{J_0}(S_\alpha^i(x))\ \ \ \ \ \text{and}\ \ \ \ \ \lim_{n\to\infty}\frac1n\sum_{i=0}^{n-1}1_{J_0}(T_\alpha^i(x))\]
of the digit $0$ in Lebesgue-a.e. $S_\alpha$- and $T_\alpha$-expansion are given by $\mathfrak{f}_S(\alpha)$ and $\mathfrak{f}_T(\alpha)$, respectively.  Indeed, with the discussion and notation given at the beginning of \S\ref{Invariant measures and frequencies of digits}, $\mathfrak{f}_S(1)$ and $\mathfrak{f}_T(1)$ also give the generic asymptotic frequencies of the digit $0$.  Note, too, that the frequencies of the digits $\pm 1$ are readily obtained from the frequency of 0.  

As in the proof of Theorem \ref{mu_alpha_thm}, set $J_{\pm 1}:=J_{-1}\cup J_1$.  Using (\ref{mu_alpha}) and the $S_\alpha$-invariance of $\nu_\alpha$, one has for any measurable $A\subset[-1,1]$,
\[\mu_\alpha(A)=\frac{\nu_\alpha(S_\alpha^{-1}(A))-\nu_\alpha(S_\alpha^{-1}(A)\cap J_{\pm 1})}{\nu_\alpha(J_0)}=\frac{\nu_\alpha(A)-\nu_\alpha(S_\alpha^{-1}(A)\cap J_{\pm 1})}{\nu_\alpha(J_0)}.\]
Setting $A=J_0$ and using the fact that $S_\alpha^{-1}(J_0)\cap J_{\pm 1}=J_{\pm 1}$, we find
\[\mu_\alpha(J_0)=\frac{\nu_\alpha(J_0)-\nu_\alpha(J_{\pm 1})}{\nu_\alpha(J_0)}=\frac{\nu_\alpha(J_0)-(1-\nu_\alpha(J_0))}{\nu_\alpha(J_0)}\]
or
\begin{equation}\label{freq_T_and_freq_S}
\mathfrak{f}_T(\alpha)=2-\frac1{\mathfrak{f}_S(\alpha)}.
\end{equation}
\begin{prop}\label{freq_funcs_continuous}
The frequency functions $\mathfrak{f}_S$ and $\mathfrak{f}_T$ are continuous.
\end{prop}
\begin{proof}
Arguments completely analogous to those in \S4 of \cite{dajani_kalle_20} give that $\mathfrak{f}_S$ is continuous.  Continuity of $\mathfrak{f}_T$ is immediate from \ref{freq_T_and_freq_S}.
\end{proof}

\begin{figure}[t]
\includegraphics[width=.5\textwidth]{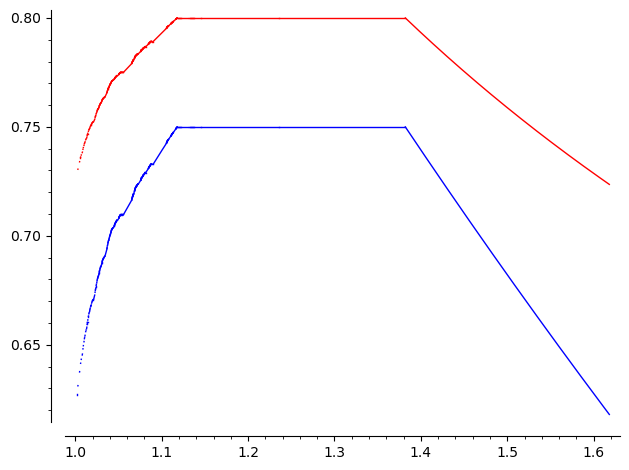}
\caption{The frequency functions $\mathfrak{f}_S(\alpha)$ (red) and $\mathfrak{f}_T(\alpha)$ (blue) plotted on all matching intervals $I_{\mathbf d}$ with $\text{len}(\mathbf d)\le 20$.  The visible plateaux correspond to the interval $[1/2+1/\beta,1+1/\beta^2]$.}
\label{frequency_function_pic}
\end{figure}

The remainder of this subsection is devoted to finding---for matching parameters $\alpha$---an explicit expression for $\mathfrak{f}_S(\alpha)$ in terms of $\alpha$ and its corresponding matching word $\mathbf d$ (see Figure \ref{frequency_function_pic}).  Density of matching parameters in $[1,\beta]$, continuity of $\mathfrak{f}_S$ and equation (\ref{freq_T_and_freq_S}) then allow us to determine $\mathfrak{f}_S(\alpha)$ and $\mathfrak{f}_T(\alpha)$ for any $\alpha\in[1,\beta]$ as limits of these explicit expressions.  These expressions are then used in \S\ref{Maximal frequency of zero} to determine the maximal frequency of the digit $0$ occurring in generic $S_\alpha$- and $T_\alpha$-expansions, and it is shown that these maximal values are attained for $\alpha$ in the interval $[1/2+1/\beta,1+1/\beta^2]$.

Assume that $\alpha\in I_\mathbf d,\ \mathbf d\in\mathcal{M}$, with matching index $m:=m(\alpha)<\infty$, and recall the density $f_\alpha$ from equation (\ref{invariant_density_simplified}).  We first find an expression for the normalising constant $C$.  By symmetry of $S_\alpha$, 
\begin{align*}
1&=\nu_{\alpha}([-1,1])\\
&=\int_{-1}^1f_\alpha(x)d\lambda(x)\\
&=\frac2C\sum_{t=0}^{m-1}\int_{-1}^1\frac1{\beta^{t+1}}1_{[S_\alpha^t(1-\alpha),S_\alpha^t(1))}(x)d\lambda(x)\\
&=\frac2C\sum_{t=0}^{m-1}\frac1{\beta^{t+1}}\left(S_\alpha^t(1)-S_\alpha^t(1-\alpha)\right).
\end{align*}
Assume $\alpha<1+1/\beta^2$ and write
\[\mathbf d=d_1\cdots d_m=1\mathbf w_{i_1}\cdots \mathbf w_{i_n}(1-i_n/2).\]
For each $i\in{0,1,2}$, let $\ell(i)\in\{2,3\}$ denote the length of the block $\mathbf w_i$---explicitly, $\ell(0)=\ell(2)=2$ and $\ell(1)=3$---and let $p:=p_{\mathbf d}:\{1,\dots,n\}\to\{1,\dots,m-3\}$ be defined by $p(k)=1+\sum_{j=1}^{k-1}\ell(i_j)$ so that $\sigma^{p(k)}(\mathbf d)=\mathbf w_{i_k}\cdots \mathbf w_{i_n}(1-i_n/2)$.  Recall from Figure \ref{differences_graph} that $S_\alpha^0(1)-S_\alpha^0(1-\alpha)=\alpha,\ S_\alpha^{m-1}(1)-S_\alpha^{m-1}(1-\alpha)=\alpha/\beta$, and that the remaining differences $S_\alpha^t(1)-S_\alpha^t(1-\alpha)$ are determined by cycles of length two or three beginning at vertex $\alpha/\beta$.  In particular, if $i_k\in\{0,2\}$, then $S_\alpha^{p(k)}(1)-S_\alpha^{p(k)}(1-\alpha)=\alpha/\beta$ and $S_\alpha^{p(k)+1}(1)-S_\alpha^{p(k)+1}(1-\alpha)=\alpha$ give a cycle of length two, while if $i_k=1$, $S_\alpha^{p(k)}(1)-S_\alpha^{p(k)}(1-\alpha)=\alpha/\beta,\ S_\alpha^{p(k)+1}(1)-S_\alpha^{p(k)+1}(1-\alpha)=\alpha$ and $S_\alpha^{p(k)+2}(1)-S_\alpha^{p(k)+2}(1-\alpha)=\beta\alpha$ give a cycle of length three.  We find for each $k\in\{1,\dots,n\}$ that
\[\sum_{t=p(k)}^{p(k)+\ell(i_k)-1}\frac1{\beta^{t+1}}\left(S_\alpha^t(1)-S_\alpha^t(1-\alpha)\right)=\frac{\ell(i_k)}{\beta^{p(k)+2}}\alpha,\]
and thus
\begin{flalign}\label{normalizer_computation}
1&=\frac2C\sum_{t=0}^{m-1}\frac1{\beta^{t+1}}\left(S_\alpha^t(1)-S_\alpha^t(1-\alpha)\right)\nonumber\\
&=\frac2C\left(\frac\alpha\beta+\sum_{k=1}^n\sum_{t=p(k)}^{p(k)+\ell(i_k)-1}\frac1{\beta^{t+1}}\left(S_\alpha^t(1)-S_\alpha^t(1-\alpha)\right)+\frac\alpha{\beta^{m+1}}\right)\nonumber\\
&=\frac{2\alpha}C\left(\frac1\beta+\sum_{k=1}^n\frac{\ell(i_k)}{\beta^{p(k)+2}}+\frac1{\beta^{m+1}}\right).
\end{flalign}
Note that (\ref{normalizer_computation}) also holds for $\alpha>1+1/\beta^2$ (i.e. $\mathbf d=10$) with the summation over $k$ set to zero.  Define a substitution $\xi:\{\mathbf w_0,\mathbf w_1,\mathbf w_2\}\to\{02,030\}$ by $\xi(\mathbf w_0)=\xi(\mathbf w_2)=02$ and $\xi(\mathbf w_1)=030$, and let $\Xi:\mathcal{M}\to\{0,1,2,3\}^*$ be given by $\Xi(\mathbf d)=101$ if $\mathbf d=10$, and
\[\Xi(\mathbf d)=1\xi(\mathbf w_{i_1})\cdots\xi(\mathbf w_{i_n})01\]
if $\mathbf d=1\mathbf w_{i_1}\cdots \mathbf w_{i_n}(1-i_n/2)\in\mathcal{M}\backslash\{10\}$.  The left- and right-most sides of (\ref{normalizer_computation}) may be written more succinctly as $1=\frac{2\alpha}{C}v(\Xi(\mathbf d))$, and thus $C=2\alpha v(\Xi(\mathbf d))$.

Having found $C$, we are now in a position to determine $\mathfrak{f}_S(\alpha)$.  Again by symmetry of $S_\alpha$, 
\begin{align*}
\mathfrak{f}_S(\alpha)&=\nu_\alpha(J_0)\\
&=1-\nu_\alpha(J_{-1})-\nu_\alpha(J_1)\\
&=1-\int_{-1}^{-1/\beta}f_\alpha(x)d\lambda(x)-\int_{1/\beta}^{1}f_\alpha(x)d\lambda(x)\\
&=1-\frac2C\sum_{t=0}^{m-1}\left(\int_{-1}^{-1/\beta}\frac1{\beta^{t+1}}1_{[S_\alpha^t(1-\alpha),S_\alpha^t(1))}(x)d\lambda(x)+\int_{1/\beta}^{1}\frac1{\beta^{t+1}}1_{[S_\alpha^t(1-\alpha),S_\alpha^t(1))}(x)d\lambda(x)\right).
\end{align*}
Write $\mathbf e:=\varphi(\mathbf d)=e_1\cdots e_m$.  Since by Proposition \ref{difference_of_orbits}, $S_\alpha^t(1)\notin J_{-1}$ and $S_\alpha^t(1-\alpha)\notin J_1$ for $t<m$, the previous line may be rewritten as
\begin{align*}
\mathfrak{f}_S(\alpha)&=1-\frac2C\left(\sum_{\substack{0\le t\le m-1\\ e_{t+1}=-1}} \frac1{\beta^{t+1}}(-1/\beta-S_\alpha^t(1-\alpha))+\sum_{\substack{0\le t\le m-1\\ d_{t+1}=1}} \frac1{\beta^{t+1}}(S_\alpha^t(1)-1/\beta)\right)\\
&=1-\frac2C\left(\sum_{\substack{0\le t\le m-1\\ d_{t+1}=1}} \frac1{\beta^{t+1}}S_\alpha^t(1)-\sum_{\substack{0\le t\le m-1\\ e_{t+1}=-1}} \frac1{\beta^{t+1}}S_\alpha^t(1-\alpha)-1/\beta\right),
\end{align*}
where we have used Proposition \ref{property_M_difference} together with the facts that
\[\sum_{\substack{0\le t\le m-1\\ e_{t+1}=-1}}1/\beta^{t+1}=-v(\mathbf e)\ \ \ \ \ \text{and}\ \ \ \ \ \sum_{\substack{0\le t\le m-1\\ d_{t+1}=1}}1/\beta^{t+1}=v(\mathbf d).\] 
Let $\mathbf d_1^0=\mathbf e_1^0=\varepsilon$ be the empty word, and for $1\le t\le m-1$ set $\mathbf d_1^t:=d_1\cdots d_t$ and $\mathbf e_1^t:=e_1\cdots e_t$.  For each $0\le t\le m-1$, equation (\ref{S_alpha^k_eqn}) gives $S_\alpha^t(1)=\beta^t(1-\alpha v(\mathbf d_1^t))$ and $S_\alpha^t(1-\alpha)=\beta^t(1-\alpha-\alpha v(\mathbf e_1^t))$.  Setting 
\begin{equation}\label{mathfrak(n)}
\mathfrak n(\mathbf d):=\#\{1\le j\le m\ |\ d_j=1\}-\#\{1\le j\le m\ |\ e_j=-1\},
\end{equation}
the frequency function may be written as
\begin{align*}
\mathfrak{f}_S(\alpha)&=1-\frac2C\left(\sum_{\substack{0\le t\le m-1\\ d_{t+1}=1}} \frac1{\beta^{t+1}}\beta^t(1-\alpha v(\mathbf d_1^t))-\sum_{\substack{0\le t\le m-1\\ e_{t+1}=-1}} \frac1{\beta^{t+1}}\beta^t(1-\alpha-\alpha v(\mathbf e_1^t))-1/\beta\right)\\
&=1-\frac2{\beta C}\left(\sum_{\substack{0\le t\le m-1\\ d_{t+1}=1}}(1-\alpha v(\mathbf d_1^t))-\sum_{\substack{0\le t\le m-1\\ e_{t+1}=-1}} (1-\alpha-\alpha v(\mathbf e_1^t))-1\right)\\
&=1-\frac2{\beta C}\left(\mathfrak{n}(\mathbf d)-\alpha\left(\sum_{\substack{0\le t\le m-1\\ d_{t+1}=1}}v(\mathbf d_1^t)-\sum_{\substack{0\le t\le m-1\\ e_{t+1}=-1}} (1+v(\mathbf e_1^t))\right)-1\right).\\
\end{align*}
Letting
\[K_{\mathbf d}:=\sum_{\substack{0\le t\le m-1\\ d_{t+1}=1}}v(\mathbf d_1^t)-\sum_{\substack{0\le t\le m-1\\ e_{t+1}=-1}} (1+v(\mathbf e_1^t))\]
and recalling that $C=2\alpha v(\Xi(\mathbf d))$, we find
\begin{equation}\label{freq_func}
\mathfrak{f}_S(\alpha)=1-\frac1{\beta v(\Xi(\mathbf d))}\left(\frac{\mathfrak{n}(\mathbf d)-1}{\alpha}-K_{\mathbf d}\right).
\end{equation}

\begin{eg}\label{freq_eg}
Let $\mathbf d=1001$.  Then $\mathbf e=\overline{0010}$, so $\mathfrak{n}(\mathbf d)=1$.  Moreover,
\[v(\Xi(\mathbf d))=v(10201)=\frac1\beta+\frac2{\beta^3}+\frac1{\beta^5}\]
and 
\[K_{\mathbf d}=v(\varepsilon)+v(100)-(1-v(\overline{00}))=-\frac1{\beta^2}.\]
Thus for all $\alpha\in I_{1001}$, 
\[\mathfrak{f}_S(\alpha)=1-\frac{1}{\beta^3(1/\beta+2/\beta^3+1/\beta^5)}=4/5.\]
A similar calculation with $\mathbf d=1010$ reveals that $\mathfrak{f}_S(\alpha)=4/5$ also for all $\alpha\in I_{1010}$.
\end{eg}

Before turning toward the maximal frequency of the digit 0, we give an alternate expression for $K_\mathbf d$ which will be helpful below.  Note that the first summation in the definition of $K_{\mathbf d}$ may be rewritten as the sum of all $v(\mathbf d_1^t),\ 1\le t\le m,$ for which $d_t$=1, excluding the greatest such index $t$.  The second sum may be similarly rewritten (though an extra term $1$ appears from the first non-zero summand of the original sum).  Now suppose $\mathbf d\neq 10$.  Recalling that $\{d_{m-2}d_{m-1}d_m,\overline{e_{m-2}e_{m-1}e_m}\}=\{001,010\}$, we have
\begin{align*}
K_\mathbf{d}&=\sum_{\substack{1\le t\le m-3\\ d_{t}=1}}v(\mathbf d_1^t)-\left(1+\sum_{\substack{1\le t\le m-3\\ e_{t}=-1}} (1+v(\mathbf e_1^t))\right)\\
&=v(1)+\sum_{\substack{1\le k\le n-1\\ i_k\in\{1,2\}}}v(1\mathbf w_{i_1}\cdots \mathbf w_{i_k})-\left(1+\sum_{\substack{1\le k\le n-1\\ 2-i_k\in\{1,2\}}}(1-v(0\mathbf w_{2-i_1}\cdots \mathbf w_{2-i_k}))\right).
\end{align*}
Recall that $p(k+1),\ 1\le k\le n-1$, gives the power for which $\sigma^{p(k+1)}(\mathbf d)=\mathbf w_{i_{k+1}}\cdots \mathbf w_{i_n}(1-i_n/2)$; in particular, $p(k+1)$ equals the length of $1\mathbf w_{i_1}\cdots \mathbf w_{i_k}$.  By Lemma \ref{valuation_of_blocks}, 
\[v(1\mathbf w_{i_1}\cdots \mathbf w_{i_k})+v(0\mathbf w_{2-i_1}\cdots \mathbf w_{2-i_k})=\frac1\beta+\frac1\beta\left(\frac1\beta-\frac1{\beta^{p(k+1)}}\right)=1-1/\beta^{p(k+1)+1}.\]
Then 
\begin{align*}
K_\mathbf{d}&=\frac1\beta+\sum_{\substack{1\le k\le n-1\\ i_k\in\{1,2\}}}v(1\mathbf w_{i_1}\cdots \mathbf w_{i_k})-\left(1+\sum_{\substack{1\le k\le n-1\\ 2-i_k\in\{1,2\}}}\left(v(1\mathbf w_{i_1}\cdots \mathbf w_{i_k})+1/\beta^{p(k+1)+1}\right)\right)\\
&=-\frac1{\beta^2}+\sum_{\substack{1\le k\le n-1\\ i_k=2}}v(1\mathbf w_{i_1}\cdots \mathbf w_{i_k})-\sum_{\substack{1\le k\le n-1\\ i_k=0}}v(1\mathbf w_{i_1}\cdots \mathbf w_{i_k})-\sum_{\substack{1\le k\le n-1\\ 2-i_k\in\{1,2\}}}1/\beta^{p(k+1)+1}.
\end{align*}
The latter summation equals
\begin{align*}
\sum_{\substack{1\le k\le n-1\\ 2-i_k\in\{1,2\}}}1/\beta^{p(k+1)+1}&=\frac1\beta v(0\mathbf w_{2-i_1}\cdots \mathbf w_{2-i_{n-1}})\\
&=\frac1\beta\left(v(\overline{\mathbf e})-\frac1{\beta^{m-3}}v(\mathbf w_{2-i_n}i_n/2)\right)\\
&=\frac1\beta\left(1-v(\mathbf d)-\frac1{\beta^{m-3}}v(\mathbf w_{2-i_n}i_n/2)\right)\\
&=\frac1\beta\left(1-v(d_1\cdots d_{m-3})-\frac1{\beta^{m-3}}v(011)\right)\\
&=\frac1\beta-\frac1\beta v(d_1\cdots d_{m-3})-\frac1{\beta^{m-1}},
\end{align*}
and thus for $\mathbf d\in\mathcal{M}\backslash\{10\}$,
\begin{equation}\label{K_d}
K_\mathbf{d}=-1+\sum_{\substack{1\le k\le n-1\\ i_k=2}}v(1\mathbf w_{i_1}\cdots \mathbf w_{i_k})-\sum_{\substack{1\le k\le n-1\\ i_k=0}}v(1\mathbf w_{i_1}\cdots \mathbf w_{i_k})+\frac1\beta v(d_1\cdots d_{m-3})+\frac1{\beta^{m-1}}.
\end{equation}

\subsection{Maximal frequency of zero}\label{Maximal frequency of zero}

Here we prove that the frequency functions $\mathfrak{f}_S$ and $\mathfrak{f}_T$ attain their maximums on the (maximal) interval $[1/2+1/\beta,1+1/\beta^2]$.  We first need some preliminary results.  Note that by (\ref{freq_func}), on the matching interval $I_{\mathbf d}$ the frequency function $\mathfrak{f}_S$ is strictly increasing with $\alpha$ for $\mathfrak{n}(\mathbf d)>1$, strictly decreasing for $\mathfrak{n}(\mathbf d)<1$ and constant for $\mathfrak{n}(\mathbf d)=1$.  By (\ref{freq_T_and_freq_S}), the same monotonicity conditions hold for $\mathfrak{f}_T$.

The first of our preliminary results states that $\mathfrak{f}_S$ (and hence $\mathfrak{f}_T$) is constant on `cascade' intervals:

\begin{lem}\label{freq_func_on_cascades}
For each $\mathbf d\in \mathcal{M}_U$, we have $\mathfrak{n}(\psi(\mathbf d))=1$.  In particular, for each $\mathbf d\in\mathcal{M}_U$, the frequency function $\mathfrak{f}_S$ is constant on $[\lim_{n\to\infty}\alpha_{\psi^n(\mathbf d)}^-,\alpha_{\mathbf d}^-]$. 
\end{lem}
\begin{proof}
It suffices to prove the first statement; the second follows immediately from this, Proposition \ref{cacscades_preserve_prop_M} and continuity of $\mathfrak{f}_S$.  Write 
\[\mathbf d=d_1\cdots d_m=1\mathbf w_{i_1}\cdots \mathbf w_{i_n}(1-i_n/2)\ \ \ \text{and}\ \ \ \mathbf e:=\varphi(\mathbf d)=e_1\cdots e_m=\overline{0\mathbf w_{2-i_1}\cdots \mathbf w_{2-i_n}(i_n/2)}.\]
Observe that
\begin{align*}
\mathbf d':=\psi(\mathbf d)&=\begin{cases}
\mathbf d\overline{\mathbf e}, & d_m=0\\
\mathbf d\overline{e_2\cdots e_m}, & d_m=1
\end{cases}\\
&=\begin{cases}
1\mathbf w_{i_1}\cdots \mathbf w_{i_n}00\mathbf w_{2-i_1}\cdots \mathbf w_{2-i_n}(i_n/2), & d_m=0\\
1\mathbf w_{i_1}\cdots \mathbf w_{i_{n-1}}001\mathbf w_{2-i_1}\cdots \mathbf w_{2-i_n}(i_n/2), & d_m=1
\end{cases}\\
&=\begin{cases}
1\mathbf w_{i_1}\cdots \mathbf w_{i_n}\mathbf w_0\mathbf w_{2-i_1}\cdots \mathbf w_{2-i_n}(i_n/2), & d_m=0\\
1\mathbf w_{i_1}\cdots \mathbf w_{i_{n-1}}\mathbf w_1\mathbf w_{2-i_1}\cdots \mathbf w_{2-i_n}(i_n/2), & d_m=1
\end{cases},
\end{align*}
so
\begin{align*}
\mathbf e':=\varphi(\mathbf d')&=\begin{cases}
\overline{0\mathbf w_{2-i_1}\cdots \mathbf w_{2-i_n}\mathbf w_2\mathbf w_{i_1}\cdots \mathbf w_{i_n}(1-i_n/2)}, & d_m=0\\
\overline{0\mathbf w_{2-i_1}\cdots \mathbf w_{2-i_{n-1}}\mathbf w_1\mathbf w_{i_1}\cdots \mathbf w_{i_n}(1-i_n/2)}, & d_m=1
\end{cases}\\
&=\begin{cases}
e_1\cdots e_{m-1}0\overline{\mathbf d}, & d_m=0\\
e_1\cdots e_{m-2}0\overline{\mathbf d}, & d_m=1\\
\end{cases}.
\end{align*}
Recall that if $d_m=0$, then $\overline{e_m}=1$.  In this case $\mathbf d'$ has exactly one more digit $1$ than does $\overline{\mathbf e'}$.  If $d_m=1$, then $\overline{e_{m-1}e_{m}}=10$.  Since $e_1=0$, we see that in this case, too, $\mathbf d'$ has exactly one more digit $1$ than does $\overline{\mathbf e'}$.  Thus in both cases $\mathfrak{n}(\mathbf d')=1$.
\end{proof}

We make note here of some computations which will be useful below.  Let $c,\ell\in\mathbb{Z}$ with $\ell\ge 0$:
\begin{flalign}
v((0c)^\ell)&=c\sum_{j=1}^\ell 1/\beta^{2j}=\frac c{\beta^2}\cdot\frac{1-1/\beta^{2\ell}}{1-1/\beta^2}=\frac c\beta(1-1/\beta^{2\ell})\label{comp1}\\
v((00c)^\ell)&=c\sum_{j=1}^\ell 1/\beta^{3j}=\frac {c}{\beta^3}\cdot\frac{1-1/\beta^{3\ell}}{1-1/\beta^3}=\frac c{2\beta}(1-1/\beta^{3\ell}) \label{comp2}\\
v((0c0)^\ell)&=\beta v((00c)^\ell)=\frac c2(1-1/\beta^{3\ell}) \label{comp3}\\
v((000c)^\ell)&=c\sum_{j=1}^\ell 1/\beta^{4j}=\frac c{\beta^4}\frac{1-1/\beta^{4\ell}}{1-1/\beta^4}=\frac c{\beta(\beta^2+1)}(1-1/\beta^{4\ell})
\label{comp4}\\
v((0c00)^\ell)&=\beta^2 v((000c)^\ell)=\frac{c\beta}{\beta^2+1}(1-1/\beta^{4\ell}).
\label{comp5}
\end{flalign}

\begin{lem}\label{freq_func_bdd_lem}
If $\alpha\in I_\mathbf{d}$ for some $\mathbf d\in\mathcal{M}$ with $\mathfrak{n}(\mathbf{d})=1$, then $\mathfrak{f}_S(\alpha)\le 4/5$.  Moreover, equality holds if and only if $\mathbf d\prec 1(\mathbf w_2\mathbf w_0)^\infty$.
\end{lem}
\begin{proof}
Note that $\mathfrak{n}(10)=0$, so we may assume $\mathbf d\succ 10$.  That $\mathfrak{f}_S(\alpha)=4/5$ for all $\alpha\in I_{1010}\cup I_{1001}$ was shown in Example \ref{freq_eg}.  Thus we may assume that $\mathbf d\succ 1010$.  Write
\[\mathbf d=d_1\cdots d_m=1\mathbf w_{i_1}\cdots \mathbf w_{i_n}(1-i_n/2)=1\mathbf X_1\mathbf Y_1\cdots \mathbf X_t\mathbf Y_t\mathbf w_{i_n}(1-i_n/2),\]
where each $\mathbf X_s$ and $\mathbf Y_s$, $1\le s\le t$, consists solely of $\mathbf w_{2i}$'s and $\mathbf w_1$'s, respectively, and each $\mathbf X_s,\ \mathbf Y_s\neq\varepsilon$ except possibly $\mathbf Y_t$.  Let $\ell_{2s-1}:=\frac12\text{len}(\mathbf X_s)$ and $\ell_{2s}:=\frac13\text{len}(\mathbf Y_s)$ denote the number of blocks $\mathbf w_i$ in $\mathbf X_s$ and $\mathbf Y_s$, respectively, and set $\ell_j:=0$ for $j>2t$.  Analogous to the function $p=p_{\mathbf d}$ defined in \S\ref{Frequencies of digits}, set $p_1:=1$ and for each $s\ge 1$, let $p_{2s}:=p_{2s-1}+2\ell_{2s-1}$ and $p_{2s+1}:=p_{2s}+3\ell_{2s}$; note, then, that
\[\sigma^{p_{2s-1}}(\mathbf d)=\mathbf X_s\mathbf Y_s\cdots \mathbf X_t\mathbf Y_t\mathbf w_{i_n}(1-i_n/2)\ \ \ \ \ \text{and}\ \ \ \ \ \sigma^{p_{2s}}(\mathbf d)=\mathbf Y_s\mathbf X_{s+1}\cdots \mathbf X_t\mathbf Y_t\mathbf w_{i_n}(1-i_n/2).\]
Let $k_{2s-1},k_{2s}\in\{1,\dots, n\}$ be the indices for which 
\[\sigma^{p_{2s-1}}(\mathbf d)=\mathbf w_{i_{k_{2s-1}}}\cdots \mathbf w_{i_{n-1}}\mathbf w_{i_n}(1-i_n/2)\ \ \ \ \ \text{and}\ \ \ \ \ \sigma^{p_{2s}}(\mathbf d)=\mathbf w_{i_{k_{2s}}}\cdots \mathbf w_{i_{n-1}}\mathbf w_{i_n}(1-i_n/2).\]
Using (\ref{comp1}) and (\ref{comp3}), we compute
\begin{align*}
v(\Xi(\mathbf d))&=v(1(02)^{\ell_1}(030)^{\ell_2}\cdots (02)^{\ell_{2t-1}}(030)^{\ell_{2t}}0201)\\
&=\frac1\beta+\sum_{s=1}^t\left(\frac1{\beta^{p_{2s-1}}}v((02)^{\ell_{2s-1}})+\frac1{\beta^{p_{2s}}}v((030)^{\ell_{2s}})\right)+\frac1{\beta^{m-3}}v(0201)\\
&=\frac1\beta+\sum_{s=1}^t\left(\frac{2}{\beta^{p_{2s-1}+1}}(1-1/\beta^{2\ell_{2s-1}})+\frac3{2\beta^{p_{2s}}}(1-1/\beta^{3\ell_{2s}})\right)+\frac1{\beta^{m-3}}(2/\beta^2+1/\beta^4).
\end{align*}
Moreover, (\ref{comp2}) gives
\begin{align*}
v(d_1\cdots d_{m-3})&=\frac1\beta+\sum_{s=1}^t\left(\frac1{\beta^{p_{2s-1}}}v(\mathbf X_s)+\frac1{\beta^{p_{2s}}}v(\mathbf Y_s)\right)\\
&=\frac1\beta+\sum_{s=1}^t\left(\frac1{\beta^{p_{2s-1}}}v(\mathbf X_s)+\frac1{\beta^{p_{2s}}}v((001)^{\ell_{2s}})\right)\\
&=\frac1\beta+\sum_{s=1}^t\left(\frac1{\beta^{p_{2s-1}}}v(\mathbf X_s)+\frac1{2\beta^{p_{2s}+1}}(1-1/\beta^{3\ell_{2s}})\right),
\end{align*}
so equation (\ref{K_d}) becomes
\begin{align*}
K_\mathbf{d}=&-\frac1\beta+\sum_{\substack{1\le k\le n-1\\ i_k=2}}v(1\mathbf w_{i_1}\cdots \mathbf w_{i_k})-\sum_{\substack{1\le k\le n-1\\ i_k=0}}v(1\mathbf w_{i_1}\cdots \mathbf w_{i_k})\\
&+\sum_{s=1}^t\left(\frac1{\beta^{p_{2s-1}+1}}v(\mathbf X_s)+\frac1{2\beta^{p_{2s}+2}}(1-1/\beta^{3\ell_{2s}})\right)+\frac1{\beta^{m-1}}.
\end{align*}
Then 
\begin{align*}
\beta v(\Xi(\mathbf d))+5K_\mathbf{d}=&1+\sum_{s=1}^t\left(\frac{2}{\beta^{p_{2s-1}}}(1-1/\beta^{2\ell_{2s-1}})+\frac{3\beta}{2\beta^{p_{2s}}}(1-1/\beta^{3\ell_{2s}})\right)+\frac1{\beta^{m-3}}(2/\beta+1/\beta^3)\\
&-\frac5\beta+5\left(\sum_{\substack{1\le k\le n-1\\ i_k=2}}v(1\mathbf w_{i_1}\cdots \mathbf w_{i_k})-\sum_{\substack{1\le k\le n-1\\ i_k=0}}v(1\mathbf w_{i_1}\cdots \mathbf w_{i_k})\right)\\
&+5\sum_{s=1}^t\left(\frac1{\beta^{p_{2s-1}+1}}v(\mathbf X_s)+\frac1{2\beta^{p_{2s}+2}}(1-1/\beta^{3\ell_{2s}})\right)+\frac5{\beta^{m-1}}\\
=&1-\frac5\beta+\sum_{s=1}^t\frac{1}{\beta^{p_{2s}}}\left(3\beta/2+5/2\beta^2\right)(1-1/\beta^{3\ell_{2s}})+\frac1{\beta^{m-3}}(2/\beta+5/\beta^2+1/\beta^3)\\
&+\sum_{s=1}^t\left(\frac{2}{\beta^{p_{2s-1}}}(1-1/\beta^{2\ell_{2s-1}})+\frac5{\beta^{p_{2s-1}+1}}v(\mathbf X_s)\right)\\
&+5\left(\sum_{\substack{1\le k\le n-1\\ i_k=2}}v(1\mathbf w_{i_1}\cdots \mathbf w_{i_k})-\sum_{\substack{1\le k\le n-1\\ i_k=0}}v(1\mathbf w_{i_1}\cdots \mathbf w_{i_k})\right).
\end{align*}
One easily verifies that both $3\beta/2+5/2\beta^2$ and $2/\beta+5/\beta^2+1/\beta^3$ equal $c:=5-\beta$.  We claim that it suffices to show that 
\begin{flalign}\label{claim}
&\sum_{s=1}^t\left(\frac{2}{\beta^{p_{2s-1}}}(1-1/\beta^{2\ell_{2s-1}})+\frac5{\beta^{p_{2s-1}+1}}v(\mathbf X_s)\right)\\
&+5\left(\sum_{\substack{1\le k\le n-1\\ i_k=2}}v(1\mathbf w_{i_1}\cdots \mathbf w_{i_k})-\sum_{\substack{1\le k\le n-1\\ i_k=0}}v(1\mathbf w_{i_1}\cdots \mathbf w_{i_k})\right)\nonumber\\
\le&\sum_{s=1}^t\frac c{\beta^{p_{2s-1}}}(1-1/\beta^{2\ell_{2s-1}}),\nonumber
\end{flalign}
with equality if and only if $\mathbf d\prec 1(\mathbf w_2\mathbf w_0)^\infty$.  Indeed, suppose the claim holds.  Then the computation above becomes
\begin{align*}
\beta v(\Xi(\mathbf d))+5K_\mathbf{d}&\le 1-\frac5\beta+c\sum_{s=1}^t\left(\frac1{\beta^{p_{2s-1}}}(1-1/\beta^{2\ell_{2s-1}})+\frac1{\beta^{p_{2s}}}(1-1/\beta^{3\ell_{2s}})\right)+\frac c{\beta^{m-3}}\\
&=1-\frac5\beta+c\sum_{s=1}^t\left(1/\beta^{p_{2s-1}}-1/\beta^{p_{2s}}+1/\beta^{p_{2s}}-1/\beta^{p_{2s+1}}\right)+\frac c{\beta^{m-3}}\\
&=1-\frac5\beta+c(1/\beta-1/\beta^{m-3})+\frac c{\beta^{m-3}}\\
&=1-\frac5\beta+\frac c\beta\\
&=0
\end{align*}
with equality if and only if $\mathbf d\prec 1(\mathbf w_2\mathbf w_0)^\infty$.  Rearranging, this inequality is equivalent to $ K_\mathbf{d}/\beta v(\Xi(\mathbf d))\le -1/5$.  From (\ref{freq_func}) and the assumption that $\mathfrak{n}(\mathbf{d})=1$, this gives
\[\mathfrak{f}_S(\alpha)=1+K_\mathbf{d}/\beta v(\Xi(\mathbf d))\le 4/5\]
with equality if and only if $\mathbf d\prec 1(\mathbf w_2\mathbf w_0)^\infty$, as desired.

It remains to show the claim from (\ref{claim}).  The constant $c$ defined above may be rewritten as $c=2+5/(\beta^2+1)$.  Subtracting $\sum_{s=1}^t(2/\beta^{p_{2s-1}})(1-1/\beta^{2\ell_{2s-1}})$ from both sides, dividing by $5$ and noting that $i_k\in\{0,2\}$ only when $k_{2s-1}\le k<k_{2s},\ 1\le s\le t$, equation (\ref{claim}) becomes
\begin{flalign}\label{claim2}
&\sum_{s=1}^t\left(\frac1{\beta^{p_{2s-1}+1}}v(\mathbf X_s)+\sum_{\substack{k_{2s-1}\le k<k_{2s}\\ i_k=2}}v(1\mathbf w_{i_1}\cdots \mathbf w_{i_k})-\sum_{\substack{k_{2s-1}\le k<k_{2s}\\ i_k=0}}v(1\mathbf w_{i_1}\cdots \mathbf w_{i_k})\right)\\
\le&\frac1{\beta^2+1}\sum_{s=1}^t\frac 1{\beta^{p_{2s-1}}}(1-1/\beta^{2\ell_{2s-1}}).\nonumber
\end{flalign}

Fix $1\le s\le t$, and write
\begin{equation}\label{X_s}
\mathbf X_s:=\mathbf w_2^{n_{s,1}}(\mathbf w_2\mathbf w_0)^{n_{s,2}}\mathbf w_0^{n_{s,3}}(\mathbf w_0\mathbf w_2)^{n_{s,4}}\cdots \mathbf w_2^{n_{s,4r_s-3}}(\mathbf w_2\mathbf w_0)^{n_{s,4r_s-2}}\mathbf w_0^{n_{s,4r_s-1}}(\mathbf w_0\mathbf w_2)^{n_{s,4r_s}},
\end{equation}
where the powers $n_{s,\ell}\ge 0$ are chosen so that $\sum_{\ell=1}^{4r_s}n_{s,\ell}$ is minimal and no three consecutive $n_{s,\ell}$ are zero except possibly the first or final three $n_{s,\ell}$.  Set $p_{s,1}:=p_{2s-1}$ and for each $1\le j\le r_s$,
\begin{alignat*}{2}
p_{s,4j-2}&:=p_{s,4j-3}+2n_{s,4j-3},\ \ \ \ \ \ \ \ & p_{s,4j-1}&:=p_{s,4j-2}+4n_{s,4j-2},\\
p_{s,4j}&:=p_{s,4j-1}+2n_{s,4j-1},\ \ \ \ \ \ \ \ & p_{s,4j+1}&:=p_{s,4j}+4n_{s,4j}.
\end{alignat*}
Note that with these definitions, $p_{s,4r_s+1}=p_{2s}$.  Equations (\ref{comp1})--(\ref{comp5}) give
\begin{align*}
\frac1{\beta^{p_{2s-1}+1}}v(\mathbf X_s)=&\frac1\beta\sum_{j=1}^{r_s}\bigg(\frac1{\beta^{p_{s,4j-3}}}v(\mathbf w_2^{n_{s,4j-3}})+\frac1{\beta^{p_{s,4j-2}}}v((\mathbf w_2\mathbf w_0)^{n_{s,4j-2}})\\
&+\frac1{\beta^{p_{s,4j-1}}}v(\mathbf w_0^{n_{s,4j-1}})+\frac1{\beta^{p_{s,4j}}}v((\mathbf w_0\mathbf w_2)^{n_{s,4j}})\bigg)\\
=&\sum_{j=1}^{r_s}\bigg(\frac1{\beta^{p_{s,4j-3}}}\frac1{\beta^2}(1-1/\beta^{2n_{s,4j-3}})+\frac1{\beta^{p_{s,4j-2}}}\frac1{\beta^2+1}(1-1/\beta^{4n_{s,4j-2}})\\
&+\frac1{\beta^{p_{s,4j}}}\frac1{\beta^2(\beta^2+1)}(1-1/\beta^{4n_{s,4j}})\bigg)
\end{align*}
and
\begin{align*}
&\sum_{\substack{k_{2s-1}\le k<k_{2s}\\ i_k=2}}v(1\mathbf w_{i_1}\cdots \mathbf w_{i_k})-\sum_{\substack{k_{2s-1}\le k<k_{2s}\\ i_k=0}}v(1\mathbf w_{i_1}\cdots \mathbf w_{i_k})\\
=&\sum_{j=1}^{r_s}\bigg(\sum_{\ell=1}^{n_{s,4j-3}}v(1\mathbf X_1\mathbf Y_1\cdots \mathbf X_{s-1}\mathbf Y_{s-1}\mathbf w_2^{n_{s,1}}\cdots (\mathbf w_0\mathbf w_2)^{n_{s,4j-4}}\mathbf w_2^\ell)\\
&-\sum_{\ell=1}^{n_{s,4j-1}}v(1\mathbf X_1\mathbf Y_1\cdots \mathbf X_{s-1}\mathbf Y_{s-1}\mathbf w_2^{n_{s,1}}\cdots (\mathbf w_2\mathbf w_0)^{n_{s,4j-2}}\mathbf w_0^\ell)\\
&+v(1\mathbf X_1\mathbf Y_1\cdots \mathbf X_{s-1}\mathbf Y_{s-1}\mathbf w_2^{n_{s,1}}\cdots (\mathbf w_0\mathbf w_2)^{n_{s,4j}})\\
&-v(1\mathbf X_1\mathbf Y_1\cdots \mathbf X_{s-1}\mathbf Y_{s-1}\mathbf w_2^{n_{s,1}}\cdots (\mathbf w_2\mathbf w_0)^{n_{s,4j-2}}\mathbf w_0^{n_{s,4j-1}}\mathbf w_0)\bigg)\\
=&\sum_{j=1}^{r_s}\bigg(\sum_{\ell=1}^{n_{s,4j-3}}v(1\mathbf X_1\mathbf Y_1\cdots \mathbf X_{s-1}\mathbf Y_{s-1}\mathbf w_2^{n_{s,1}}\cdots (\mathbf w_0\mathbf w_2)^{n_{s,4j-4}}\mathbf w_2^\ell)\\
&-n_{s,4j-1}v(1\mathbf X_1\mathbf Y_1\cdots \mathbf X_{s-1}\mathbf Y_{s-1}\mathbf w_2^{n_{s,1}}\cdots (\mathbf w_2\mathbf w_0)^{n_{s,4j-2}})\\
&+\frac{1}{\beta^{p_{s,4j}}}\frac1{\beta(\beta^2+1)}(1-1/\beta^{4n_{s,4j}})\bigg).
\end{align*}
Thus the left-hand side of (\ref{claim2}) equals
\begin{align*}&\sum_{s=1}^t\sum_{j=1}^{r_s}\bigg(\frac1{\beta^{p_{s,4j-3}}}\frac1{\beta^2}(1-1/\beta^{2n_{s,4j-3}})+\frac1{\beta^{p_{s,4j-2}}}\frac1{\beta^2+1}(1-1/\beta^{4n_{s,4j-2}})+\frac1{\beta^{p_{s,4j}}}\frac1{\beta^2+1}(1-1/\beta^{4n_{s,4j}})\\
&+\sum_{\ell=1}^{n_{s,4j-3}}v(1\mathbf X_1\mathbf Y_1\cdots \mathbf X_{s-1}\mathbf Y_{s-1}\mathbf w_2^{n_{s,1}}\cdots (\mathbf w_0\mathbf w_2)^{n_{s,4j-4}}\mathbf w_2^\ell)\\
&-n_{s,4j-1}v(1\mathbf X_1\mathbf Y_1\cdots \mathbf X_{s-1}\mathbf Y_{s-1}\mathbf w_2^{n_{s,1}}\cdots (\mathbf w_2\mathbf w_0)^{n_{s,4j-2}})\bigg).
\end{align*}
Moreover, using the definition of $p_{s,4j-i}$, we find that each summand on the right-hand side of (\ref{claim2}) may be expanded
\begin{align*}
\frac1{\beta^{p_{2s-1}}}(1-1/\beta^{2\ell_{2s-1}})=&1/\beta^{p_{2s-1}}-1/\beta^{p_{2s}}\\
=&1/\beta^{p_{s,1}}-1/\beta^{p_{s,4r_s+1}}\\
=&\sum_{j=1}^{r_s}\bigg(\frac1{\beta^{p_{s,4j-3}}}(1-1/\beta^{2n_{s,4j-3}})+\frac1{\beta^{p_{s,4j-2}}}(1-1/\beta^{4n_{s,4j-2}})\\
&+\frac1{\beta^{p_{s,4j-1}}}(1-1/\beta^{2n_{s,4j-1}})+\frac1{\beta^{p_{s,4j}}}(1-1/\beta^{4n_{s,4j}})\bigg).
\end{align*}
Subtracting $\sum_{s=1}^t\sum_{j=1}^{r_s}\frac1{\beta^{p_{s,4j-i}}}\frac1{\beta^2+1}(1-1/\beta^{4n_{s,4j-i}}),\ i=0,2,$ from both sides, (\ref{claim2}) becomes
\begin{align*}
&\sum_{s=1}^t\sum_{j=1}^{r_s}\bigg(\frac1{\beta^{p_{s,4j-3}}}\frac1{\beta^2}(1-1/\beta^{2n_{s,4j-3}})+\sum_{\ell=1}^{n_{s,4j-3}}v(1\mathbf X_1\mathbf Y_1\cdots \mathbf X_{s-1}\mathbf Y_{s-1}\mathbf w_2^{n_{s,1}}\cdots (\mathbf w_0\mathbf w_2)^{n_{s,4j-4}}\mathbf w_2^\ell)\\
&-n_{s,4j-1}v(1\mathbf X_1\mathbf Y_1\cdots \mathbf X_{s-1}\mathbf Y_{s-1}\mathbf w_2^{n_{s,1}}\cdots (\mathbf w_2\mathbf w_0)^{n_{s,4j-2}})\bigg)\\
\le&\sum_{s=1}^{t}\sum_{j=1}^{r_s}\bigg(\frac1{\beta^{p_{s,4j-3}}}\frac1{\beta^2+1}(1-1/\beta^{2n_{s,4j-3}})+\frac1{\beta^{p_{s,4j-1}}}\frac1{\beta^2+1}(1-1/\beta^{2n_{s,4j-1}})\bigg).
\end{align*}
Rearranging and using the fact that $1/\beta^2-1/(\beta^2+1)=1/(\beta^2(\beta^2+1))$, the previous inequality is equivalent to 
\begin{flalign}\label{claim3}
&\sum_{s=1}^t\sum_{j=1}^{r_s}\bigg(\sum_{\ell=1}^{n_{s,4j-3}}v(1\mathbf X_1\mathbf Y_1\cdots \mathbf X_{s-1}\mathbf Y_{s-1}\mathbf w_2^{n_{s,1}}\cdots (\mathbf w_0\mathbf w_2)^{n_{s,4j-4}}\mathbf w_2^\ell)\\
&+\frac1{\beta^{p_{s,4j-3}}}\frac1{\beta^2(\beta^2+1)}(1-1/\beta^{2n_{s,4j-3}})\bigg)\nonumber\\
\le&\sum_{s=1}^t\sum_{j=1}^{r_s}\left(n_{s,4j-1}v(1\mathbf X_1\mathbf Y_1\cdots \mathbf X_{s-1}\mathbf Y_{s-1}\mathbf w_2^{n_{s,1}}\cdots (\mathbf w_2\mathbf w_0)^{n_{s,4j-2}})+\frac1{\beta^{p_{s,4j-1}}}\frac1{\beta^2+1}(1-1/\beta^{2n_{s,4j-1}})\right).\nonumber
\end{flalign}
Consider the summand (with respect to the summation over $j$) on the left-hand side of the previous inequality.  We will show that this is less than or equal to $n_{s,4j-3}v(\mathbf d)$, with equality if and only if $n_{s,4j-3}=0$.  If $n_{s,4j-3}=0$, both the summand and $n_{s,4j-3}v(\mathbf d)$ are zero; assume $n_{s,4j-3}>0$.  We must show
\begin{equation}\label{claim_LHS}
\sum_{\ell=1}^{n_{s,4j-3}}(v(\mathbf d)-v(1\mathbf X_1\mathbf Y_1\cdots \mathbf X_{s-1}\mathbf Y_{s-1}\mathbf w_2^{n_{s,1}}\cdots (\mathbf w_0\mathbf w_2)^{n_{s,4j-4}}\mathbf w_2^\ell))>\frac1{\beta^{p_{s,4j-3}}}\frac1{\beta^2(\beta^2+1)}(1-1/\beta^{2n_{s,4j-3}}).
\end{equation}
The left-hand side of the previous line equals
\[\sum_{\ell=1}^{n_{s,4j-3}}\left(\frac1{\beta^{p_{s,4j-3}+2\ell}}v(\mathbf w_2^{n_{s,4j-3}-\ell}(\mathbf w_2\mathbf w_0)^{n_{s,4j-2}}\cdots \mathbf w_{i_n}(1-i_n/2))\right).\]
Note that $(\mathbf w_2\mathbf w_0)^{n_{s,4j-2}}\cdots \mathbf w_{i_n}(1-i_n/2)\succ(\mathbf w_0\mathbf w_2)^\infty$: if not, then the former word begins with $(\mathbf w_0\mathbf w_2)^{n'}\mathbf w_0\mathbf w_i$ for some $n'\ge 0$ and $i\in\{0,1\}$.  But then
\[\mathbf w_2^{n_{s,4j-3}}(\mathbf w_2\mathbf w_0)^{n_{s,4j-2}}\cdots \mathbf w_{i_n}(1-i_n/2)=\mathbf w_2^{n_{s,4j-3}}(\mathbf w_0\mathbf w_2)^{n'}\mathbf w_0\mathbf w_i=\mathbf w_2^{n_{s,4j-3}-1}(\mathbf w_0\mathbf w_2)^{n'+1}\mathbf w_i,\]
contradicting the minimality of the sum of powers $\sum_{\ell=1}^{4r_s}n_{s,\ell}$.  Thus the left-hand side of (\ref{claim_LHS}) is strictly greater than
\begin{align*}
&\sum_{\ell=1}^{n_{s,4j-3}}\left(\frac1{\beta^{p_{s,4j-3}+2\ell}}v(\mathbf w_2^{n_{s,4j-3}-\ell})+\frac1{\beta^{p_{s,4j-2}}}v((\mathbf w_0\mathbf w_2)^\infty)\right)\\
=&\frac1{\beta^{p_{s,4j-3}}}\sum_{\ell=1}^{n_{s,4j-3}}\left(\frac1{\beta^{2\ell+1}}(1-1/\beta^{2n_{s,4j-3}-2\ell})+\frac1{\beta^2+1}\frac1{\beta^{2n_{s,4j-3}+1}}\right)\\
=&\frac1{\beta^{p_{s,4j-3}}}\left(\frac1{\beta^2}(1-1/\beta^{2n_{s,4j-3}})-\frac{n_{s,4j-3}}{\beta^{2n_{s,4j-3}+1}}+\frac1{\beta^2+1}\frac{n_{s,4j-3}}{\beta^{2n_{s,4j-3}+1}}\right)\\
=&\frac1{\beta^{p_{s,4j-3}}}\left(\frac1{\beta^2}(1-1/\beta^{2n_{s,4j-3}})-\frac{\beta^2}{\beta^2+1}\frac{n_{s,4j-3}}{\beta^{2n_{s,4j-3}+1}}\right).
\end{align*}
It suffices to show that the right-hand side of the previous line is greater than or equal to the right-hand side of (\ref{claim_LHS}).  Multiplying both quantities by $\beta^{p_{s,4j-3}+2}(\beta^2+1)$, this is equivalent to showing
\[(\beta^2+1)(1-1/\beta^{2n_{s,4j-3}})-\beta^3\frac{n_{s,4j-3}}{\beta^{2n_{s,4j-3}}}\ge1-\frac1{\beta^{2n_{s,4j-3}}},\]
which simplifies to
\[1-\frac1{\beta^{2n_{s,4j-3}}}\ge\beta\frac{n_{s,4j-3}}{\beta^{2n_{s,4j-3}}}.\]
The left- and right-hand sides of the previous line increase and decrease, respectively, as functions of integers $n_{s,4j-3}>0$.
Since the inequality holds for $n_{s,4j-3}=1$, we conclude that (\ref{claim_LHS}) holds.  Thus the summand on the left-hand side of (\ref{claim3}) is less than or equal to $n_{s,4j-3}v(\mathbf d)$, with equality if and only if $n_{s,4j-3}=0$.

Next, consider the summand on the right-hand side of (\ref{claim3}).  We shall show that this is greater than or equal to $n_{s,4j-1}v(\mathbf d)$ with equality if and only if $n_{s,4j-1}=0$.  Again if $n_{s,4j-1}=0$, both the summand and $n_{s,4j-1}v(\mathbf d)$ equal zero, so assume $n_{s,4j-1}>0$.  The desired inequality is equivalent to
\begin{equation}\label{claim_RHS}
n_{s,4j-1}(v(\mathbf d)-v(1\mathbf X_1\mathbf Y_1\cdots \mathbf X_{s-1}\mathbf Y_{s-1}\mathbf w_2^{n_{s,1}}\cdots (\mathbf w_2\mathbf w_0)^{n_{s,4j-2}}))<\frac1{\beta^{p_{s,4j-1}}}\frac1{\beta^2+1}(1-1/\beta^{2n_{s,4j-1}}).
\end{equation}
The left-hand side of the previous line equals 
\[\frac{n_{s,4j-1}}{\beta^{p_{s,4j-1}}}v(\mathbf w_0^{n_{s,4j-1}}(\mathbf w_0\mathbf w_2)^{n_{s,4j}}\cdots \mathbf w_{i_n}(1-i_n/2))=\frac{n_{s,4j-1}}{\beta^{p_{s,4j}}}v((\mathbf w_0\mathbf w_2)^{n_{s,4j}}\cdots \mathbf w_{i_n}(1-i_n/2)).\]
For similar reasons as above, one finds that
$(\mathbf w_0\mathbf w_2)^{n_{s,4j}}\cdots \mathbf w_{i_n}(1-i_n/2)\prec (\mathbf w_2\mathbf w_0)^\infty$.  It follows that the left-hand side of (\ref{claim_RHS}) is strictly less than
\[\frac{n_{s,4j-1}}{\beta^{p_{s,4j}}}v((\mathbf w_2\mathbf w_0)^\infty)=\frac{n_{s,4j-1}}{\beta^{p_{s,4j}}}\frac{\beta}{\beta^2+1}.\]
Multiplying both sides of (\ref{claim_RHS}) by $\beta^{p_{s,4j}}(\beta^2+1)$ and recalling that $p_{s,4j}-p_{s,4j-1}=2n_{s,4j-1}$, it thus suffices to show that
\[\beta n_{s,4j-1}\le \beta^{2n_{s,4j-1}}-1,\]
which clearly holds for each $n_{s,4j-1}\ge 1$.  This proves that the summand on the right-hand side of (\ref{claim3}) is greater than or equal to $n_{s,4j-1}v(\mathbf d)$ with equality if and only if $n_{s,4j-1}=0$.

Note that (\ref{mathfrak(n)}) may be rewritten as
\begin{align*}
\mathfrak{n}(\mathbf d)&=(1+\#\{1\le k\le n-1\ |\ i_{k}\in\{1,2\}\}+1)-(\#\{1\le k\le n-1\ |\ 2-i_{k}\in\{1,2\}\}+1)\\
&=1+\#\{1\le k\le n-1\ |\ i_{k}=2\}-\#\{1\le k\le n-1\ |\ i_{k}=0\}.
\end{align*}
Since $\mathfrak{n}(\mathbf d)=1$ by assumption, we have 
\[\#\{1\le k\le n-1\ |\ i_{k}=2\}=\#\{1\le k\le n-1\ |\ i_{k}=0\}.\]
Recalling that $\mathbf d=1\mathbf X_1\mathbf Y_1\cdots \mathbf X_t\mathbf Y_t \mathbf w_{i_n}(1-i_n/2)$, (\ref{X_s}) and the fact that each $\mathbf Y_s$ consists solely of $w_1$'s, we find
\[\#\{1\le k\le n-1\ |\ i_{k}=2\}=\sum_{s=1}^t\sum_{j=1}^{r_s}(n_{s,4j-3}+n_{s,4j-2}+n_{s,4j})\]
and
\[\#\{1\le k\le n-1\ |\ i_{k}=0\}=\sum_{s=1}^t\sum_{j=1}^{r_s}(n_{s,4j-2}+n_{s,4j-1}+n_{s,4j}),\]
so
\begin{equation}\label{sum_ns_equal}
\sum_{s=1}^t\sum_{j=1}^{r_s}n_{s,4j-3}=\sum_{s=1}^t\sum_{j=1}^{r_s}n_{s,4j-1}.
\end{equation}
Using this and our prior observations regarding the left- and right-hand sides of (\ref{claim3}), we have
\begin{align*}
&\sum_{s=1}^t\sum_{j=1}^{r_s}\bigg(\sum_{\ell=1}^{n_{s,4j-3}}v(1\mathbf X_1\mathbf Y_1\cdots \mathbf X_{s-1}\mathbf Y_{s-1}\mathbf w_2^{n_{s,1}}\cdots (\mathbf w_0\mathbf w_2)^{n_{s,4j-4}}\mathbf w_2^\ell)\\
&+\frac1{\beta^{p_{s,4j-3}}}\frac1{\beta^2(\beta^2+1)}(1-1/\beta^{2n_{s,4j-3}})\bigg)\\
\le&v(\mathbf d)\sum_{s=1}^t\sum_{j=1}^{r_s}n_{s,4j-3}\\
=&v(\mathbf d)\sum_{s=1}^t\sum_{j=1}^{r_s}n_{s,4j-1}\\
\le&\sum_{j=1}^{r_s}\left(n_{s,4j-1}v(1\mathbf X_1\mathbf Y_1\cdots \mathbf X_{s-1}\mathbf Y_{s-1}\mathbf w_2^{n_{s,1}}\cdots (\mathbf w_2\mathbf w_0)^{n_{s,4j-2}})+\frac1{\beta^{p_{s,4j-1}}}\frac1{\beta^2+1}(1-1/\beta^{2n_{s,4j-1}})\right)
\end{align*}
with equality throughout if and only if each $n_{s,4j-1}=n_{s,4j-3}=0$.  Thus the inequality in (\ref{claim3})---and hence in (\ref{claim})---holds.  It remains to show that $\mathbf d\prec 1(\mathbf w_2\mathbf w_0)^\infty$ if and only if each $n_{s,4j-1}=n_{s,4j-3}=0$.

Suppose that $\mathbf d\succeq 1(\mathbf w_2\mathbf w_0)^\infty$.  Then $\mathbf d$ begins with $1(\mathbf w_2\mathbf w_0)^{n'}\mathbf w_2\mathbf w_i$ for some $n'\ge 0$ and $i\in\{1,2\}$.  This implies that either $n_{1,1}$ or $n_{1,5}$ is positive.  For the converse, suppose that some $n_{s,4j-3}$ or $n_{s,4j-1}$ is positive.  By (\ref{sum_ns_equal}), we can choose some $n_{s,4j-3}>0$ with $(s,j)$ (lexicographically) minimal.  Note that
\[\sigma^{p_{s,4j-3}-1}(\mathbf d)=d_i\mathbf w_2^{n_{s,4j-3}}(\mathbf w_2\mathbf w_0)^{n_{s,4j-2}}\cdots \mathbf w_{i_n}(1-i_n/2)\]
with $d_i\in\{0,1\}$.  Suppose $d_i=0$ (the case that $d_i=1$ is similar).
Then $j>1$, and
\begin{align*}
\sigma^{p_{s,4j-7}}(\mathbf d)=\mathbf w_2^{n_{s,4j-7}}(\mathbf w_2\mathbf w_0)^{n_{s,4j-6}}\mathbf w_0^{n_{s,4j-5}}(\mathbf w_0\mathbf w_2)^{n_{s,4j-4}}\mathbf w_2^{n_{s,4j-3}}(\mathbf w_2\mathbf w_0)^{n_{s,4j-2}}\cdots \mathbf w_{i_n}(1-i_n/2).
\end{align*}
Since $d_i=0$, we must have $n_{s,4j-4}=0$.  Moreover, $n_{s,4j-5}>0$ contradicts the minimality of $\sum_{\ell=1}^{4r_s}n_{s,\ell}$, so $n_{s,4j-5}=0$.  Since no three consecutive $n_{s,\ell}$'s can be zero (except possibly the first and final three $n_{s,\ell}$), it follows that $n_{s,4j-6}>0$.  Thus
\begin{equation}\label{sigma_4j-6}
\sigma^{p_{s,4j-6}-1}(\mathbf d)=d_{i'}(\mathbf w_2\mathbf w_0)^{n_{s,4j-6}}\mathbf w_2^{n_{s,4j-3}}(\mathbf w_2\mathbf w_0)^{n_{s,4j-2}}\cdots \mathbf w_{i_n}(1-i_n/2)
\end{equation}
for some $d_{i'}\in\{0,1\}$.  Suppose $d_{i'}=0$.  Then $j>2$, and 
\begin{align*}
\sigma^{p_{s,4j-11}}(\mathbf d)=&\mathbf w_2^{n_{s,4j-11}}(\mathbf w_2\mathbf w_0)^{n_{s,4j-10}}\mathbf w_0^{n_{s,4j-9}}(\mathbf w_0\mathbf w_2)^{n_{s,4j-8}}\mathbf w_2^{n_{s,4j-7}}(\mathbf w_2\mathbf w_0)^{n_{s,4j-6}}\mathbf w_2^{n_{s,4j-3}}(\mathbf w_2\mathbf w_0)^{n_{s,4j-2}}\\
&\cdots \mathbf w_{i_n}(1-i_n/2).
\end{align*}
Since $d_{i'}=0$, we have $n_{s,4j-8}=n_{4j-7}=0$.  If $n_{s,4j-9}>0$, the fact that $n_{s,4j-3}>0$ contradicts the minimality of $\sum_{\ell=1}^{4r_s}n_{s,\ell}$.  But $n_{s,4j-9}=0$ is also a contradiction since this implies three consecutive $n_{s,\ell}$'s are zero.  Thus $d_{i'}=1$.  Now suppose $\sigma^{p_{s,4j-6}-1}(\mathbf d)\prec 1(\mathbf w_2\mathbf w_0)^\infty$.  From (\ref{sigma_4j-6}), we find that $n_{s,4j-3}=1$ and
\[\sigma^{p_{s,4j-6}-1}(\mathbf d)=1(\mathbf w_2\mathbf w_0)^{n_{s,4j-6}}\mathbf w_2(\mathbf w_0\mathbf w_2)^{n'}\mathbf w_0\mathbf w_{i''}\cdots \mathbf w_{i_n}(1-i_n/2)\]
for some $n'\ge 0$ and $i''\in\{0,1\}$.  In any case, this contradicts the minimality of $\sum_{\ell=1}^{4r_s}n_{s,\ell}$.  Thus (using the fact that $\mathbf d\in\mathcal{M}$),
\[\mathbf d\succeq \sigma^{p_{s,4j-6}-1}(\mathbf d)\succeq 1(\mathbf w_2\mathbf w_0)^\infty,\]
and we conclude that $\mathbf d\prec 1(\mathbf w_2\mathbf w_0)^\infty$ if and only if each $n_{s,4j-1}=n_{s,4j-3}=0$.
\end{proof}

Note that for each $n\ge 1$, the word $\mathbf d^n:=1(\mathbf w_2\mathbf w_0)^n001\prec 1(\mathbf w_2\mathbf w_0)^\infty$ satisfies Property $M$.  Moreover, $v(\mathbf d^n)$ approaches $v(1(\mathbf w_2\mathbf w_0)^\infty)=2\beta/(\beta+2)$ from below, and thus $1/v(\mathbf d^n)$ approaches $(\beta+2)/(2\beta)=1/2+1/\beta$ from above.  If $\mathbf d\in\mathcal{M}$ satisfies $\mathbf d\prec 1(\mathbf w_2\mathbf w_0)^\infty$, then there is some $n\ge 1$ for which $\mathbf d\prec\mathbf d^n\prec 1(\mathbf w_2\mathbf w_0)^\infty$ and $1/2+1/\beta<1/v(\mathbf d^n)<1/v(\mathbf d)$.  Since $I_{\mathbf d^n}\cap I_{\mathbf d}=\varnothing$ and $I_{\mathbf d^n}$ and $I_{\mathbf d}$ contain $1/v(\mathbf d^n)$ and $1/v(\mathbf d)$, respectively, it follows that $I_{\mathbf d}\subset (1/2+1/\beta,\beta]$.  Similarly reasoning shows that if $\mathbf d\succ 1(\mathbf w_2\mathbf w_0)^\infty$, then $I_{\mathbf d}\subset(1,1/2+1/\beta)$, and in fact $1/2+1/\beta$ is a non-matching parameter.  

With these observations and the previous lemmas, we are now ready to prove the main result of this section:

\begin{thm}\label{max_of_freq_funcs}
The frequency functions $\mathfrak{f}_S,\mathfrak{f}_T:[1,\beta]\to[0,1]$ attain their maximums $\mathfrak{f}_S(\alpha)=4/5$ and $\mathfrak{f}_T(\alpha)=3/4$ on the maximal interval $[1/2+1/\beta,1+1/\beta^2]$.
\end{thm}
\begin{proof}
By (\ref{freq_T_and_freq_S}), it suffices to show the statement for $\mathfrak{f}_S$.
Recall from Example \ref{freq_eg} that $\mathfrak{f}_S$ equals $4/5$ on $I_{1010}\cup I_{1001}=(1+1/\beta^4,1+1/\beta^2)\backslash\{1+1/\beta^3\}$.  Moreover, $\mathfrak{f}_S$ is decreasing on $I_{10}=(1+1/\beta^2,\beta]$ since $\mathfrak{n}(10)=0$.  By continuity of $\mathfrak{f}_S$, the statement is proven for $\alpha\in[1+1/\beta^4,\beta]$.  

We now show that $\mathfrak{f}_S(\alpha)\le 4/5$ for $\alpha\in[1,1+1/\beta^4)$, with equality if $\alpha\ge 1/2+1/\beta$.  Since $\mathfrak{f}_S$ is continuous and is monotone on each matching interval $I_{\mathbf d}$, and since the set of matching parameters $\cup_{\mathbf d\in\mathcal{M}}I_\mathbf d$ is dense, it suffices to show the desired statements for the endpoints $\alpha_{\mathbf d}^\pm$ of matching intervals in $[1,1+1/\beta^4)$.  
Notice that each endpoint $\alpha_{\mathbf d}^+,\alpha_{\mathbf d}^-\in[1,1+1/\beta^4)$ is the limit (from above) of some sequence of endpoints of cascade intervals.  In particular, if $\mathbf d\in\psi(\mathcal{M}_U)$, then $I_{\mathbf d}$ is itself a cascade interval and we take constant sequences.  Suppose $\mathbf d\in\mathcal{M}_U\backslash\psi(\mathcal{M}_U)$.  Since each lower endpoint $\alpha_{\mathbf d}^-$ equals the upper endpoint $\alpha_{\psi(\mathbf d)}^+$ of $I_{\psi(\mathbf d)}$ by Proposition \ref{cacscades_preserve_prop_M}, we can again take the constant sequence.  Now consider $\alpha_{\mathbf d}^+$.  Let $\varepsilon>0$, and choose some matching parameter $\alpha'\in I_{\mathbf d'}$ satisfying $\alpha_{\mathbf d}^+<\alpha'<\alpha_{\mathbf d}^++\varepsilon$.  Since matching intervals are disjoint, Proposition \ref{cacscades_preserve_prop_M} implies that the cascade interval $I_{\psi(\mathbf d')}$ lies strictly between $\alpha_{\mathbf d}^+$ and $\alpha'$, and thus its endpoints are within a distance of $\varepsilon$ of $\alpha_{\mathbf d}^+$.  It follows $\alpha_{\mathbf d}^+$ is the limit (from above) of a sequence of endpoints of cascade intervals.  
Again by continuity of $\mathfrak{f}_S$, it now suffices to show the desired statements for endpoints of cascade intervals.  These follow directly from Lemmas \ref{freq_func_on_cascades} and \ref{freq_func_bdd_lem} and the observation above that if $I_{\mathbf d}\subset (1/2+1/\beta,\beta]$, then $\mathbf d\prec 1(\mathbf w_2\mathbf w_0)^\infty$. 

Maximality of the interval $[1/2+1/\beta,1+1/\beta^2]$ follows from the fact that $\mathfrak{f}_S$ is strictly decreasing on $(1+1/\beta^2,\beta]$, density of matching parameters in $[1,\beta]$ and Lemmas \ref{freq_func_on_cascades} and \ref{freq_func_bdd_lem}.
\end{proof}

Theorem \ref{main_freq_thm} is now a collection of previous results:

\begin{proof}[Proof of Theorem \ref{main_freq_thm}]
This is a direct consequence of Proposition \ref{freq_funcs_continuous}, Theorem \ref{max_of_freq_funcs} and Equations (\ref{freq_T_and_freq_S}) and (\ref{freq_func}).
\end{proof}

\section{Appendix: proofs of technical lemmas}\label{Appendix: proofs of technical lemmas}

We include here two technical results, which together with Lemma \ref{lexicographical_ordering} prove Lemma \ref{reciprocal_endpt_lem}.  Recall that $\Delta(\mathbf u)$ denotes the cylinder set of points $x\in[0,1]$ for which the $\beta$-expansion of $x$ begins with $\mathbf u$.

\begin{lem}\label{reciprocal_expansions}
Let $\mathbf d=d_1\cdots d_m\in\mathcal{M}_U$ and $\mathbf e:=\varphi(\mathbf d)=e_1\cdots e_m$.  The $\beta$-expansions of $1/\alpha_{\mathbf d}^-,\ 1/\alpha_{\mathbf d}^+$  and $1-1/\alpha_{\mathbf d}^+$ are given by 
\begin{align*}
(b_j(1/\alpha_{\mathbf d}^-))_{j\ge 1}&=\begin{cases}
(\mathbf d\overline{e_2\cdots e_{m-2}}0)^\infty, & d_m=1\\
(\mathbf d\overline{e_1\cdots e_{m-1}}0)^\infty, & d_m=0
\end{cases},\\
(b_j(1/\alpha_{\mathbf d}^+))_{j\ge 1}&=\begin{cases}
(d_1\cdots d_{m-1}0)^\infty, & d_m=1\\
(d_1\cdots d_{m-2}0)^\infty, & d_m=0\\
\end{cases}
\ \ \ \text{and}\\
(b_j(1-1/\alpha_{\mathbf d}^+))_{j\ge 1}&=\begin{cases}
\overline{\mathbf e}^\infty, & d_m=1\\
0(\overline{e_2\cdots e_m})^\infty, & d_m=0
\end{cases}.
\end{align*}
\end{lem}
\begin{proof}
We consider only the $\beta$-expansion of $1/\alpha_{\mathbf d}^-$ for $d_m=1$; the proofs of the other expansions are similar. 
It suffices to show that $1/\alpha_{\mathbf d}^-\in\Delta(\mathbf d\overline{e_2\cdots e_{m-2}}0)$ and $B^{2m-2}(1/\alpha_{\mathbf d}^-)=1/\alpha_{\mathbf d}^-$.  First, note that
\begin{align*}
v(\mathbf d\overline{e_2\cdots e_{m-2}}0)&=v(\mathbf d)-(1/\beta^m)v(e_2\cdots e_{m-2})\\
&=v(\mathbf d)-(1/\beta^{m-1})v(e_1\cdots e_{m-2})\\
&=v(\mathbf d)-(1/\beta^{m-1})(v(\mathbf e)+1/\beta^{m-1})\\
&=v(\mathbf d)-(1/\beta^{m-1})(v(\mathbf d)-1+1/\beta^{m-1})\\
&=(1-1/\beta^{m-1})(v(\mathbf d)+1/\beta^{m-1}).
\end{align*}
Using this and Equation (\ref{cylinder_interval}), $1/\alpha_{\mathbf d}^-\in\Delta(\mathbf d\overline{e_2\cdots e_{m-2}}0)$ if and only if
\[(1-1/\beta^{m-1})(v(\mathbf d)+1/\beta^{m-1})\le 1/\alpha_{\mathbf d}^-<(1-1/\beta^{m-1})v(\mathbf d)+1/\beta^{m-1}.\]
Since $d_m=1$, the first inequality holds if and only if
\[(1-1/\beta^{m-1})(v(\mathbf d)+1/\beta^{m-1})\le \frac{\beta^mv(\mathbf d)+\beta}{\beta^m+\beta},\]
or
\[(\beta^m+\beta)(1-1/\beta^{m-1})(v(\mathbf d)+1/\beta^{m-1})\le \beta^mv(\mathbf d)+\beta.\]
Factoring $\beta^m$ from the first and multiplying it through the third term, the left-hand side is equal to
\[(1+1/\beta^{m-1})(1-1/\beta^{m-1})(\beta^mv(\mathbf d)+\beta)=(1-1/\beta^{2m-2})(\beta^mv(\mathbf d)+\beta),\]
which is less than $\beta^mv(\mathbf d)+\beta$.  The second inequality is true if and only if 
\[\frac{\beta^mv(\mathbf d)+\beta}{\beta^m+\beta}<(1-1/\beta^{m-1})v(\mathbf d)+1/\beta^{m-1}.\]
Multiplying both sides by $\beta^m+\beta$, this is equivalent to
\[\beta^mv(\mathbf d)+\beta<(\beta^m-1/\beta^{m-2})v(\mathbf d)+\beta+1/\beta^{m-2},\]
or $(1/\beta^{m-2})v(\mathbf d)<1/\beta^{m-2}$.  This holds since $v(\mathbf d)<v((10)^\infty)=1$.  Thus $1/\alpha_{\mathbf d}^-\in\Delta(\mathbf d\overline{e_2\cdots e_{m-2}}0)$.  With this and Equation (\ref{B^k_eqn}), 
\begin{align*}
B^{2m-2}(1/\alpha_{\mathbf d}^-)&=\beta^{2m-2}(1/\alpha_{\mathbf d}^--v(\mathbf d\overline{e_2\cdots e_{m-2}}0))\\
&=\beta^{2m-2}\left(\frac{\beta^mv(\mathbf d)+\beta}{\beta^m+\beta}-(1-1/\beta^{m-1})(v(\mathbf d)+1/\beta^{m-1})\right)\\
&=\beta^{2m-2}\left(\frac{\beta^mv(\mathbf d)+\beta-(\beta^m-1/\beta^{m-2})(v(\mathbf d)+1/\beta^{m-1})}{\beta^m+\beta}\right)\\
&=\frac{\beta^mv(\mathbf d)+\beta}{\beta^m+\beta}\\
&=1/\alpha_{\mathbf d}^-.
\end{align*}
\end{proof}

\begin{lem}\label{lex_ineq_lemma}
Let $\mathbf d=d_1\cdots d_m\in\mathcal{M}_U$ and $\mathbf e:=\varphi(\mathbf d)=e_1\cdots e_m$.  If $d_m=1$, then for each $j>0$,
\[\sigma^j((\mathbf d\overline{e_2\cdots e_{m-2}}0)^\infty)\preceq (\mathbf d\overline{e_2\cdots e_{m-2}}0)^\infty\]
and 
\[\sigma^j(\overline{\mathbf e}^\infty)\preceq (d_1\cdots d_{m-1}0)^\infty.\]
If $d_m=0$, then for each $j>0$, 
\[\sigma^j((\mathbf d\overline{e_1\cdots e_{m-1}}0)^\infty)\preceq (\mathbf d\overline{e_1\cdots e_{m-1}}0)^\infty\]
and 
\[\sigma^j(0(\overline{e_2\cdots e_m})^\infty)\preceq(d_1\cdots d_{m-2} 0)^\infty.\]
\end{lem}
\begin{proof}
We prove the statements for $d_m=1$; the other proofs are similar.
Write 
\[\mathbf d=1\mathbf{w}_{i_1}\mathbf{w}_{i_2}\cdots\mathbf{w}_{i_n}(1-i_n/2)\]
and
\[\mathbf e=\overline{0\mathbf{w}_{2-i_1}\mathbf{w}_{2-i_2}\cdots\mathbf{w}_{2-i_n}(i_n/2)}\]
with each $i_k\in\{0,1,2\}$ and $i_n=0$.  Due to periodicity, it suffices to show the first inequality for $0\le j<m-2$.  Note that $d_m=1$ implies $\overline{e_{m-1}}=1$.  If $j\ge m$, then 
\[\sigma^{j}((\mathbf d\overline{e_2\cdots e_{m-2}}0)^\infty)=(\overline{e_{j-m+2}\cdots e_{m-2}}0\mathbf d\overline{e_2\cdots e_{j-m+1}})^\infty\prec \overline{e_{j-m+2}\cdots e_{m-1}e_m}\preceq \mathbf d\prec (\mathbf d\overline{e_2\cdots e_{m-2}}0)^\infty.\]
Now suppose $0\le j<m$.  It suffices to show that
\[d_{j+1}\cdots d_m\overline{e_2\cdots e_{m-2}}0d_1\cdots d_j\preceq \mathbf d\overline{e_2\cdots e_{m-2}}0.\]
This trivially holds if $j=0$, so assume $j>0$.  Since $\sigma^j(\mathbf d)\preceq\mathbf d$, we have $d_{j+1}\cdots d_m\preceq d_1\cdots d_{m-j}$.  If this inequality is strict, we are finished.  Suppose equality holds.  
Then we wish to show
\[\overline{e_2\cdots e_{m-2}}0d_1\cdots d_j\preceq d_{m-j+1}\cdots d_m\overline{e_2\cdots e_{m-2}}0.\]
Since $\overline{e_{m-1}}=1$, it suffices to show
\begin{equation}\label{lex_ineq}
\overline{e_2\cdots e_{m-1}}\preceq d_{m-j+1}\cdots d_m\overline{e_2\cdots e_{m-j-1}}.
\end{equation}
If $j=m-1$, this is trivial, so suppose $j<m-1$.  By assumption, $d_{j+1}\cdots d_m=d_1\cdots d_{m-j}$, so $d_{j+1}=d_1=1=d_m=d_{m-j}$.  Now $d_2=0$ implies $j\neq m-2$, and similarly $d_{m-2}=0$ implies $j\neq m-3$.  Hence $j<m-3$, and $d_{j+1}=d_{m-j}=1$ imply that $d_{j+2}$ and $d_{m-j+1}$  are the beginnings of some blocks $\mathbf w_{i_p}$ and $\mathbf w_{i_\ell}$, respectively.  (Similarly, $\overline{e_{j+2}}$ and $\overline{e_{m-j+1}}$ are the beginnings of $\mathbf w_{2-i_p}$ and $\mathbf w_{2-i_\ell}$, respectively.)  Then $d_1\cdots d_{m-j}=d_{j+1}\cdots d_m$ may be written as
\[1\mathbf w_{i_1}\cdots \mathbf w_{i_{\ell-1}}=1\mathbf w_{i_{p}}\cdots \mathbf w_{i_{n-1}}001.\]
In particular, $i_{\ell-1}=1$, and $\mathbf w_{2-i_{\ell-1}}=\mathbf w_1$ implies $\overline{e_{m-j}}=1$.

The desired inequality (\ref{lex_ineq}) may be written in terms of blocks:
\[\mathbf w_{2-i_1}\cdots \mathbf w_{2-i_{n}}\preceq \mathbf w_{i_\ell}\cdots \mathbf w_{i_{n-1}}\mathbf w_1\mathbf w_{2-i_1}\cdots \mathbf w_{2-i_{\ell-2}}.\]
Suppose for the sake of contradiction that this inequality does not hold, and let $1\le k\le n$ be minimal such that $\mathbf w_{2-i_k}$ differs from the $k^{\text{th}}$ block on the right-hand side.  Then
\[\mathbf w_{2-i_k}\succ \begin{cases}
\mathbf w_{i_{\ell+k-1}}, & k<n-\ell+1\\
\mathbf w_{1}, & k=n-\ell+1\\
\mathbf w_{2-i_{k-(n-\ell)-1}}, & k>n-\ell+1 \\
\end{cases},\]
and we consider these three cases separately:
\begin{enumerate}
\item[(i)] If $k<n-\ell+1$, then
\[(2-i_1,\dots,2-i_{k-1})=(i_\ell,\dots,i_{\ell+k-2})\]
and $2-i_k>i_{\ell-k-1}$ imply 
\[(2-i_\ell,\dots,2-i_{\ell+k-2})=(i_1,\dots,i_{k-1})\]
and $2-i_{\ell-k-1}>i_{k}$.  This gives
\[1\mathbf w_{2-i_\ell}\cdots \mathbf w_{2-i_{\ell+k-1}}\succ 1\mathbf w_{i_1}\cdots \mathbf w_{i_k}.\]
Recall that $\overline{e_{m-j+1}}$ is the beginning of the block $\mathbf w_{2-i_\ell}$, so the previous line together with $\overline{e_{m-j}}=1$ imply $\sigma^{m-j-1}(\overline{\mathbf e})\succ\mathbf d$, a contradiction.

\item[(ii)]  If $k=n-\ell+1$, then
\[(2-i_1,\dots,2-i_{n-\ell})=(i_\ell,\dots,i_{n-1})\]
and $2-i_{n-\ell+1}>1$ imply 
\[(2-i_\ell,\dots,2-i_{n-1})=(i_1,\dots,i_{n-\ell})\]
and $i_{n-\ell+1}=0$.  Since $2-i_n=2$, this implies
\[1\mathbf w_{2-i_\ell}\cdots \mathbf w_{2-i_n}\succ 1\mathbf w_{i_1}\cdots \mathbf w_{i_{n-\ell+1}}.\]
As in case (i), this gives the contradiction that $\sigma^{m-j-1}(\overline{\mathbf e})\succ\mathbf d$.

\item[(iii)]  If $k>n-\ell+1$, then
\[(2-i_1,\dots,2-i_{n-\ell})=(i_\ell,\dots,i_{n-1})\]
and $2-i_{n-\ell+1}=1$ implies
\[(2-i_\ell,\dots,2-i_{n-1})=(i_1\dots,i_{n-\ell})\]
and $i_{n-\ell+1}=1$.  Again since $2-i_n=2$,
\[1\mathbf w_{2-i_\ell}\cdots \mathbf w_{2-i_n}\succ 1\mathbf w_{i_1}\cdots \mathbf w_{i_{n-\ell+1}},\]
and the contradiction of cases (i) and (ii) arises.
\end{enumerate}
This proves for each $j>0$ that 
\[\sigma^j((\mathbf d\overline{e_2\cdots e_{m-2}}0)^\infty)\preceq (\mathbf d\overline{e_2\cdots e_{m-2}}0)^\infty.\]

It remains to show that
\[\sigma^j(\overline{\mathbf e}^\infty)\preceq (d_1\cdots d_{m-1}0)^\infty,\]
or, equivalently,
\[\overline{e_{j+1}\cdots e_me_1\cdots e_j}\preceq d_1\cdots d_{m-1}0\]
for $0\le j<m$.  
Suppose for the sake of contradiction that this inequality does not hold.  If there is some $k\le m-j$ for which 
\[\overline{e_{j+1}\cdots e_{j+k}}\succ d_1\cdots d_k,\]
then $\sigma^{j}(\overline{\mathbf e})\succ \mathbf d$, a contradiction.  Thus there is some minimal $1\le k\le j$ for which 
\[\overline{e_{j+1}\cdots e_me_1\cdots e_k}\succ d_1\cdots d_{m-j+k}.\]
The previous line may be written in block form
\[1\mathbf w_{2-i_\ell}\cdots \mathbf w_{2-i_{n-1}}\mathbf w_2\mathbf w_0\mathbf w_{2-i_1}\cdots \mathbf w_{2-i_{p}}\succ 1\mathbf w_{i_1}\cdots \mathbf w_{i_q}\]
for some $\ell,p,q\in\{1,\dots n\}$.  In particular,
\[(0,2-i_1,\dots,2-i_{p-1})=(i_{q-p},i_{q-p+1},\dots,i_{q-1})\]
and $2-i_p>i_q$ imply 
\[(2-i_{q-p},2-i_{q-p+1},\dots,2-i_{q-1})=(2,i_1,\dots,i_{p-1})\]
and $2-i_q>i_p$.  Since $\mathbf w_{2-i_{q-p}}=01$, there is some $s\ge 0$ such that 
\[\sigma^s(\overline{\mathbf e})=1\mathbf w_{2-i_{q-p+1}}\cdots \mathbf w_{2-i_{q-1}}\mathbf w_{2-i_q}\cdots\mathbf w_{2-i_{n}}0\succ 1\mathbf w_{i_1}\cdots \mathbf w_{i_{p-1}}\mathbf w_{i_p}\cdots \mathbf w_{i_n}1=\mathbf d,\]
contrary to the assumption that $\mathbf d\in\mathcal{M}$.
\end{proof}


\bibliography{ergodic_properties_of_a_parameterised_family_of_symmetric_golden_maps}
\bibliographystyle{acm}

\end{document}